\documentclass[final,notitlepage,12pt,reqno]{amsart}
\usepackage{graphicx}
\usepackage{calc}
\usepackage{url}
\newlength{\depthofsumsign}
\newcommand{\nsum}[1][1.0]{
   \mathop{%
        \raisebox
            {-#1\depthofsumsign+1\depthofsumsign}
            {\scalebox
                {#1}
                {$\displaystyle\sum$}%
            }
    }
}

\makeatletter
\let\I\@undefined
\makeatother

\usepackage{geometry}
\usepackage{indentfirst}
\usepackage[normalem]{ulem}
\usepackage{float}
\usepackage{amsthm}

\usepackage{enumitem}[2011/09/28]
\setenumerate{align=left, leftmargin=0pt,labelsep=.5em, labelindent=0\parindent,listparindent=\parindent,itemindent=*}
\usepackage{array}
\usepackage[all]{xy}

\usepackage{exscale,relsize}
\usepackage{multirow}

\usepackage{caption}[2012/02/19]

\usepackage{longtable,lscape}

\usepackage{colonequals,extarrows}

\usepackage{amssymb,bm,amsmath}
\usepackage{amsfonts}

\usepackage[OT2,T1,T2A]{fontenc}

\usepackage[french,german,russian,english]{babel}
\usepackage{appendix}

\DeclareMathOperator{\dolS}{\mathsf S\kern-.5em\raisebox{.75pt}{$\mathsf |$}\kern-.15em\raisebox{.75pt}{$\mathsf |$}}

\DeclareMathOperator{\I}{Im}
\DeclareMathOperator{\R}{Re}

\DeclareMathOperator{\Span}{span}

\DeclareMathOperator{\Li}{Li}

\DeclareMathOperator{\D}{d}

\def\XXint#1#2#3{{\setbox0=\hbox{$#1{#2#3}{\int}$}
     \vcenter{\hbox{$#2#3$}}\kern-.5\wd0}}

\def\qed{\hfill$ \blacksquare$}
\def\eor{\hfill$ \square$}

\theoremstyle{plain}
\newtheorem{theorem}{Theorem}[section]
\newtheorem{proposition}[theorem]{Proposition}
\newtheorem{lemma}[theorem]{Lemma}
\newtheorem{corollary}[theorem]{Corollary}

\newenvironment{remark}[1][Remark]{\begin{trivlist}
\item[\hskip \labelsep {\bfseries #1}]}{\end{trivlist}}

\theoremstyle{definition}

\numberwithin{equation}{section}

\usepackage{Baskervaldx}
\usepackage[baskervaldx]{newtxmath}
\usepackage[cal=cm,scr=rsfs,frak=euler]{mathalfa}

\usepackage[scr=rsfs]{mathalfa}

\DeclareMathAlphabet{\mathsf}{OT1}{\sfdefault}{m}{n}
\SetMathAlphabet{\mathsf}{bold}{OT1}{\sfdefault}{m}{n}
\DeclareSymbolFontAlphabet{\mathbb}{AMSb}

\begin{document}

\pagenumbering{roman}
\selectlanguage{english}
\title{Notes on certain binomial harmonic sums of Sun's type}
\author{Yajun Zhou}
\address{Program in Applied and Computational Mathematics (PACM), Princeton University, Princeton, NJ 08544} \email{yajunz@math.princeton.edu}\curraddr{\textrm{} \textsc{Academy of Advanced Interdisciplinary Studies (AAIS), Peking University, Beijing 100871, P. R. China}}\email{yajun.zhou.1982@pku.edu.cn}
\date{\today}\thanks{\indent * This research was supported in part  by the Applied Mathematics Program within the Department of Energy
(DOE) Office of Advanced Scientific Computing Research (ASCR) as part of the Collaboratory on
Mathematics for Mesoscopic Modeling of Materials (CM4)}

\maketitle

\begin{abstract}
     We prove and generalize some recent conjectures of Z.-W.\ Sun on infinite series whose summands involve products of harmonic numbers and several binomial coefficients. We evaluate various classes of   infinite sums in closed form by interpreting them as automorphic objects on the moduli spaces for  Legendre curves $Y^{ g+1}=(1-X)^{ g}X(1-t X)$ of positive genera $ g\in\{1,2,3,5\}$.  \end{abstract}


{\footnotesize\setcounter{tocdepth}{2}
\tableofcontents}
\pagenumbering{arabic}

\section{Introduction} Set binomial coefficients $ \binom mk\colonequals\frac{m!}{k!(m-k)!}$ for $m\in\mathbb Z_{\geq0},k\in\mathbb Z\cap[0,m]$ and harmonic numbers $ \mathsf H_m^{(r)}\colonequals\sum_{k=1}^m\frac1{k^{r}}$ of order $ r\in\mathbb Z_{>0}$ for $ m\in\mathbb Z_{\geq0}$. Customarily, one abbreviates $ \mathsf H_m^{(1)}$ as $ \mathsf H_m^{}$, and simply calls it   the $m$-th harmonic number. In this note, \textit{binomial harmonic sums} are convergent series in  the form of \begin{align}
\sum_{k=1}^\infty\frac{\binom{2k}k^{\mathscr N}}{k^s}\left( \frac{T}{2^{2\mathscr N}} \right)^{k}\left[ \prod_{j=1}^M\mathsf H_{k-\smash[t]{1}}^{(r_j)} \right]\left[\prod_{\vphantom{j}\smash[b]{j'}=1}^{M'}{{\mathsf H}}_{\smash[t]{2}k-1}^{( {r}'_{j'})}\right],\label{eq:binomH}
\end{align} where $\mathscr N,s\in\mathbb Z
$, $ M,M'\in\mathbb Z_{\geq0}$, and $ |T|<1$.   We will be equally interested in some  variations of  \eqref{eq:binomH}, as considered by Z.-W.\ Sun in his large collections of recent conjectures \cite{Sun2023,Sun2022}.  For example, we may  do one or more of the following things:\begin{itemize}
\item
allow $ |T|=1$ when $ s\in\mathbb Z_{>1}$;\item trade $\binom{2k}{k} ^{\mathscr N}k^{-s},k\in\mathbb Z_{>0}$ for $\binom{2k}{k}  ^{\mathscr N}(2k+1)^{-s},k\in\mathbb Z_{\geq0}$;\item
 accommodate  other binomial coefficients like  $ \binom{3k}k$, $ \binom{4k}{2k}$, and $ \binom{6k}{3k}$; \item invite harmonic numbers like $\mathsf H^{(r)}_{3k-1} $, $\mathsf H^{(r)}_{4k-1} $, and $\mathsf H^{(r)}_{6k-1} $  into the summands.\end{itemize}
In our recent work \cite[\S3]{Zhou2022mkMpl}, we have produced closed-form evaluations for \eqref{eq:binomH} under the additional constraint that  $\mathscr N\in\{-1,0,1\}$, by studying integrals over Goncharov's multiple polylogarithms \cite{Goncharov1997,Goncharov1998}, which are effectively motivic periods on the moduli space of genus-zero curves \cite{Brown2009a,Brown2009arXiv,Brown2009b}. Here, the condition  $\mathscr N\in\{-1,0,1\}$  is relevant to  $ \varepsilon$-expansions in quantum electrodynamics and quantum chromodynamics \cite{Ablinger2017,Ablinger2022,ABRS2014,ABS2011,ABS2013,BorweinBroadhurstKamnitzer2001,Broadhurst1996,Broadhurst1999,DavydychevKalmykov2001,DavydychevKalmykov2004,KalmykovVeretin2000,Kalmykov2007,Weinzierl2004bn}. Explicit evaluations of   \eqref{eq:binomH}  for   $\mathscr N\in\{-1,1\}$ are also available when the central binomial coefficients  $\binom{2k}k $ in the summands  are replaced by $ \binom{3k}k$ or $ \binom{4k}{2k}$ \cite{SunZhou2024sum3k4k,Zhou2023SunCMZV}.

For $\nu\in(-1,0)$ and $ t\in\mathbb C\smallsetminus[1,\infty)$, define the Legendre function of degree $\nu$ by\begin{align}
P_\nu(1-2t)\colonequals-\frac{\sin(\nu\pi)}{\pi}\int_0^{1}\left[ \frac{X(1-tX)}{1-X} \right]^\nu\frac{\D X}{1-X},
\end{align}where the path of integration runs along the unit interval. In the current work devoted (mainly) to  the computations of    \eqref{eq:binomH}  for $ \mathscr N\in\{2,3,4\}$, our analytic tools will be the key players in Ramanujan's elliptic function theory \cite[Chapter 33]{RN5}, namely {\allowdisplaybreaks\begin{align}
\sum_{k=0}^\infty\binom{2k}k^2\left(\frac{t}{2^{4}}\right)^{k}=P_{-1/2}(1-2t)={}&\frac{1}{\pi}\int_0^1\frac{\D X}{\sqrt{(1-X)X(1-tX)}},\label{eq:P2si}\\\sum_{k=0}^\infty\binom{2k}k\binom{3k}k\left(\frac{t}{3^3}\right)^{k}=P_{-1/3}(1-2t)={}&\frac{\sqrt{3}}{2 \pi }\int_0^1\frac{\D X}{\sqrt[3]{(1-X)^{2}X(1-tX)}},\label{eq:P3sumint}\\\sum_{k=1}^\infty\binom{2k}k\binom{4k}{2k}\left(\frac{t}{2^6}\right)^{k}=P_{-1/4}(1-2t)={}&\frac{1}{\sqrt{2} \pi }\int_0^1\frac{\D X}{\sqrt[4]{(1-X)^{3}X(1-tX)}},\label{eq:P4sumint}\\\sum_{k=1}^\infty\binom{3k}k\binom{6k}{3k}
 \left(\frac{t}{2^43^3}\right)^{k}=P_{-1/6}(1-2t)={}&\frac{1}{2 \pi }\int_0^1\frac{\D X}{\sqrt[6]{(1-X)^{5}X(1-tX)}},
\end{align}}for $ |t|<1$. These abelian   integrals on Legendre curves $Y^{ g+1}=(1-X)^{ g}X(1-t X)$ of positive genera $ g\in\{1,2,3,5\}$
 correspond to  Sun's binomial harmonic sums in the $\mathscr N=2 ,s=0$   scenario \cite[\S2]{Sun2022}. Hypergeometric deformations of the Legendre functions $ P_\nu$ will   generate harmonic numbers in the summands. By considering products of several Legendre functions, one can produce further analogs and extensions with different values of $ \mathscr N$ and $s$ in   Sun's binomial harmonic sums.

In \S\ref{sec:FZ_Sun}, we exploit the Taylor expansion of the hypergeometric  series\begin{align}
_2F_1\left( \left.\begin{array}{@{}c@{}}
a_{1},a_{2} \\
b_{1} \\
\end{array}\right|t \right)=1+\sum_{k=1}^\infty\frac{(a_{1})_k(a_{2})_k}{(b_{1})_k}\frac{t^k}{k!},\quad|t|\leq 1,t\neq1, \label{eq:2F1_defn}
\end{align}
with $ (A)_{k}=\prod_{m=0}^{k-1}(A+m)$ being the rising factorial, and investigate the following perturbations  (Lemma \ref{lm:FB})\begin{align}
{_2}F_1\left(\left. \begin{array}{@{}c@{}}
-\nu+\varepsilon,\nu+1+\varepsilon\ \\
1\ \\
\end{array} \right| t\right)\text{ and }{_2}F_1\left(\left. \begin{array}{@{}c@{}}
-\nu,\nu+1\ \\
1+\varepsilon\ \\
\end{array} \right| t\right)\label{eq:2F1_eps}
\end{align}of the Legendre functions \begin{align}
P_\nu(1-2t)={_2}F_1\left(  \left.\begin{array}{@{}c@{}}
-\nu,\nu+1 \\
1 \\
\end{array}\right| t\right)\label{eq:Pnu2F1}
\end{align} in the $ \varepsilon\to0$ regime. This investigation brings us, among other things, series identities like\footnote{In this note, we prescribe the complex logarithms of non-zero numbers  as $ \R\log \xi=\log|\xi|$ and $\I \log \xi=\arg \xi\in(-\pi,\pi]$, while setting fractional powers as $ \xi^\beta\colonequals e^{\beta\log\xi}$ for $\xi\in\mathbb C\smallsetminus\{0\} $. }\begin{align}
\sum_{k=1}^\infty\binom{2k}k^2\left(\frac{t}{2^4}\right)^{k}(2\mathsf H_{2k}-\mathsf H_k)={}&-\frac{P_{-1/2}(1-2t)}{2}\log(1-t),\label{eq:K_eg1}\\\sum_{k=1}^\infty \binom{2k}k^2\left(\frac{t}{2^4}\right)^{k}\mathsf H_k={}&\frac{\pi P_{-1/2}(2t-1)}{2}-\frac{P_{-1/2}(1-2t)}{2}\log\frac{2^{4}(1-t)}{t},\label{eq:K_eg2}
\end{align} where     $|t|\leq 1$ and $ t\notin [-1,0]\cup\{1\}$.  All these $ t$-parametrized formulae cover and extend Sun's recent experimental discoveries \cite[Conjectures 10, 12--16]{Sun2022} as well as  their analytic verifications by   Wei \cite{Wei2023c,Wei2023Sun}.
In the same section, we will also introduce modular parametrizations of series identities, to pave the way for the rest of this note.

In \S\ref{sec:ClausenCouplingSun}, we treat infinite sums whose addends contain   $ \binom{2k}k^3$ in  their numerators. Our discussions there  will also apply to other cubic forms of  binomial coefficients like  $ \binom{2k}k^2\binom{3k}k$, $\binom{2k}k^2\binom{4k}{2k} $, and $ \binom{2k}k\binom{3 k}{k}\binom{6k}{3k}$. The starting point in \S\ref{sec:ClausenCouplingSun} is Clausen's coupling formula (Lemma \ref{lm:CFZ}) that represents the products of two Legendre functions as generalized hypergeometric $ _3F_2$ series, where the Taylor expansion
\begin{align}{_pF_q}\left(\left.\begin{array}{@{}c@{}}
a_{1},\dots,a_p \\[4pt]
b_{1},\dots,b_q \\
\end{array}\right| t\right):=1+\sum_{k=1}^\infty\frac{\prod_{\ell=1}^p(a_{\ell})_k}{\prod_{m=1}^q(b_m)_k }\frac{t^k}{k!}\label{eq:defn_pFq}\end{align}
extends \eqref{eq:2F1_defn}. Amidst various applications of  Clausen's  coupling formula and its $\varepsilon $-expansions, we will   demonstrate\begin{align}\begin{split}&
\sum_{k=0}^\infty\binom{2k}k\binom{3 k}{k}\binom{6 k}{3k}\left(-\dfrac{1}{2^{18} \cdot3^3 \cdot5^3 \cdot23^3 \cdot29^3}\right)^k\left[2\cdot3^2\cdot 7\cdot11\cdot19\cdot127\cdot163+{}\right.\\{}&+\left.3\big(2\cdot3^2\cdot 7\cdot11\cdot19\cdot127\cdot163k+13\cdot 1045493\big)(2 \mathsf H_{6 k}-\mathsf H_{3 k}-\mathsf H_k)\right]\\={}&-\frac{2^7 \sqrt{3\cdot5^3 \cdot23^3 \cdot29^3}}{\pi }\log\dfrac{1}{2^{18} \cdot3^3 \cdot5^3 \cdot23^3 \cdot29^3}\end{split}
\end{align} as Sun's companion   \cite[Conjecture 29(iv)]{Sun2022} to the Chudnovsky--Chudnovsky series \cite[(1.5)]{ChudnovskyChudnovsky1988} for $ \frac1\pi$.

In \S\ref{sec:GreenSun} and \S\ref{sec:CG_coupling}, we study some binomial harmonic series that are related to Sun's studies but are not part of Sun's conjectures.

In \S\ref{sec:GreenSun}, we consider integral representations for derivatives of the Legendre functions $ P_\nu(1-2t)$ with respect to their degrees $\nu\in(-1,0)$, or equivalently, leading-order   $ \varepsilon$-expansions of \begin{align}
{_2}F_1\left(\left. \begin{array}{@{}c@{}}
-\nu-\varepsilon,\nu+1+\varepsilon\ \\
1\ \\
\end{array} \right| t\right).
\end{align}  These efforts lead us to infinitely many analogs of Sun's series that represent automorphic Green's functions (Theorem \ref{thm:GreenSun}) and Epstein zeta functions (Theorem \ref{thm:EpsteinSun}) on the  moduli spaces for  the Legendre curves  $Y^{ g+1}=(1-X)^{ g}X(1-t X)$ of genera $ g\in\{1,2,3,5\}$.
 One particular example is the following series evaluation (cf.\  \cite[Conjecture 9]{Sun2022}):\begin{align}
\sum_{k=1}^\infty\frac{\binom{2k}k^2}{2^{5k}}\left[\mathsf H_{2k}^{(2)}-\frac14 \mathsf H_k^{(2)}\right]=\frac{\sqrt{\pi}\big[\Gamma\big(\frac14\big)\big]^2}{4}\left( \frac{1}{8} -\frac{G}{\pi^2}\right),
\end{align} where   $ \Gamma(s)\colonequals\int_0^\infty e^{-t}t^{s-1}\D t$ defines  Euler's gamma function  for $ s>0$ and $G\colonequals \sum_{k=0}^\infty\frac{(-1)^k}{(2k+1)^2}$ is Catalan's constant \cite{Catalan1865,Catalan1883}. In the same section, we will also revisit the Guillera--Rogers theory \cite{GuilleraRogers2014} for the $ \mathscr N=-s=-3$  case of \eqref{eq:binomH}, using automorphic representations of Epstein zeta functions (Theorem \ref{thm:GR_sum}).

In  \S\ref{sec:CG_coupling}, we give a new perspective on our previous works \cite{Zhou2013Pnu,Zhou2013Int3Pnu}
computing certain integrals over the products of three Legendre functions. The closed forms of these integral formulae, when properly reinterpreted, grant us  access to series evaluations like\begin{align}\sum_{k=0}^\infty \frac{\binom{2k}k^4}{2^{8k}}\frac{4k+1}{(k+1)(1-2k)}={}&\dfrac{8}{\pi^2},\label{eq:CGeg1}\\
\sum_{k=0}^\infty\frac{\binom{2k}k^4}{2^{8k}}\frac{1}{4k+1}={}&\frac{\big[\Gamma\big(\frac14\big)\big]^8}{96\pi^{5}},\label{eq:CGeg2}\\\sum_{k=0}^\infty\frac{\binom{2k}k^4}{2^{8k}}\frac{1}{4k+1}\left[ \frac{2}{(4k+1)^{2}} -\frac{\pi^{2}}{3}-2\mathsf H_{2k}^{(2)}+\mathsf H_k^{(2)}\right]={}&\dfrac{\big[\Gamma\big(\frac14\big)\big]^8}{4\pi^{5}} \left( G-\dfrac{5\pi^{2}}{48} \right).\label{eq:CGeg3}
\end{align}

Before closing this introduction, we note that Sun's recent conjectures \cite[\S4]{Sun2022} have set   \eqref{eq:binomH} in the context of $ \mathscr N\in\{-7,-5,5,7\}$ and $ \min\{M,M'\}>0$. Some of these conjectural identities  have been confirmed by Hou--He--Wang \cite[Theorem 1.7]{HouHeWang2023}, Wei (see \cite[Theorems 1.1--1.4]{Wei2023b}, \cite[Theorems 1.4--1.10]{Wei2023Sun},  as well as \cite[Theorem 1.3]{Wei2023c})  and Wei--Xu \cite[Theorems 1.3--1.4]{WeiXu2023} via combinatorial transformations that have effectively reduced the problems to binomial harmonic sums with smaller values of $|\mathscr N|$. Similar  reductions of  large-$ |\mathscr N|$ scenarios in    \eqref{eq:binomH} for $ \min\{M,M'\}\geq0$ are also found in Wilf--Zeilberger (WZ) approaches employed by Amdeberhan--Zeilberger \cite{AZ1998}, Guillera \cite{Guillera2003,Guillera2010}, Chu--Zhang \cite{ChuZhang2014}, Cohen--Guillera \cite{CohenGuillera2021}, and Au \cite{Au2025a,Au2025b}, which are fundamentally different from the methods we develop in this note.
\section{Legendre functions and Sun's series\label{sec:FZ_Sun}}

In \S\ref{subsec:FZ_LS}, we prove the theorem below by differentiating \eqref{eq:2F1_eps} with respect to $ \varepsilon$.

\begin{theorem}[Legendre--Sun series]\label{thm:LegendreSun} \begin{enumerate}[leftmargin=*,  label=\emph{(\alph*)},ref=(\alph*),
widest=d, align=left] \item For     $|t|\leq 1$ and $ t\neq1$, we have the following  identities:{\allowdisplaybreaks\begin{align}\begin{split}
\mathsf S_{-1/2}^L(t)\colonequals{}&2\sum_{k=1}^\infty\binom{2k}k^2\left(\frac{t}{2^4}\right)^{k}(2\mathsf H_{2k}-\mathsf H_k)=-P_{-1/2}(1-2t)\log(1-t),\end{split}\label{eq:LS2}\\\begin{split}
\mathsf S_{-1/3}^L(t)\colonequals{}&\sum_{k=1}^\infty\binom{2k}k\binom{3k}k\left(\frac{t}{3^3}\right)^{k}(3 \mathsf H_{3k}- \mathsf H_k)=-P_{-1/3}(1-2t)\log(1-t),\end{split}\label{eq:LS3}\\\begin{split}
\mathsf S_{-1/4}^L(t)\colonequals{}&2\sum_{k=1}^\infty\binom{2k}k\binom{4k}{2k}\left(\frac{t}{2^6}\right)^{k}(2\mathsf H_{4k}-\mathsf H_{2k})=-P_{-1/4}(1-2t)\log(1-t),\end{split}\label{eq:LS4}\\\begin{split}
\mathsf S_{-1/6}^L(t)\colonequals{}&\sum_{k=1}^\infty\binom{3k}k\binom{6k}{3k}
 \left(\frac{t}{2^43^3}\right)^{k}(6\mathsf H_{6k}-3\mathsf H_{3k}-2\mathsf H_{2k}+\mathsf H_k)=-P_{-1/6}(1-2t)\log(1-t).\end{split}\label{eq:LS6}
\end{align}} \item For    $|t|\leq 1$ and $ t\notin [-1,0]\cup\{1\}$, we have the following
identities:{\allowdisplaybreaks\begin{align}\begin{split}
\widetilde{\mathsf S}_{-1/2}^L(t)\colonequals{}&-\sum_{k=1}^\infty \binom{2k}k^2\left(\frac{t}{2^4}\right)^{k}\mathsf H_k={}-\frac{\pi P_{-1/2}(2t-1)}{2}+\frac{P_{-1/2}(1-2t)}{2}\log\frac{2^{4}(1-t)}{t},\end{split}\label{eq:LS2-}\\\begin{split}
\widetilde{\mathsf S}_{-1/3}^L(t)\colonequals{}&-\sum_{k=1}^\infty\binom{2k}k\binom{3k}k\left(\frac{t}{3^3}\right)^{k}\mathsf H_k={}-\frac{\pi P_{-1/3}(2t-1)}{\sqrt{3}}+\frac{P_{-1/3}(1-2t)}{2}\log\frac{3^{3}(1-t)}{t},\end{split}\\\begin{split}\widetilde{\mathsf S}_{-1/4}^L(t)\colonequals{}&-\sum_{k=1}^\infty\binom{2k}k\binom{4k}{2k}\left(\frac{t}{2^6}\right)^{k}\mathsf H_k={}-\frac{\pi P_{-1/4}(2t-1)}{\sqrt{2}}+\frac{P_{-1/4}(1-2t)}{2}\log\frac{2^{6}(1-t)}{t},\end{split}\\\begin{split}
\widetilde{\mathsf S}_{-1/6}^L(t)\colonequals{}&-\sum_{k=1}^\infty\binom{3k}k\binom{6k}{3k}
 \left(\frac{t}{2^43^3}\right)^{k}\mathsf H_k={}-\pi P_{-1/6}(2t-1)+\frac{P_{-1/6}(1-2t)}{2}\log\frac{2^43^3(1-t)}{t}.\end{split}
\end{align}}\end{enumerate}For $ t\in[-1,0)$, one has $ \widetilde{\mathsf S}_{\nu}^L(t)=\lim_{\varepsilon\to0^+}\widetilde{\mathsf S}_{\nu}^L(t\pm i\varepsilon^+)$ for $ \nu\in\left\{-\frac12,-\frac13,-\frac14,-\frac16\right\}$.\qed\end{theorem}
\begin{remark}We note that $ P_\nu(0)$ simplifies to $ \frac{\sqrt{\pi }}{\Gamma \left(\frac{1-\nu}{2}\right) \Gamma \left(\frac{\nu +2}{2}\right)}$ \cite[\S3.4(20)]{HTF1}. Setting $t=\frac{1}{2}$ in \eqref{eq:LS2}--\eqref{eq:LS6}, while employing some standard transformations for Euler's gamma function, we recover  Sun's identities \cite[(33) and (38) for   $ \mathsf S^L_{-1/4}(0)$  and $ \mathsf S^L_{-1/6}(0)$, respectively]{Sun2022},    which previously followed from the combinatorial analyses by  Wei \cite{Wei2023Sun}.  Plugging  the left halves of \eqref{eq:P3sumint} and \eqref{eq:P4sumint} into  \eqref{eq:LS3} and \eqref{eq:LS4}, one can check that Sun's conjectures \cite[(29) for $ \mathsf S^L_{-1/3}(-1/8)$,  (34) for $ \mathsf S^L_{-1/4}(8/9)$, (35) for  $ \mathsf S^L_{-1/4}(1/9)$, (36) for   $ \mathsf S^L_{-1/4}(-1/3)$, and (37) for $ \mathsf S^L_{-1/4}(-1/63)$]{Sun2022}  hold true, as noted by Wei \cite[Theorem 2.1]{Wei2023c}. The right-hand sides  of  \cite[(29) and   (34)--(37)]{Sun2022} will be further reduced  after
Corollary \ref{cor:LS_CM} below.\eor\end{remark}
\begin{remark}Let  $\mathbf K\big(\sqrt{t}\big)=\frac{\pi}{2}P_{-1/2}(1-2t)$   be the complete elliptic integral of the first kind. Its relations to Legendre functions of degrees $ -\frac{1}{3}$, $-\frac{1}{4}$, and $-\frac{1}{6}$ are elucidated by Ramanujan's theories for elliptic functions to alternative bases \cite[Chapter 33]{RN5}. 

Furthermore, the Chowla--Selberg formula \cite[(7)]{ChowlaSelberg} allows one to evaluate  $ \mathbf K\big(\sqrt{t}\big)$ in closed form (in terms of special values of Euler's gamma function at rational arguments) whenever $ z=\frac{i}{2}\mathbf K\big(\sqrt{1-t}\big)\big/\mathbf K\big(\sqrt{t}\big)$ represents a CM point [such that both $ \R z$ and $(\I z)^2$ are rational numbers, while $\I z>0$]. This brings us an infinite family of series identities, such as (cf.\ \cite[Appendix A.3]{LatticeSum2013}){\allowdisplaybreaks\begin{align}
\mathsf S_{-1/2}^L\left( \big(\sqrt{2}-1\big)^{2} \right)={}&-\frac{\Gamma\big(\frac18\big)\Gamma\big(\frac38\big)}{4\pi\sqrt{\pi}}\sqrt{\frac{2+\sqrt{2}}{2}}\left[\log2+\log\big(\smash{\sqrt{2}}-1\big)\right],\label{eq:LS2_k2}\\\mathsf S_{-1/2}^L\left( \frac{2-\sqrt{3}}{4} \right)={}&-\frac{\big[\Gamma\big(\frac13\big)\big]^3}{2\pi^{2}}\frac{\sqrt[4]3}{\sqrt[3]2}\left[\log\big(2+\smash{\sqrt{3}}\big)-2\log2\right],\label{eq:LS2_k3}\\\mathsf S_{-1/2}^L\left( \big(\sqrt{2}-1\big)^{4} \right)={}&-\frac{\big[\Gamma\big(\frac14\big)\big]^2}{8\pi \sqrt{\pi}{}}\big(2+\sqrt{2}\big)\left[2\log\big(\smash{\sqrt{2}}-1\big)+\frac{5}{2}\log2\right],\label{eq:LS2_k4}\\\mathsf S_{-1/2}^L\left( \frac{8-3 \sqrt{7}}{16}  \right)={}&-\frac{\Gamma\big(\frac17\big)\Gamma\big(\frac27\big)\Gamma\big(\frac47\big)}{2\pi^{2}}\frac{\log\big( 8+3 \sqrt{7} \big)-4\log 2}{\sqrt[4]7},\label{eq:LS2_k7}
\end{align}}thereby generalizing Sun's conjectures. \eor\end{remark}
\begin{remark}Sun's empirical evaluations  of infinite series  \cite[Conjectures 10, 12–16]{Sun2022} are motivated and accompanied by their supercongruence counterparts modulo $p^2$. It is not the purpose of  this note to touch upon the $p$-adic structure  of the complete elliptic integrals  \cite{McSpiritOno2023} or hypergeometric $ _2F_1$ series in general  \cite{FLST2022hypFF,GreeneStanton1986,TuYang2018}, yet the localization of Theorem \ref{thm:LegendreSun} over finite fields is perhaps a worthwhile direction for future research.    \eor\end{remark}

In \S\ref{subsec:mod_para_LegendreSun}, we give modular parametrizations of the infinite series in Theorem \ref{thm:LegendreSun}. The modular functions and modular forms appearing therein will also be used later in \S\ref{sec:ClausenCouplingSun} and \S\ref{sec:GreenSun}.

In \S\ref{subsec:int_LS}, we evaluate more binomial harmonic sums, by differentiating and integrating  the formulae in Theorem \ref{thm:LegendreSun}.

\subsection{Frobenius--Zagier process and Legendre--Sun series\label{subsec:FZ_LS}}Consider Euler's integral representation for hypergeometric functions \cite[Theorem 2.2.1]{AAR}\begin{align}&
_2F_1\left( \left.\begin{array}{@{}c@{}}
a,b \\
c \\
\end{array}\right|t \right)\notag\\={}&\frac{\Gamma(c)}{\Gamma(b)\Gamma(c-b)}\int_0^1\frac{X^{b-1}(1-X)^{c-b-1}}{(1-tX)^{a}}\D X,\quad \R c>\R b>0,-\pi<\arg(1-t)<\pi.\label{eq:Euler_int}\end{align}The next lemma shows that certain derivatives of $ _2F_1$ with respect to its parameters $a$, $b$, and $c$ can be expressed through Legendre functions and logarithms.

\begin{lemma}[Frobenius--Zagier]\label{lm:FB}For $ \nu\in(-1,0)$ and $ t\in(\mathbb C\smallsetminus\mathbb R)\cup(0,1)$, the following identities hold:\begin{align}\begin{split}&
\left.\frac{\partial}{\partial \varepsilon}\right|_{\varepsilon=0}{_2}F_1\left(\left. \begin{array}{@{}c@{}}
-\nu+\varepsilon,\nu+1+\varepsilon\ \\
1\ \\
\end{array} \right| t\right)\\={}&-P_\nu(1-2t)\log(1-t)\equalscolon P_\nu(1-2t)\mathfrak L_{\nu}(t),\end{split}\label{eq:FZ1}\\\begin{split}&\left.\frac{\partial}{\partial \varepsilon}\right|_{\varepsilon=0}{_2}F_1\left(\left. \begin{array}{@{}c@{}}
-\nu,\nu+1\ \\
1+\varepsilon\ \\
\end{array} \right| t\right)\\={}&\frac{\pi P_{\nu}(2t-1)}{2\sin(\nu\pi)}+\frac{P_\nu(1-2t)}{2}\left[ -2\gamma_{0}-\psi^{(0)}(-\nu)-\psi^{(0)}(\nu+1)+\log \frac{1-t}{t} \right]\\\equalscolon&\,P_\nu(1-2t)\widetilde{\mathfrak L}_{\nu}(t),\end{split}\label{eq:FZ2}
\end{align}where $ \psi^{(0)}(s)\colonequals \frac{\D}{\D s}\log\Gamma(s)$ is the digamma function, and $ \gamma_0\colonequals-\psi^{(0)}(1)$ is the Euler--Mascheroni constant.\end{lemma}\begin{proof}Both   \eqref{eq:FZ1} and   \eqref{eq:FZ2} have been proved in  \cite[Lemma 2.2.1]{AGF_PartII}. The main idea is as follows:\begin{itemize}
\item
If a  function $ f_\nu(t)\in C^2(0,1)$ satisfies\begin{align}
\frac{\D}{\D t}\left[ t(1-t)\frac{\D f_{\nu}(t)}{\D t} \right]+\nu(\nu+1)f_\nu(t)=0\label{eq:LegendreODE}
\end{align}for a certain $ \nu\in(-1,0)$ and every  $t\in(0,1)$, then it is a linear combination of $ P_\nu(1-2t)$ and $P_\nu(2t-1)$.\item If it is also true that $ f_\nu(0^+)$ and $f_\nu(1^-)
$ are both finite, then  $f_\nu(t)\equiv0,t\in(0,1)$. \end{itemize}The difference between the two sides of    \eqref{eq:FZ1} [or   \eqref{eq:FZ2}] fits the description in the two items above, and the corresponding  identity  analytically continues to   $ t\in(\mathbb C\smallsetminus\mathbb R)\cup(0,1)$.    \end{proof}\begin{remark}As shown in  \cite[Lemma 2.2.1 and Remark 2.2.1.1]{AGF_PartII}, the lemma above is inspired by the classical Frobenius process  \cite[(1.3.8)]{Slater} and Zagier's solution for the $ \nu=-\frac{1}{3}$ case \cite{Zagier1998}, hence the first half in the title of this subsection. In \cite[\S2]{Wei2023c}, Wei  proved \eqref{eq:FZ1} via Euler's hypergeometric transformation [cf.\ \eqref{eq:EulerPnumu} below], without explicitly invoking the Frobenius process. In \cite[\S2]{Wei2023Sun}, Wei's approach also resulted in two special values of \eqref{eq:FZ2} for $ \nu\in\left\{ -\frac{1}{4},-\frac{1}{3} \right\}$ and $t=\frac12$.  \eor\end{remark}\begin{remark}By analytic continuation, the validity of \eqref{eq:FZ1} extends to $ t\in(-\infty,0]$. Moreover, the right-hand side of  \eqref{eq:FZ2}  encounters no  jump discontinuities  as the variable $t$ moves
across the ray $ (-\infty,0]$, thus \eqref{eq:FZ2} remains effective for
$ t\in(-\infty,0]$, so long as one reinterprets its right-hand side through limit procedures.\eor\end{remark}

It is  an elementary exercise on Euler's (di)gamma functions to transcribe Lemma \ref{lm:FB}  into  Theorem \ref{thm:LegendreSun},  when $ \nu\in\left\{-\frac16,-\frac14,-\frac13,-\frac12\right\}$.

\subsection{Modular parametrizations of Legendre--Sun series\label{subsec:mod_para_LegendreSun}}

In the following corollary to Theorem \ref{thm:LegendreSun}, we consider $ N=4\sin^2(\nu\pi)\in\{1,2,3,4\}$ corresponding to $ \nu\in\left\{-\frac16,-\frac14,-\frac13,-\frac12\right\}$. We say that a
Legendre--Sun series $ \mathsf S_\nu^L(t)$ or
$\widetilde{ \mathsf S}_\nu^L(t)$ has complex multiplication if $0<|t|\leq1$, $t\neq1$ and \begin{align}
z_{t,N}\colonequals\lim_{\varepsilon\to0^+}\frac{iP_\nu(2(t+i\varepsilon)-1)}{\sqrt{N}P_\nu(1-2(t+i\varepsilon))}\in \mathfrak H\colonequals\{w\in\mathbb C|\I w>0\}\label{eq:z_N_lim}\end{align}is a CM point (namely, it is a point in the upper half-plane $  \mathfrak H$ that generates a quadratic extension of $ \mathbb Q$, satisfying
$ \I z_{t,N}>0$ and $ [\mathbb Q(z_{t,N}):\mathbb Q]=2$). We denote the totality of algebraic numbers by $ \overline{\mathbb Q}$.

\begin{corollary}[Some arithmetic properties of Legendre--Sun series]\label{cor:LS_CM}Every Legendre--Sun series with complex multiplication can be expressed in closed form \big[cf.\ \eqref{eq:FZ1} and \eqref{eq:FZ2} for the definitions of   $ \mathfrak  L_{\nu}(t)$ and $\widetilde{\mathfrak L}_{\nu}(t)$\big]:\begin{align}
  \mathsf S_\nu^L(t)={}&P_\nu(1-2t)\mathfrak L_{\nu}(t),\\
\widetilde{ \mathsf S}_\nu^L(t)={}&P_\nu(1-2t)\widetilde{\mathfrak L}_{\nu}(t),
\end{align}along with effective algorithms to compute the left-hand sides of the following relations:\begin{align}
\mathfrak
L_{\nu}(t)\in{}&\log\overline{\mathbb Q},\label{eq:logQbar}\\\widetilde{\mathfrak L}_\nu(t)-\pi i z_{t,N}\in{}&\log\overline{\mathbb Q},\label{eq:logQbar'}\\\log P_\nu(1-2t) \in{}&\log\overline{\mathbb Q}+\mathbb Q\log\pi+\Span_{\mathbb Q}\{\log \Gamma(r)|r\in\mathbb Q\cap(0,1)\}.\label{eq:CS}
\end{align}\end{corollary}\begin{proof}Define  modular invariants  \begin{align}
\alpha_N(z)\colonequals{}&
\left\{1+\frac{1}{N^{6/(N-1)}}\left[ \frac{\eta(z)}{\eta(Nz)} \right]^{24/(N-1)}\right\}^{-1}=1-\alpha_N\left( -\frac{1}{N z} \right),\label{eq:alpha_N}\\j(z)\colonequals{}&\frac{2^{8}\{1-\alpha_{4}(z/2)+[\alpha_{4}(z/2)]^{2}\}^{3}}{[\alpha_{4}(z/2)]^{2}[1-\alpha_{4}(z/2)]^2}=j\left(- \frac{1}{z} \right),\label{eq:jKlein}
\end{align}for $ N\in\{2,3,4\}$ by the Dedekind eta function $ \eta(z)\colonequals e^{\pi iz/12}\prod_{n=1}^\infty(1-e^{2\pi inz}),z\in\mathfrak H$. According to \cite[\S12, Lemmata 2 and 3]{SelbergChowla}, we have $ \alpha_N(z)\in\overline{\mathbb Q}\cup\{\infty\},N\in\{2,3,4\}$ and  $ j(z)\in\overline{\mathbb Q}$ when $ [\mathbb Q(z):\mathbb Q]=2$. This verifies  \eqref{eq:logQbar} and \eqref{eq:logQbar'} for  $  t\in(\mathbb C\smallsetminus\mathbb R)\cup(0,1)$.

By Ramanujan's base changes \cite[Chapter 33, Corollary~3.4, Theorem~9.9 and Theorem~11.6]{RN5}, we have \begin{align}
[\eta(z)]^{24}=\begin{cases}\frac{1}{j(z)}\lim_{\varepsilon\to0^+}\big[P_{-1/6}\big(\sqrt{[j(z+i\varepsilon)-1728]/j(z+i\varepsilon)}\big)\big]^{12}  \\
\frac{\alpha_{2}(z)[1-\alpha_2(z)]^{2}}{2^{6}} [P_{-1/4}(1-2\alpha_2(z))]^{12}\\
\frac{\alpha_{3}(z)[1-\alpha_3(z)]^{3}}{3^{3}} [P_{-1/3}(1-2\alpha_3(z))]^{12} \\\frac{\alpha_{4}(z)[1-\alpha_4(z)]^{4}}{2^{4}} [P_{-1/2}(1-2\alpha_4(z))]^{12}
\\
\end{cases}
\end{align}
if the ranges of the context-dependent $z$ are specified by \begin{align}
\I z>0,\quad |\R z|<\frac{1}{2},\quad \left\vert z+\frac{1}{N} \right\vert>\frac{1}{N},\quad \left\vert z-\frac{1}{N} \right\vert>\frac{1}{N}\label{eq:D234}
\end{align} for $ N=4\sin^2(\nu\pi)\in\{2,3,4\}$, and \begin{align}
\I z>0,\quad |\R z|<\frac{1}{2},\quad |z|\geq1\label{eq:D1}
\end{align}    for  $ N=4\sin^2(-\frac\pi6)=1$.
The Chowla--Selberg theory \cite{ChowlaSelberg,SelbergChowla} essentially provides an explicit formula for every CM value of the Dedekind eta function $\eta(z)$, such that $ \log \eta(z)\in \log\overline{\mathbb Q}+\mathbb Q\log\pi+\Span_{\mathbb Q}\{\log \Gamma(r)|r\in\mathbb Q\cap(0,1)\}$. Therefore,  we have confirmed \eqref{eq:CS} for  $ \nu\in\left\{-\frac16,-\frac14,-\frac13,-\frac12\right\}$ and $  t\in(\mathbb C\smallsetminus\mathbb R)\cup(0,1)$.

Thus far, our discussions are confined to the scenarios where $t\notin[-1,0)$ in the Legendre--Sun series with complex multiplications. The cases of $ t\in[-1,0)$ can be handled by  the limit procedures in \eqref{eq:z_N_lim}. We leave the details to diligent readers.
\end{proof}\begin{remark}The left-hand sides  of    \cite[(29) and (34)–-(37)]{Sun2022} represent Legendre--Sun series with complex multiplications, corresponding to  $ z_{t,3}=\frac{1+i\sqrt{3}}{2}$ and $ z_{t,4}\in\left\{\frac{i}{2},i,\frac{1+i\sqrt{3}}{2},\frac{1+i\sqrt{7}}{2}\right\}$. After invoking Ramanujan's base changes and the Chowla--Selberg theory, we may reduce   the right-hand sides of     \cite[(29) and (34)--(37)]{Sun2022}  to
{\allowdisplaybreaks\begin{align}\mathsf S_{-1/3}^L\left( -\frac{1}{8} \right)={}&-\frac{\big[\Gamma\big(\frac13\big)\big]^3}{2\pi^{2}}\log\frac{8}{9},
\tag{S-29$'$}\\\mathsf S_{-1/4}^L\left( \frac{8}{9} \right)={}&-\frac{\sqrt{6}\big[\Gamma\big(\frac14\big)\big]^2}{2\pi\sqrt{\pi}}\log3,\tag{S-34$'$}\\\mathsf S_{-1/4}^L\left( \frac{1}{9} \right)={}&-\frac{\sqrt{3}\big[\Gamma\big(\frac14\big)\big]^2}{4\pi\sqrt{\pi}}\log\frac{9}{8},\tag{S-35$'$}\\\mathsf S_{-1/4}^L\left(- \frac{1}{3} \right)={}&-\frac{\sqrt{3}\big[\Gamma\big(\frac13\big)\big]^3}{2^{11/6}\pi^{2}}\log\frac{3}{4},\tag{S-36$'$}\\\mathsf S_{-1/4}^L\left( -\frac{1}{63} \right)={}&-\frac{\sqrt{3}\Gamma\big(\frac17\big)\Gamma\big(\frac27\big)\Gamma\big(\frac47\big)}{4\sqrt{2}\pi^{2}}\log\frac{63}{64}.\tag{S-37$'$}
\end{align}By similar procedures, one can construct \begin{align}
\mathsf S_{-1/6}^{L}\left(\frac{1}{2}- \frac{7}{10}\sqrt{\frac{2}{5}} \right)={}&-\frac{\Gamma\big(\frac18\big)\Gamma\big(\frac38\big)}{4\pi\sqrt{\pi}}\sqrt[4]{\frac52}\log\left(\frac{1}{2}+ \frac{7}{10}\sqrt{\frac{2}{5}} \right),\tag{\ref{eq:LS2_k2}$'$}\\
\mathsf S_{-1/6}^{L}\left(\frac{1}{2}- \frac{11}{10\sqrt{5}} \right)={}&-\frac{\big[\Gamma\big(\frac13\big)\big]^3}{4\pi^{2}}\frac{\sqrt{3}\sqrt[4]5}{\sqrt[3]2}\log \left(\frac{1}{2}+ \frac{11}{10\sqrt{5}} \right),\tag{\ref{eq:LS2_k3}$'$}\\
\mathsf S_{-1/6}^{L}\left(\frac{1}{2}- \frac{21}{22}\sqrt{\frac{3}{11}} \right)={}&-\frac{\big[\Gamma\big(\frac14\big)\big]^2}{4\pi\sqrt{2\pi}}\sqrt[4]{33}\log\left(\frac{1}{2}+ \frac{21}{22}\sqrt{\frac{3}{11}} \right),\tag{\ref{eq:LS2_k4}$'$}\\
\mathsf S_{-1/6}^{L}\left(\frac{1}{2}- \frac{171}{170}\sqrt{\frac{21}{85}} \right)={}&-\frac{\Gamma\big(\frac17\big)\Gamma\big(\frac27\big)\Gamma\big(\frac47\big)}{8\pi^{2}}\sqrt[4]{\frac{255}{7}}\log\left(\frac{1}{2}+ \frac{171}{170}\sqrt{\frac{21}{85}} \right),\tag{\ref{eq:LS2_k7}$'$}
\end{align}}as companions to \eqref{eq:LS2_k2}--\eqref{eq:LS2_k7}.
\eor\end{remark} \subsection{Derivatives and integrals related to  Legendre--Sun series\label{subsec:int_LS}}

\subsubsection{Derivatives with respect to degrees of  Legendre functions}

Evaluating the higher-order  derivatives \cite[\S3.4(20)]{HTF1} \begin{align} \left.\frac{\partial ^{2m}}{\partial \nu^{2m}}\right|_{\nu=-1/2}P_\nu(0)=\left.\frac{\partial ^{2m}}{\partial \nu^{2m}}\right|_{\nu=-1/2}\frac{\sqrt{\pi }}{\Gamma \big(\frac{1-\nu}{2}\big) \Gamma \big(\frac{\nu +2}{2}\big)},\quad m\in\mathbb Z_{>0}\label{eq:Pnu0deriv}\end{align} via the $t=\frac12 $ case of \begin{align}
P_{\nu}(1-2t)={}&\sum_{k=0}^\infty\frac{(-\nu)_k(\nu+1)_k}{(k!)^{2}}t^k,\label{eq:PnuTaylor}
\end{align}one obtains an infinite family of series identities. One can use    \eqref{eq:FZ1}--\eqref{eq:FZ2} to compute the same set of higher-order  derivatives, some of which are displayed in Table \ref{tab:central_deriv}.
 \begin{table}[t]\caption{Selected series involving $ \binom{2k}k^2$,  where $ h_k^{(r)}\colonequals\mathsf H_{2k}^{(r)}-\frac{1}{2^{r}}\mathsf H_k^{(r)}$, $ \lambda\colonequals\log2$, $\zeta(s)\colonequals\sum_{k=1}^\infty\frac{1}{k^s} $, $ \beta(s)\colonequals\sum_{k=0}^\infty\frac{(-1)^k}{(2k+1)^{s}}$, and $ G\colonequals\beta(2)$ \label{tab:central_deriv}}

\begin{scriptsize}\begin{align*}\begin{array}{r@{}l}\hline\hline
\displaystyle \sum_{k=0}^\infty\vphantom{\frac{\frac{\frac11}{1}}{1}}h_k^{(2)} \frac{\binom{2k}k^2}{2^{5k}}={}&-\displaystyle \frac{\big[\Gamma\big(\frac{1}{4}\big)\big]^2}{4\pi\sqrt{\pi}}\left(G-\frac{\pi^2}8\right) \\[12pt]
\displaystyle \sum_{k=0}^\infty\left\{  h_k^{(4)}-\left[ h_k^{(2)}\right]^{2}\right\} \frac{\binom{2k}k^2}{2^{5k}}={}&-\displaystyle \frac{\big[\Gamma\big(\frac{1}{4}\big)\big]^2}{4\pi\sqrt{\pi}}\left[ \beta(4)+\frac{1}{2}\left(G-\frac{\pi^2}8\right)^{2}-\frac{\pi^{4}}{96}\right] \\[12pt]\displaystyle \sum_{k=0}^\infty \left\{ h_k^{(6)} -\frac{3h_k^{(4)}-\left[ h_k^{(2)}\right]^{2}}{2}h_k^{(2)}\right\}  \frac{\binom{2k}k^2}{2^{5k}}={}&-\displaystyle \frac{\big[\Gamma\big(\frac{1}{4}\big)\big]^2}{4 \pi\sqrt{\pi}} \left\{ \beta (6)+\frac{3}{4} \left[ \beta (4)-\frac{\pi^{4}}{96}\right] \left(G-\frac{\pi^2}8\right)\vphantom{\big(^2}\right.\\{}&\displaystyle\left.{}+\frac{1}{8} \left(G-\frac{\pi^2}8\right)^3 -\frac{\pi ^6}{960}\right\}\\[5pt]\hline \displaystyle \sum_{k=0}^\infty\vphantom{\frac{\frac{\frac11}{1}}{1}}\left[ h_k^{(3)}- h_k^{(1)} h_k^{(2)}\right]  \frac{\binom{2k}k^2}{2^{5k}}={}&\displaystyle \frac{\big[\Gamma\big(\frac{1}{4}\big)\big]^2}{16\pi\sqrt{\pi}}\left(G-\frac{\pi^2}8\right)\lambda  \\
\displaystyle \sum_{k=0}^\infty\left\{ h_{k} ^{(5)}-\frac{h_k^{(1)} h_k^{(4)}}{2}-h_k^{(2)} h_k^{(3)}+\frac{h_k^{(1)}\left[ h_k^{(2)}\right]^{2}}{2}\right\} \frac{\binom{2k}k^2}{2^{5k}}={}&\displaystyle \frac{\big[\Gamma\big(\frac{1}{4}\big)\big]^2}{32\pi\sqrt{\pi}}\left[ \beta(4)+\frac{1}{2}\left(G-\frac{\pi^2}8\right)^{2}-\frac{\pi^{4}}{96}\right]\lambda\\[15pt]\hline \displaystyle \sum_{k=0}^\infty \vphantom{\frac{\frac{\frac11}{1}}{1}}h_k^{(2)}\mathsf H_k^{\vphantom{()}}\frac{\binom{2k}k^2}{2^{5k}}={}&\displaystyle \frac{\big[\Gamma\big(\frac{1}{4}\big)\big]^2}{8\pi\sqrt{\pi}} \left[ 7\zeta(3)+G(4\lambda-\pi)-\frac{(4\lambda+\pi)\pi^{2}}{8} \right] \\\displaystyle \sum_{k=0}^\infty \left\{  h_k^{(4)}-\left[ h_k^{(2)}\right]^{2}\right\}\mathsf H_k^{\vphantom{()}} \frac{\binom{2k}k^2}{2^{5k}}={}&\displaystyle \frac{\big[\Gamma\big(\frac{1}{4}\big)\big]^2}{16\pi\sqrt{\pi}}\left\{ 31\zeta(5)+2\left[ \beta(4)-\frac{G^{2}}{2} \right](4\lambda-\pi)\right.\\{}&\displaystyle\left.{}+14\zeta(3)\left(G-\frac{\pi^2}8\right)-\frac{\pi^{2}G(4\lambda+\pi)}{4}-\frac{\pi ^4 \lambda }{48}-\frac{7 \pi ^5}{192} \right\}  \\[7pt]\hline\hline
\end{array}\end{align*}\end{scriptsize}
\end{table}

 Here,  the first entry in Table \ref{tab:central_deriv} [namely the $ m=1$ scenario of \eqref{eq:Pnu0deriv} and \eqref{eq:PnuTaylor}] was singled out by  Chu \cite[\S5, p.\ 588]{Chu2022} and  rediscovered in Sun's numerical experiments \cite[Conjecture 9]{Sun2022}, before being reworked by Wei \cite[\S3]{Wei2023c}. Later in Table \ref{tab:EZF4}, we will generalize \cite[Conjecture 9]{Sun2022} along a different direction.

Taking the second-order derivative of   \cite[\S3.4(6)]{HTF1}\begin{align}
\Gamma (1-\varepsilon )P_\nu^\varepsilon(0)={_{2}F_{1}}\left(\left.\begin{array}{@{}c@{}}
-\nu,\nu+1\\
1-\varepsilon\\
\end{array}\right| \frac{1}{2}\right)=\frac{2^{\varepsilon } \Gamma (1-\varepsilon) \Gamma \big(\frac{1+\varepsilon +\nu }{2} \big)\cos\frac{  (\nu +\varepsilon )\pi}{2} }{\sqrt{\pi } \Gamma \big(\frac{2-\varepsilon +\nu }{2} \big)}
\end{align}    with respect to $\varepsilon $, before specializing to $\varepsilon=0$ and $\nu=-\frac12$, we get \begin{align}
\sum_{k=0}^\infty\frac{\binom{2k}k^2}{2^{5k}}\left[ \mathsf H_k^{(2)}+\mathsf H_k^{2\vphantom)} \right]={}&\frac{2\big[\Gamma\big(\frac14\big)\big]^2}{\pi\sqrt{\pi}}\left[ G+\left(\frac{\pi}{4}  -\log2\right)^{2}-\frac{\pi^2}{12} \right],\label{eq:Chu_order'}
\end{align}as observed by Chu \cite[\S5, p.\ 587]{Chu2022}.

Applying Clausen's formula \cite[last equation]{Clausen1828} \begin{align}\frac{(\arcsin\sqrt{x})^2}{x}={_3F_2}\left(\left.\begin{array}{@{}c@{}}
1,1,1 \\
\frac{3}{2},2 \\
\end{array}\right| x\right)\label{eq:arcsin_3F2}
\end{align}to Campbell's identity \cite[(5.2)]{Campbell2023CCG}\begin{align}
\sum_{k=0}^\infty\frac{\binom{2k}k^2}{2^{5k}}\mathsf H_k^{(2)}=\frac{\sqrt{\pi}\big[\Gamma\big(\frac14\big)\big]^2}{12}-\frac{2\sqrt{2}}{\pi}\int_0^1\frac{(\arcsin\sqrt{x})^2}{\sqrt{x}\sqrt{1-x^2}}\D x,
\end{align} we obtain \begin{align}
\sum_{k=0}^\infty\frac{\binom{2k}k^2}{2^{5k}}\mathsf H_k^{(2)}=\frac{\sqrt{\pi}\big[\Gamma\big(\frac14\big)\big]^2}{12}-\frac{\big[\Gamma\big(\frac14\big)\big]^2}{9\pi\sqrt{\pi}} {_4F_3}\left(\left.\begin{array}{@{}c@{}}
1,1,1,\frac{3}{2} \\[4pt]
\frac{7}{4},\frac{7}{4},2 \\
\end{array}\right| 1\right)-\frac{8\sqrt{\pi}}{\big[\Gamma\big(\frac14\big)\big]^2}{_4F_3}\left(\left.\begin{array}{@{}c@{}}
\frac{1}{2},\frac{1}{2},1,1 \\[4pt]
\frac{5}{4},\frac{5}{4},\frac{3}{2} \\
\end{array}\right| 1\right)\label{eq:CC''}
\end{align}through termwise integration of the Taylor expansion for $_3F_2 $ [see \eqref{eq:defn_pFq}]. Currently, we do not know how to reduce the two $ _4F_3$ expressions in \eqref{eq:CC''} to well-known mathematical constants. Numerical experiments based on the \texttt{PSLQ} algorithm seem to suggest that both \begin{align}
\left\{ {_4F_3}\left(\left.\begin{array}{@{}c@{}}
1,1,1,\frac{3}{2} \\[4pt]
\frac{7}{4},\frac{7}{4},2 \\
\end{array}\right| 1\right),G,\pi^{2},\pi\log2,\log^{2}2\right\}
\end{align} and \begin{align}
\left\{ {_4F_3}\left(\left.\begin{array}{@{}c@{}}
\frac{1}{2},\frac{1}{2},1,1 \\[4pt]
\frac{5}{4},\frac{5}{4},\frac{3}{2} \\
\end{array}\right| 1\right),G,\pi^{2},\pi\log2,\log^{2}2\right\}
\end{align}are linearly independent over $\mathbb Q$.

Unlike our previous analysis  \cite[\S3]{Zhou2022mkMpl} of \eqref{eq:binomH}  for arbitrary polynomials of harmonic numbers in the summands, we have fewer degrees of freedom    when $ \binom{2k}k^2$ is brought
into play. Discounting the  $ _4F_3$ expressions in \eqref{eq:CC''}, we are still unable to render the individual sums  \cite{Campbell2023Sun,Chu2022}\begin{align}
\sum_{k=0}^\infty\frac{\binom{2k}k^2}{2^{5k}}\mathsf H_k^{2\vphantom)},\quad \sum_{k=0}^\infty\frac{\binom{2k}k^2}{2^{5k}}\mathsf H_k^{(2)}
,\quad \sum_{k=0}^\infty\frac{\binom{2k}k^2}{2^{5k}}\mathsf H_{2k}^{2\vphantom)},\quad \sum_{k=0}^\infty\frac{\binom{2k}k^2}{2^{5k}}\mathsf H_{2k}^{(2)}
\end{align}into closed forms.

\subsubsection{Integrals involving Legendre--Sun series}

By termwise integrations over the series representations of $\mathsf S^L_\nu(t) $ and $ \widetilde{\mathsf S}_\nu^L(t)$, one can evaluate several classes of non-trivial
infinite sums, as described in the next two corollaries.

\begin{corollary}[Binomial modulations of  Legendre--Sun series]\label{cor:binomLS}\begin{enumerate}[leftmargin=*,  label=\emph{(\alph*)},ref=(\alph*),
widest=d, align=left] \item We have {\allowdisplaybreaks\begin{align}
\begin{split}\int_0^1\frac{\mathsf S_{-1/2}^L(t)}{\sqrt{1-t}}\D t={}&4\sum_{k=0}^\infty\frac{\binom{2k}k}{2^{2k}}\frac{2\mathsf H_{2k}-\mathsf H_k}{2k+1}=4\pi \log 2,\end{split}\label{eq:LS2_ib}\\\begin{split}\int_0^1\frac{\mathsf S_{-1/3}^L(t)}{\sqrt{1-t}}\D t={}&2\sum_{k=0}^\infty\frac{2^{2k}\binom{3k}k}{3^{3k}}\frac{3\mathsf H_{3k}-\mathsf H_k}{2k+1}=9(2- \log 3),\end{split}\label{eq:LS3_ib}
\\\begin{split}\int_0^1\frac{\mathsf S_{-1/4}^L(t)}{\sqrt{1-t}}\D t={}&4\sum_{k=0}^\infty\frac{\binom{4k}{2k}}{2^{4k}}\frac{2\mathsf H_{4k}-\mathsf H_{2k}}{2k+1}=4\sqrt{2} (2-\log 2),\end{split}\label{eq:LS4_ib}
\\\begin{split}\int_0^1\frac{\mathsf S_{-1/6}^L(t)}{\sqrt{1-t}}\D t={}&2\sum_{k=0}^\infty\frac{\binom{3k}k\binom{6k}{3k}
 }{(2^{2}3^3)^k\binom{2k}k}\frac{6\mathsf H_{6k}-3\mathsf H_{3k}-2\mathsf H_{2k}+\mathsf H_k}{2k+1}= \frac{3\sqrt{3}}2\left(3+\log \frac{2^{4}}{3^{3}}\right).\end{split}
\end{align}More generally, for $ m\in\mathbb Z_{\geq0}$, we have \begin{align}\begin{split}
\int_0^1\frac{\mathsf S_{-1/2}^L(t)}{(1-t)^{\frac12-m}}\D t={}&4\sum_{k=0}^\infty\frac{\binom{2k}k}{2^{2k}}\frac{(2\mathsf H_{2k}-\mathsf H_k)(2k-1)!!(2m-1)!!}{(2k+2m+1)!!}\\\in{}&\mathbb Q\pi+\mathbb Q\pi \log 2,\end{split}\\\begin{split}
\int_0^1\frac{\mathsf S_{-1/3}^L(t)}{(1-t)^{\frac12-m}}\D t={}&2\sum_{k=0}^\infty\frac{2^{2k}\binom{3k}k}{3^{3k}}\frac{(3\mathsf H_{3k}-\mathsf H_k)(2k-1)!!(2m-1)!!}{(2k+2m+1)!!}\\\in{}&\mathbb Q+\mathbb Q \log 3,\end{split}\\\begin{split}
\int_0^1\frac{\mathsf S_{-1/4}^L(t)}{(1-t)^{\frac12-m}}\D t={}&4\sum_{k=0}^\infty\frac{\binom{4k}{2k}}{2^{4k}}\frac{(2\mathsf H_{4k}-\mathsf H_{2k})(2k-1)!!(2m-1)!!}{(2k+2m+1)!!}\\\in{}&\mathbb Q\sqrt{2}+\mathbb Q\sqrt{2} \log 2,\end{split}\\\begin{split}
\int_0^1\frac{\mathsf S_{-1/6}^L(t)}{(1-t)^{\frac12-m}}\D t={}&2\sum_{k=0}^\infty\frac{\binom{6k}{3k}\binom{3k}k
 }{(2^{2}3^3)^k\binom{2k}k}\frac{(6\mathsf H_{6k}-3\mathsf H_{3k}-2\mathsf H_{2k}+\mathsf H_k)(2k-1)!!(2m-1)!!}{(2k+2m+1)!!}\\\in{}&\mathbb Q\sqrt{3}+\mathbb Q\sqrt{3} \log \frac{2^{4}}{3^{3}},\end{split}\label{eq:LS6_ib'}
\end{align}}where the coefficients in the $ \mathbb Q$-linear combinations are explicitly computable. \big[Here, double factorials of odd integers
are defined as $ (-1)!!\colonequals 1$ and $ (2n+1)!!\colonequals\frac{(2n+1)!}{2^{n}n!}$ for $ n\in\mathbb Z_{\geq0}$.\big]\item We have {\allowdisplaybreaks\begin{align}
\begin{split}\int_0^1\frac{ \mathsf S_{-1/2}^L(t)+\widetilde{\mathsf S}_{-1/2}^L(t) }{\pi\sqrt{t(1-t)}}\D t={}&\sum_{k=0}^\infty\frac{\binom{2k}k^3}{2^{6k}}(4\mathsf H_{2k}-3\mathsf H_k)=\frac{\big[\Gamma\big(\frac14\big)\big]^4}{2\pi^3}\log 2,\end{split}\label{eq:LS2_ia}\\\begin{split}\int_0^1\frac{ \mathsf S_{-1/3}^L(t)+\widetilde{\mathsf S}_{-1/3}^L(t) }{\pi\sqrt{t(1-t)}}\D t={}&\sum_{k=0}^\infty\frac{\binom{2k}k^2\binom{3k}k}{2^{2k}3^{3k}}(3\mathsf H_{3k}-2\mathsf H_k)=\frac{9\sqrt[3]{2}\big[\Gamma\big(\frac13\big)\big]^6}{16\pi^4}\log 2,\end{split}\\\begin{split}\int_0^1\frac{ \mathsf S_{-1/4}^L(t)+\widetilde{\mathsf S}_{-1/4}^L(t) }{\pi\sqrt{t(1-t)}}\D t={}&\sum\frac{\binom{2k}k^2\binom{4k}{2k}}{2^{8k}}(4\mathsf H_{4k}-2\mathsf H_{2k}-\mathsf H_k)=\frac{2\pi\log 2}{\big[\Gamma\big(\frac58\big)\Gamma\big(\frac78\big)\big]^2},\end{split}\\\begin{split}\int_0^1\frac{ \mathsf S_{-1/6}^L(t)+\widetilde{\mathsf S}_{-1/6}^L(t) }{\pi\sqrt{t(1-t)}}\D t={}&\sum_{k=0}^\infty\frac{\binom{2k}k\binom{3k}k\binom{6k}{3k}}{2^{6k}3^{3k}}(6\mathsf H_{6k}-3\mathsf H_{3k}-2\mathsf H_{2k})=\frac{2\pi\log 2}{\big[\Gamma\big(\frac7{12}\big)\Gamma\big(\frac{11}{12}\big)\big]^2}.\end{split}
\end{align}}More generally, for $ m\in\mathbb Z_{\geq0}$,  we have {\allowdisplaybreaks\begin{align}\begin{split}\int_0^1\frac{ \mathsf S_{-1/2}^L(t)+\widetilde{\mathsf S}_{-1/2}^L(t) }{\pi[t(1-t)]^{\frac12-m}}\D t={}&\sum_{k=0}^\infty\frac{\binom{2k}k^3}{2^{6k+3m}}\frac{(4\mathsf H_{2k}-3\mathsf H_k)(2k+2m-1)!!}{\binom{k+2m}{k}(2k-1)!!m!^{}}\\\in{}&\mathbb Q\frac{\big[\Gamma\big(\frac14\big)\big]^4}{\pi^{3}}+\mathbb Q\frac{\big[\Gamma\big(\frac14\big)\big]^4}{\pi^{3}}\log 2,\end{split}\\\begin{split}\int_0^1\frac{ \mathsf S_{-1/3}^L(t)+\widetilde{\mathsf S}_{-1/3}^L(t) }{\pi[t(1-t)]^{\frac12-m}}\D t={}&\sum_{k=0}^\infty\frac{\binom{2k}k^2\binom{3k}k}{2^{2k+3m}3^{3k}}\frac{(3\mathsf H_{3k}-2\mathsf H_k)(2k+2m-1)!!}{\binom{k+2m}{k}(2k-1)!!m!^{}}\\\in{}&\mathbb Q\frac{\sqrt[3]{2}\big[\Gamma\big(\frac13\big)\big]^6}{\pi^4}+\mathbb Q\frac{\sqrt[3]{2}\big[\Gamma\big(\frac13\big)\big]^6}{\pi^4}\log 2,\end{split}\\\begin{split}\int_0^1\frac{ \mathsf S_{-1/4}^L(t)+\widetilde{\mathsf S}_{-1/4}^L(t) }{\pi[t(1-t)]^{\frac12-m}}\D t={}&\sum_{k=0}^\infty\frac{\binom{2k}k^2\binom{4k}{2k}}{2^{8k+3m}}\frac{(4\mathsf H_{4k}-2\mathsf H_{2k}-\mathsf H_k)(2k+2m-1)!!}{\binom{k+2m}{k}(2k-1)!!m!^{}}\\\in{}&\mathbb Q\frac{\pi}{\big[\Gamma\big(\frac58\big)\Gamma\big(\frac78\big)\big]^2}+\mathbb Q\frac{\pi\log 2}{\big[\Gamma\big(\frac58\big)\Gamma\big(\frac78\big)\big]^2},\end{split}\\\begin{split}\int_0^1\frac{ \mathsf S_{-1/6}^L(t)+\widetilde{\mathsf S}_{-1/6}^L(t) }{\pi[t(1-t)]^{\frac12-m}}\D t={}&\sum_{k=0}^\infty\frac{\binom{2k}k\binom{3k}k\binom{6k}{3k}}{2^{6k+3m}3^{3k}}\frac{(6\mathsf H_{6k}-3\mathsf H_{3k}-2\mathsf H_{2k})(2k+2m-1)!!}{\binom{k+2m}{k}(2k-1)!!m!^{}}\\\in{}&\mathbb Q\frac{\pi}{\big[\Gamma\big(\frac7{12}\big)\Gamma\big(\frac{11}{12}\big)\big]^2}+\mathbb Q\frac{\pi\log 2}{\big[\Gamma\big(\frac7{12}\big)\Gamma\big(\frac{11}{12}\big)\big]^2},\end{split}\label{eq:LS6_ia'}\end{align}}where the coefficients in the $ \mathbb Q$-linear combinations are explicitly computable.\end{enumerate}\end{corollary}
\begin{proof}\begin{enumerate}[leftmargin=*,  label={(\alph*)},ref=(\alph*),
widest=d, align=left]\item We paraphrase  \cite[item 7.127]{GradshteynRyzhik} (which reproduces \cite[\S18.1(15)]{ET2}) as\begin{align}
\int_0^1 P_{\nu}(1-2t)(1-t)^\sigma\D t=\frac{[\Gamma(1+\sigma)]^2}{\Gamma(2+\nu+\sigma)\Gamma(1-\nu+\sigma)},\quad \R \sigma>-1.
\end{align}In particular, for each $ m\in\mathbb Z_{\geq0}$, the derivative of the formula above at $ \sigma=m-\frac12$ brings us\begin{align}\begin{split}
&\int_0^1 \frac{P_{\nu}(1-2t)}{(1-t)^{\frac12-m}}\log(1-t)\D t\\={}&\frac{\left[\Gamma \big(m+\frac{1}{2}\big)\right]^2 }{\Gamma \big(m-\nu +\frac{1}{2}\big) \Gamma \big(m+\nu +\frac{3}{2}\big)}\left[-\psi ^{(0)}\left(m-\nu +\frac{1}{2}\right)-\psi ^{(0)}\left(m+\nu +\frac{3}{2}\right)+2 \psi ^{(0)}\left(m+\frac{1}{2}\right)\right]\end{split}
\end{align}for $ \nu\neq-1/2$. Exploring the $\nu\to-1/2$ limit and the special cases of $\nu\in\left\{-\frac13,-\frac14,-\frac16\right\}$, while bearing in mind that \begin{align}
\int_0^1 \frac{t^{k}\D t}{(1-t)^{\frac12-m}}=\frac{2^{1+2k}}{\binom{2k}k}\frac{(2k-1)!!(2m-1)!!}{(2k+2m+1)!!},
\end{align}we may turn Theorem \ref{thm:LegendreSun}(a) into \eqref{eq:LS2_ib}--\eqref{eq:LS6_ib'}.
\item We  specialize \cite[item 7.132.1]{GradshteynRyzhik} (which corrects \cite[\S18.1(16)]{ET2}) into\begin{align}\begin{split}&
\int_0^1 P_{\nu}(1-2t)[t(1-t)]^\sigma\D t\\={}&\frac{\pi}{2^{1+2\sigma}}\frac{[\Gamma(1+\sigma)]^2}{\Gamma\big(\frac{\nu +3}{2}+\sigma\big)\Gamma\big(\sigma -\frac{\nu -2}{2}\big)\Gamma\big(\frac{2+\nu}{2}\big)\Gamma\big(\frac{1-\nu }{2}\big)},\quad \R \sigma>-1.\end{split}\label{eq:Pnu_avg}
\end{align}For each $ m\in\mathbb Z_{\geq0}$,  the zeroth- and first-order derivatives of the equation above at  $ \sigma=m-\frac12$ leave us\begin{align}\begin{split}&
\int_0^1 \frac{P_{\nu}(1-2t)}{[t(1-t)]^{\frac12-m}}\D t=\int_0^1 \frac{P_{\nu}(2t-1)}{[t(1-t)]^{\frac12-m}}\D t\\={}&\frac{\pi   \left[\Gamma \big(m+\frac{1}{2}\big)\right]^2}{2^{2m}\Gamma \big(\frac{1-\nu}{2}\big) \Gamma \big(\frac{\nu+2 }{2}\big) \Gamma \big(m-\frac{\nu-1 }{2}\big) \Gamma \big(m+\frac{\nu+1 }{2}\big)}\equalscolon M_{m,\nu}\end{split}
\end{align}and \begin{align}\begin{split}
&\int_0^1 \frac{P_{\nu}(1-2t)}{[t(1-t)]^{\frac12-m}}[\log t+\log(1-t)]\D t\\={}&\left[ 2 \psi ^{(0)}\left(m+\frac{1}{2}\right)-\psi ^{(0)}\left(m+\frac{\nu+1 }{2}\right)-\psi ^{(0)}\left(m-\frac{\nu-1 }{2}\right)-2 \log 2 \right]M_{m,\nu}.\end{split}
\end{align}As in our proof of  \eqref{eq:LS2_ib}--\eqref{eq:LS6_ib'}, we may  set   $\nu\in\left\{-\frac12,-\frac13,-\frac14,-\frac16\right\}$ in the last two equations and integrate over the identities in Theorem  \ref{thm:LegendreSun}, to verify  \eqref{eq:LS2_ia}--\eqref{eq:LS6_ia'}.
\qedhere\end{enumerate}\begin{remark}The series evaluation in \eqref{eq:LS2_ib} is a special case of \cite[Corollary 3.8]{Zhou2022mkMpl}. One may  also compute  \eqref{eq:LS3_ib} and \eqref{eq:LS4_ib} by extending some arguments in \cite[\S3]{Zhou2023SunCMZV}.  \eor\end{remark}
\end{proof}\begin{corollary}[Harmonic modulations of  Legendre--Sun series]\label{cor:HkLS}We have {\allowdisplaybreaks\begin{align}\begin{split}
-\int_0^1\mathsf S^L_{-1/2}(t)\log(1-t)\D t={}&2\sum_{k=0}^\infty\frac{\binom{2k}k^2}{2^{4k}}\frac{(2\mathsf H_{2k}-\mathsf H_k)\mathsf H_{k+1}}{k+1}\\={}&\frac{96}{\pi }-\frac{8 \pi }{3}-\frac{128 \log 2}{\pi }+\frac{64 \log ^22}{\pi },\end{split}\label{eq:LS2_ic}\\\begin{split}
-\int_0^1\mathsf S^L_{-1/3}(t)\log(1-t)\D t={}&\sum_{k=0}^\infty\frac{\binom{2k}k\binom{3k}k}{3^{3k}}\frac{(3\mathsf H_{3k}-\mathsf H_k)\mathsf H_{k+1}}{k+1}\\={}&\frac{567 \sqrt{3}}{8 \pi }-\frac{9 \sqrt{3} \pi }{4}-\frac{243 \sqrt{3} \log3}{4 \pi }+\frac{81 \sqrt{3} \log ^23}{4 \pi },\end{split}\\\begin{split}
-\int_0^1\mathsf S^L_{-1/4}(t)\log(1-t)\D t={}&2\sum_{k=0}^\infty\frac{\binom{2k}k\binom{4k}{2k}}{2^{6k}}\frac{(2\mathsf H_{4k}-\mathsf H_{2k})\mathsf H_{k+1}}{k+1}\\={}&\frac{3328 \sqrt{2}}{27 \pi }-\frac{40 \sqrt{2} \pi }{9}-\frac{512 \sqrt{2} \log 2}{3 \pi }+\frac{96 \sqrt{2} \log ^22}{\pi },\end{split}\\\begin{split}
-\int_0^1\mathsf S^L_{-1/6}(t)\log(1-t)\D t={}&\sum_{k=0}^\infty\frac{\binom{3k}k\binom{6k}{3k}}{2^{4k}3^{3k}}\frac{(6\mathsf H_{6k}-3\mathsf H_{3k}-2\mathsf H_{2k}+\mathsf H_k)\mathsf H_{k+1}}{k+1}\\={}&\frac{40176}{125 \pi }-\frac{66 \pi }{5}-\frac{5184 \log 2}{25 \pi }-\frac{3888 \log 3}{25 \pi }\\{}&+\frac{288 \log ^22}{5 \pi }+\frac{432(\log 2)( \log 3)}{5 \pi }+\frac{162 \log ^23}{5 \pi }.\end{split}
\end{align}}\end{corollary}
\begin{proof}All these are  consequences of \begin{align}\left.\frac{\partial^2}{\partial\sigma^2}\right|_{\sigma=0}
\int_0^1 P_{\nu}(1-2t)(1-t)^\sigma\D t=\left.\frac{\partial^2}{\partial\sigma^2}\right|_{\sigma=0}\frac{[\Gamma(1+\sigma)]^2}{\Gamma(2+\nu+\sigma)\Gamma(1-\nu+\sigma)}
\end{align}and integrations over the formulae in Theorem  \ref{thm:LegendreSun}(a). \end{proof}\begin{remark}One can of course extend the methods above to Theorem  \ref{thm:LegendreSun}(b) and also consider higher-order derivatives in  $ \sigma$. However, the resulting patterns are less systematic than Corollary \ref{cor:binomLS}, so we omit the details.\eor\end{remark}

\section{Clausen couplings  and Ramanujan--Sun series\label{sec:ClausenCouplingSun}}In response to a question   raised by Sun on MathOverflow, Kam Cheong Au posted on  \url{https://mathoverflow.net/questions/436205/}
a sketched proof of \cite[Conjecture 29]{Sun2022}. In this section, we will give a detailed exposition of  Au's answer. To facilitate the statement of Au's results in the theorem below, we recall the Legendre--Ramanujan function $ R_\nu$ from \cite[(2.2.11)]{AGF_PartI}: \begin{align}\begin{split}
R_{\nu}( \xi):={}&\frac{1-\xi^{2}}{P_{\nu}(\xi)}\frac{\D P_{\nu}(\xi)}{\D \xi}+\frac{1-\xi^{2}}{P_{\nu}(-\xi)}\frac{\D P_{\nu}(-\xi)}{\D \xi}\\{}&-\frac{\sin(\nu\pi)}{\pi}\left\{\frac{1}{ [P_{\nu}(\xi)]^{2}\I \frac{iP_{\nu}(-\xi)}{P_{\nu}(\xi)}}-\frac{1}{[P_{\nu}(-\xi)]^{2}\I \frac{iP_{\nu}(\xi)}{P_{\nu}(-\xi)}}\right\},\label{eq:R_nu_defn}\end{split}
\end{align}which is well-defined for  $\xi\in(\mathbb C\smallsetminus\mathbb R)\cup(-1,1)$, $P_\nu(\xi)\neq0$, and  $P_\nu(-\xi)\neq0$. If $ \nu\in\left\{-\frac12,-\frac13,-\frac14,-\frac16\right\}$, then the analytic expression in \eqref{eq:R_nu_defn} extends continuously to all $ \xi\in\mathbb C$,  being smooth across the branch cut of $ P_\nu(\xi),\xi\in\mathbb C\smallsetminus(-\infty,-1]$:\begin{align}
\begin{cases}R_\nu(1)=-R_\nu(-1):=\lim_{\xi\to1}R_\nu(\xi)=0; &  \\
R_\nu(\xi)=-R_\nu(-\xi):= R_\nu(\xi+i0^+)\equiv R_\nu(\xi-i0^+), & \xi>1.\ \\
\end{cases} \label{eq:R_nu_defn_ext}
\end{align}\begin{theorem}[Ramanujan--Sun series]\label{thm:RamanujanSun}For  complex numbers $a,b,t$ satisfying $ 0<|4t(1-t)|\leq1$, $\R t\leq\frac{1}{2}$, and a  sequence $ (c_k)_{k\in\mathbb Z_{\geq0}}$ with at most $ O(\log k)$ growth rate, define the convergent series {\allowdisplaybreaks\begin{align}\mathsf\Sigma_{-1/2}\left(a,b;c_{k};X_{-1/2}(t)=\frac{t(1-t)}{2^4}\right)\colonequals{}&\sum_{k=0}^\infty\binom{2k}k^3(ak+b)c_{k}[X_{-1/2}(t)]^k,\label{eq:SigmaG2}
\\
\mathsf\Sigma_{-1/3}\left(a,b;c_{k};X_{-1/3}(t)=\frac{t(1-t)}{3^3}\right)\colonequals{}&\sum_{k=0}^\infty\binom{2k}k^2\binom{3k}k(ak+b)c_{k}[X_{-1/3}(t)]^k,\label{eq:SigmaG3}\\
\mathsf\Sigma_{-1/4}\left(a,b;c_{k};X_{-1/4}(t)=\frac{t(1-t)}{2^{6}}\right)\colonequals{}&\sum_{k=0}^\infty\binom{2k}k^2\binom{4k}{2k}(ak+b)c_{k}[X_{-1/4}(t)]^k,\label{eq:SigmaG4}\\
\mathsf\Sigma_{-1/6}\left(a,b;c_{k};X_{-1/6}(t)=\frac{t(1-t)}{2^{4}3^3}\right)\colonequals{}&\sum_{k=0}^\infty\binom{2k}k\binom{3 k}{k}\binom{6 k}{3k} (ak+b)c_{k}[X_{-1/6}(t)]^k.
\label{eq:SigmaG6}\end{align}}For each  series of Ramanujan's type\begin{align}
\mathsf \Sigma_\nu (2(2t-1),-R_\nu(1-2t);1;X_\nu(t))=\frac{2\sin(\nu\pi)}{\pi\I \frac{iP_{\nu}(2t-1)}{P_{\nu}(1-2t)}},\label{eq:Rama_pi}
\end{align}there is a corresponding series of Sun's type\begin{align}
\begin{split}&
\mathsf \Sigma_\nu (2(2t-1),-R_\nu(1-2t);c_{\nu,k};X_\nu(t))+
\mathsf \Sigma_\nu (0,2(2t-1);1;X_\nu(t))\\={}& -\frac{2\sin(\nu\pi)\log X_\nu(t)}{\pi\I \frac{iP_{\nu}(2t-1)}{P_{\nu}(1-2t)}}-2i\frac{\R \frac{iP_{\nu}(2t-1)}{P_{\nu}(1-2t)}}{\I \frac{iP_{\nu}(2t-1)}{P_{\nu}(1-2t)}},\end{split}\label{eq:Sun_log}
\end{align}and yet another series in a similar spirit \begin{align}
\begin{split}
&\mathsf \Sigma_\nu (2(2t-1),-R_\nu(1-2t);c^{2}_{\nu,k}+d_{\nu,k}^{};X_\nu(t))+
\mathsf \Sigma_\nu (0,4(2t-1);d_{\nu,k};X_\nu(t))\\={}&-\frac{2\pi\left\vert \frac{iP_{\nu}(2t-1)}{P_{\nu}(1-2t)} \right\vert^{2}}{\sin(\nu\pi)\I \frac{iP_{\nu}(2t-1)}{P_{\nu}(1-2t)}}+4i\frac{\R \frac{iP_{\nu}(2t-1)}{P_{\nu}(1-2t)}}{\I \frac{iP_{\nu}(2t-1)}{P_{\nu}(1-2t)}}\log X_\nu(t)+\frac{2\sin(\nu\pi)}{\pi\I \frac{iP_{\nu}(2t-1)}{P_{\nu}(1-2t)}}\log^2X_\nu(t),\label{eq:Au_log_sqr}
\end{split}
\end{align}where\begin{align}
c_{\nu,k}\colonequals\begin{cases}6(\mathsf H_{2k}-\mathsf H_k), & \nu=-\frac{1}{2}, \\
3\mathsf H_{3k}+2\mathsf H_{2k}-5\mathsf H_k, & \nu=-\frac{1}{3}, \\
4(\mathsf H_{4k}-\mathsf H_k), & \nu=-\frac{1}{4}, \\
3(2 \mathsf H_{6 k}-\mathsf H_{3 k}-\mathsf H_k), & \nu=-\frac{1}{6}, \\
\end{cases}
\end{align}and \begin{align}
d_{\nu,k}\colonequals\begin{cases}6\left[-2\mathsf H_{2k}^{(2)}+\mathsf H_k^{(2)}
 \right], & \nu=-\frac{1}{2}, \\[5pt]
-9\mathsf H_{3k}^{(2)}-4\mathsf H_{2k}^{(2)}+5\mathsf H_k^{(2)}, & \nu=-\frac{1}{3}, \\[5pt]
4\left[-4\mathsf H_{4k}^{(2)}+\mathsf H_k^{(2)}
 \right], & \nu=-\frac{1}{4}, \\[5pt]
3\left[-12\mathsf H_{6k}^{(2)}+3\mathsf H_{3k}^{(2)}+\mathsf H_k^{(2)}
 \right], & \nu=-\frac{1}{6}. \\
\end{cases}
\end{align}\end{theorem}\begin{remark}If\begin{align}
z_{t,4\sin^2(\nu\pi)}\colonequals\lim_{\varepsilon\to0^+}\frac{P_\nu(2(t+i\varepsilon)-1)}{2i\sin(\nu \pi )P_\nu(1-2(t+i\varepsilon))}\tag{\ref{eq:z_N_lim}$'$}\label{eq:z_N_lim'}
\end{align} defines a CM\ point for a certain $ \nu\in\left\{-\frac16,-\frac14,-\frac13,-\frac12\right\}$, such that $ \I z_{t,N}>0$ and $ [\mathbb Q(z_{t,N}):\mathbb Q]=2$ for $ N=4\sin^2(\nu\pi)\in\{1,2,3,4\}$, then both  $t$ and $ R_\nu(1-2t)$ are  algebraic numbers, being explicitly solvable in radical forms \cite[Proposition 2.2.2]{AGF_PartI}. Such statements about algebraicity follow from Ramanujan's elliptic function theories to alternative bases \cite[Chapter 33]{RN5} and the classical knowledge for  complex multiplication \cite[\S6]{Zagier2008Mod123}.   \eor\end{remark}

 After dealing with the analytic and modular parts of Theorem \ref{thm:RamanujanSun} in \S\S\ref{subsec:Clausen}--\ref{subsec:RamanujanSun}, we will construct more series identities in \S\ref{subsec:corClausen}, by differentiating and integrating (precursors to) Theorem \ref{thm:RamanujanSun}.

\subsection{Clausen couplings of Legendre functions\label{subsec:Clausen}}
We begin  with the study of some infinite series concerning \begin{align}s_{\nu,k}\colonequals\binom{2k}k\frac{(-\nu)_k(\nu+1)_k}{(k!)^{2}}\label{eq:s_nu_k}\end{align}as well as \begin{align}\varsigma_{\nu,k}\colonequals{}&
\psi ^{(0)}(k-\nu )+\psi ^{(0)}(k+\nu +1)-3 \psi ^{(0)}(k+1)+\psi ^{(0)}\left(k+\frac{1}{2}\right)\label{eq:varsigma_nu_k}
\end{align}and \begin{align}
\tau_{\nu,k}\colonequals{}&
(\varsigma_{\nu,k})^{2}+\psi ^{(1)}(k-\nu )+\psi ^{(1)}(k+\nu +1)-3 \psi ^{(1)}(k+1)+\psi ^{(1)}\left(k+\frac{1}{2}\right),\label{eq:tau_nu_k}
\end{align}where $ \psi^{(m)}(z)\colonequals\frac{\D ^{m+1}}{\D z^{m+1}}\log\Gamma(z)$.
In particular, we have  expressions involving binomial coefficients{\allowdisplaybreaks\begin{align}s_{\nu,k}={}&\begin{cases}\frac{\binom{2 k}{k}^3}{2^{6 k}}, & \nu=-\frac12, \\[2pt]
\frac{\binom{2 k}{k}^2 \binom{3 k}{k}}{2^{2 k} 3^{3 k}}, & \nu=-\frac13, \\[2pt]
\frac{\binom{2 k}{k}^2 \binom{4 k}{2 k}}{2^{8 k}}, & \nu=-\frac14, \\[2pt]
\frac{\binom{2 k}{k} \binom{3 k}{k} \binom{6 k}{3 k}}{2^{6 k} 3^{3 k}}, & \nu=-\frac16, \\
\end{cases}\tag{\ref{eq:s_nu_k}$'$}\label{eq:s_nu_k'}\end{align}and  harmonic numbers\begin{align}
\varsigma_{\nu,k}={}&\begin{cases}6 \left(\mathsf H_{2 k}-\mathsf H_k-\log 2\right), & \nu=-\frac12,\ \\
3\mathsf  H_{3 k}+2 \mathsf H_{2 k}-5 \mathsf H_k-2 \log 2-3 \log 3, & \nu=-\frac13, \\
4 \left(\mathsf H_{4 k}-\mathsf H_k-2 \log 2\right), & \nu=-\frac14, \\
3 \left(2 \mathsf H_{6 k}-\mathsf H_{3 k}-\mathsf H_k-2 \log 2-\log 3\right), & \nu=-\frac16, \\
\end{cases}\tag{\ref{eq:varsigma_nu_k}$'$}\label{eq:vargsigma_nu_k'}\\\tau_{\nu,k}-(\varsigma_{\nu,k})^{2}={}&\begin{cases}6\left[-2\mathsf H_{2k}^{(2)}+\mathsf H_k^{(2)}
+\frac{\pi^{2}}{6} \right],
& \nu=-\frac{1}{2}, \\[4pt]
-9\mathsf H_{3k}^{(2)}-4\mathsf H_{2k}^{(2)}+5\mathsf H_k^{(2)}
+\frac{4\pi^{2}}{3} ,& \nu=-\frac{1}{3}, \\[4pt]
4\left[-4\mathsf H_{4k}^{(2)}+\mathsf H_k^{(2)}
+\frac{\pi^{2}}{2} \right], & \nu=-\frac{1}{4}, \\[4pt]
3\left[-12\mathsf H_{6k}^{(2)}+3\mathsf H_{3k}^{(2)}+\mathsf H_k^{(2)}
+\frac{4\pi^{2}}{3} \right], & \nu=-\frac{1}{6}. \\
\end{cases}\tag{\ref{eq:tau_nu_k}$'$}\label{eq:tau_nu_k'}
\end{align}}\begin{lemma}[Series representation for products of Legendre functions]\label{lm:CFZ}For $ \nu\in(-1,0)$ and $ t\in\big(0,\frac12\big]$,\footnote{One may analytically continue to $ 0<|4t(1-t)|\leq1$ and  $ \R t\leq \frac12$, as stated in Theorem \ref{thm:RamanujanSun}. In terms of $ z_{t,N}$ [see \eqref{eq:z_N_lim'}] for $ N\in\{1,2,3,4\}$, the condition $ \R t\leq \frac12$ may be rewritten as $ |z_{t,N}|\geq\frac1{\sqrt N}$ and $ |\R z_{t,N}|\leq\frac12$. } we have \begin{align}[P_{\nu}(1-2t)]^2={}&\sum_{k=0}^\infty[t(1-t)]^ks_{\nu,k}\equalscolon\sigma_\nu(4t(1-t)),\label{eq:P+P+}\end{align}\begin{align}\begin{split}
&P_{\nu}(1-2t)\left[ \frac{\pi P_{\nu}(2t-1)}{\sin(\nu\pi)}+P_\nu(1-2t)\log \frac{1}{4t(1-t)} \right]\\={}&\sum_{k=0}^\infty[t(1-t)]^ks_{\nu,k}\varsigma_{\nu,k}\equalscolon\varsigma_\nu(4t(1-t)),\end{split}\label{eq:P+P-}
\end{align}and\begin{align}\begin{split}&
\left[ \frac{\pi P_{\nu}(1-2t)}{\sin(\nu\pi)} \right]^{2}+\left[ \frac{\pi P_{\nu}(2t-1)}{\sin(\nu\pi)}+P_\nu(1-2t)\log \frac{1}{4t(1-t)} \right]^2\\={}&\sum_{k=0}^\infty[t(1-t)]^ks_{\nu,k}\tau_{\nu,k}\equalscolon\tau_\nu(4t(1-t)).\label{eq:P+P+_+_P-P-}\end{split}
\end{align}\end{lemma}
\begin{proof}Consider the following quadratic transformation  \cite[(3.1.3)]{AAR} of hypergeometric functions for  $ t\in\big(0,\frac12\big]$: \begin{align}
{_2}F_1\left(\left. \begin{array}{@{}c@{}}
-\nu+\varepsilon,\nu+1+\varepsilon\ \\
1+\varepsilon\ \\
\end{array} \right| t\right)={_2}F_1\left(\left. \begin{array}{@{}c@{}}
\frac{-\nu+\varepsilon}{2},\frac{\nu+1+\varepsilon}{2}\ \\
1+\varepsilon\ \\
\end{array} \right| 4t(1-t)\right),
\end{align}where $ \varepsilon$ is a free parameter for the moment. Combining this with Clausen's identity for squares of $_2F_1$ \cite[Exercise 2.13 or 3.17(d)]{AAR}, we have \begin{align}
\left[{_2}F_1\left(\left. \begin{array}{@{}c@{}}
-\nu+\varepsilon,\nu+1+\varepsilon\ \\
1+\varepsilon\ \\
\end{array} \right| t\right)\right]^2={_3}F_2\left(\left. \begin{array}{@{}c@{}}
-\nu+\varepsilon,\nu+1+\varepsilon,\frac{1}{2}+\varepsilon\ \\
1+\varepsilon,1+2\varepsilon\ \\
\end{array} \right| 4t(1-t)\right)\label{eq:Clausen}
\end{align}for $ 2\varepsilon\notin\mathbb Z_{<0}$.  Setting $ \varepsilon=0$ in \eqref{eq:Clausen}, we arrive at \eqref{eq:P+P+}.

Thanks to the Frobenius--Zagier process (Lemma \ref{lm:FB}), we can differentiate  \eqref{eq:Clausen} with respect to  $\varepsilon $ at  $ \varepsilon=0$, which verifies  \eqref{eq:P+P-}.

In fact,  using  Euler's  hypergeometric transformation \cite[(2.2.7)]{AAR} \begin{align}
\begin{split}&{_2}F_1\left(\left. \begin{array}{@{}c@{}}
-\nu+\varepsilon,\nu+1+\varepsilon\ \\
1+\varepsilon\ \\
\end{array} \right| t\right)[t(1-t)]^{\varepsilon/2}\\={}&{_2}F_1\left(\left. \begin{array}{@{}c@{}}
-\nu,\nu+1\ \\
1+\varepsilon\ \\
\end{array} \right| t\right)\left(\frac{t}{1-t}\right)^{\varepsilon/2}\equalscolon\Gamma(1+\varepsilon)P^{-\varepsilon}_\nu(1-2t),\end{split}
\label{eq:EulerPnumu}\end{align}one may rework the derivations of  \eqref{eq:P+P+}--\eqref{eq:P+P-} by differentiating the associated Legendre function $ P_\nu^\mu(x),x\in(-1,1)$ with respect to its order $\mu$ \cite{Brychkov2010Pnumu}.

Furthermore, for $ \varepsilon\notin\mathbb Z$,  one can deduce \begin{align}
\frac{\varepsilon\pi}{\sin(\varepsilon\pi)}P^{-\varepsilon}_\nu(1-2t)P^{\varepsilon}_\nu(1-2t)={_3}F_2\left(\left. \begin{array}{@{}c@{}}
-\nu,\nu+1,\frac{1}{2}\ \\1+\varepsilon,1-\varepsilon\ \\
\end{array} \right| 4t(1-t)\right)\label{eq:Clausen'}
\end{align} from a companion to Clausen's identity \cite[Exercise 2.13]{AAR}. As we reexpress\begin{align}
\lim_{\varepsilon\to0}\frac{\partial^{2}}{\partial\varepsilon^{2}}\left\{ \frac{\left[\Gamma(1+\varepsilon)P^{-\varepsilon}_\nu(1-2t)\right]^{2}}{[t(1-t)]^{\varepsilon}}-\frac{\varepsilon\pi}{\sin(\varepsilon\pi)}P^{-\varepsilon}_\nu(1-2t)P^{\varepsilon}_\nu(1-2t)\right\}
\end{align}through termwise derivatives of the generalized hypergeometric series [cf.\ \eqref{eq:Clausen} and \eqref{eq:Clausen'}]\begin{align}
\left.\frac{\partial^{2}}{\partial\varepsilon^{2}}\right|_{\varepsilon=0}\left[ {_3}F_2\left(\left. \begin{array}{@{}c@{}}
-\nu+\varepsilon,\nu+1+\varepsilon,\frac{1}{2}+\varepsilon\ \\
1+\varepsilon,1+2\varepsilon\ \\
\end{array} \right| 4t(1-t)\right)- {_3}F_2\left(\left. \begin{array}{@{}c@{}}
-\nu,\nu+1,\frac{1}{2}\ \\
1+\varepsilon,1-\varepsilon\ \\
\end{array} \right| 4t(1-t)\right)\right],
\end{align} we get \eqref{eq:P+P+_+_P-P-}.
\end{proof}\begin{remark}Equivalent forms of \eqref{eq:P+P+}--\eqref{eq:P+P+_+_P-P-} have appeared in Kam Cheong Au's answer (posted under the username ``{pisco}'') to Zhi-Wei Sun's question on \url{https://mathoverflow.net/questions/436205/} about \cite[Conjecture 29]{Sun2022}. Our proof here,  being slightly different from Au's  argument, follows the presentations of Chudnovsky--Chudnovsky
\cite[(1.3)]{ChudnovskyChudnovsky1988} and Chan--Wan--Zudilin \cite[\S2]{ChanWanZudilin2013} more closely.
\eor\end{remark}\begin{remark}For $ \nu=-\frac12$, both \eqref{eq:P+P+} and \eqref{eq:P+P-}  were derived by Watson \cite{Watson1939}, who also computed derivatives  with respect to  a free parameter $\varepsilon$. \eor\end{remark}\begin{remark}Clausen's identity has a natural analog over finite fields \cite{EvansGreene2009}, which may be helpful in future investigations of Sun's supercongruences in \cite[Conjecture 29]{Sun2022}.\eor\end{remark}
\subsection{Ramanujan--Sun series\label{subsec:RamanujanSun}}
Exploiting the Wro\'nskian determinant      \begin{align} W[P_\nu(x),P_\nu(-x)]\colonequals P_\nu(x)\frac{\D P_\nu(-x)}{\D x}-P_\nu(-x)\frac{\D P_\nu(x)}{\D x}=-\frac{2\sin(\nu\pi)}{\pi(1-x^{2})},\label{eq:W_Pnu}\end{align}  one can compute\begin{align}
\frac{\pi\I \frac{iP_{\nu}(-\xi)}{P_{\nu}(\xi)}}{2\sin(\nu\pi)}\left[\big(1-\xi^{2}\big)\frac{\D }{\D \xi}-R_\nu(\xi)\right]\begin{pmatrix}[P_\nu(\xi)]^2 \\
P_\nu(\xi)P_\nu(-\xi) \\
[P_\nu(-\xi)]^2 \\
\end{pmatrix}=\begin{pmatrix}1 \\
-i\R \frac{iP_{\nu}(-\xi)}{P_{\nu}(\xi)} \\[5pt]
-\left\vert \frac{iP_{\nu}(-\xi)}{P_{\nu}(\xi)} \right\vert^{2} \\
\end{pmatrix}\label{eq:Rnu_canc}
\end{align} for $  \nu\in(-1,0)$, where the function $ R_\nu(\xi)$ is defined in \eqref{eq:R_nu_defn} for $ \xi\in(\mathbb C\smallsetminus\mathbb R)\cup(-1,1)$, $P_\nu(\xi)\neq0$,  and $P_\nu(-\xi)\neq0$, with no restrictions on $\nu$. As pointed out in Au's aforementioned post, the first (resp.\ second) entry on the right-hand side of the equation above is responsible for Ramanujan's celebrated series representations of $ \frac1\pi$ \cite[Chapter 14]{Cooper2017Theta} (resp.\ Sun's infinite series in \cite[Conjecture 29]{Sun2022}). In particular, with the notations in Theorem \ref{thm:RamanujanSun}, we have \begin{align}
\mathsf \Sigma_\nu(a,b;c_k;X_{\nu}(t))=\sum_{k=0}^\infty [t(1-t)]^k(ak+b)c_k s_{\nu,k},
\end{align}where $ s_{\nu ,k}$ is given in \eqref{eq:s_nu_k'} for $ \nu\in\left\{-\frac12,-\frac13,-\frac14,-\frac16\right\}$, so\begin{align}
\mathsf \Sigma_\nu\left(-2\xi,-R_{\nu}(\xi);1;X_{\nu}\left(\frac{1-\xi}{2}\right)\right)=\left[\big(1-\xi^{2}\big)\frac{\D }{\D \xi}-R_\nu(\xi)\right][P_\nu(\xi)]^2=\frac{2\sin(\nu\pi)}{\pi\I \frac{iP_{\nu}(-\xi)}{P_{\nu}(\xi)}}
\end{align}explains \eqref{eq:Rama_pi}, an identity of Ramanujan's type, while{\allowdisplaybreaks\begin{align}
\begin{split}&
\mathsf \Sigma_\nu\left(-2\xi,-R_{\nu}(\xi);\varsigma_{\nu,k};X_{\nu}\left(\frac{1-\xi}{2}\right)\right)+\mathsf \Sigma_\nu\left(0,-2\xi;1;X_{\nu}\left(\frac{1-\xi}{2}\right)\right)\\={}&\left[\big(1-\xi^{2}\big)\frac{\D }{\D \xi}-R_\nu(\xi)\right]\left\{ P_{\nu}(\xi)\left[ \frac{\pi P_{\nu}(-\xi)}{\sin(\nu\pi)}+P_\nu(\xi)\log \frac{1}{1-\xi^{2}} \right] \right\}\\{}&-[P_\nu(\xi)]^2\big(1-\xi^{2}\big)\frac{\D }{\D \xi}\log \frac{1}{1-\xi^{2}}\\={}&-\frac{2\sin(\nu\pi)\log\big(1-\xi^2\big)}{\pi\I \frac{iP_{\nu}(-\xi)}{P_{\nu}(\xi)}}-2i\frac{\R \frac{iP_{\nu}(-\xi)}{P_{\nu}(\xi)}}{\I \frac{iP_{\nu}(-\xi)}{P_{\nu}(\xi)}}\end{split}
\end{align}}accounts for \eqref{eq:Sun_log}, an identity of Sun's type.
As hinted in Au's post, one may also take this procedure one step further, to produce  series identities that are not found in either Ramanujan's or Sun's works. Concretely speaking, after   spelling out a linear combination of the three rows in  \eqref{eq:Rnu_canc} as \begin{align}
\begin{split}&\mathsf \Sigma_\nu\left(-2\xi,-R_{\nu}(\xi);\tau_{\nu,k};X_{\nu}\left(\frac{1-\xi}{2}\right)\right)+\mathsf \Sigma_\nu\left(0,-4\xi;\varsigma_{\nu,k};X_{\nu}\left(\frac{1-\xi}{2}\right)\right)\\={}&
\left[\big(1-\xi^{2}\big)\frac{\D }{\D \xi}-R_\nu(\xi)\right]\left\{ \left[ \frac{\pi P_{\nu}(\xi)}{\sin(\nu\pi)} \right]^{2}+\left[ \frac{\pi P_{\nu}(-\xi)}{\sin(\nu\pi)}+P_\nu(\xi)\log \frac{1}{1-\xi^{2}} \right]^2 \right\}\\{}&-2P_{\nu}(\xi)\left[ \frac{\pi P_{\nu}(-\xi)}{\sin(\nu\pi)}+P_\nu(\xi)\log \frac{1}{1-\xi^{2}} \right]\big(1-\xi^{2}\big)\frac{\D }{\D \xi}\log \frac{1}{1-\xi^{2}}
\\={}&\frac{2\pi\left( 1- \left\vert \frac{iP_{\nu}(-\xi)}{P_{\nu}(\xi)} \right\vert^{2}\right)}{\sin(\nu\pi)\I \frac{iP_{\nu}(-\xi)}{P_{\nu}(\xi)}}+4i\frac{\R \frac{iP_{\nu}(-\xi)}{P_{\nu}(\xi)}}{\I \frac{iP_{\nu}(-\xi)}{P_{\nu}(\xi)}}\log\big(1-\xi^2\big)+\frac{2\sin(\nu\pi)}{\pi\I \frac{iP_{\nu}(-\xi)}{P_{\nu}(\xi)}}\log^2\big(1-\xi^2\big),\end{split}
\end{align}and getting rid of some trailing terms with the help from \eqref{eq:Rama_pi} and \eqref{eq:Sun_log}, we arrive at \eqref{eq:Au_log_sqr}.

Furthermore,  after careful analytic continuation, one may extend the Ramanujan--Sun series in  Theorem \ref{thm:RamanujanSun}  to cases where $ t\in\big(\frac{1-\sqrt{2}}{2},0\big)$. For example, a record-breaking \cite{Preston1992} formula of Chudnovsky--Chudnovsky  \cite[(1.5)]{ChudnovskyChudnovsky1988}\begin{align}\begin{split}
\mathsf\Sigma_{-1/6}\left(a_{163},b_{163};1;X_{163}\right)\colonequals{}&\sum_{k=0}^\infty\binom{2k}k\binom{3 k}{k}\binom{6 k}{3k} \left(a_{163}k+b_{163}\right)(X_{163})^k\\={}&\frac{2^7 \sqrt{3\cdot5^3 \cdot23^3 \cdot29^3}}{\pi }\end{split}
\end{align} for \begin{align}
\left\{ \begin{array}{@{}r@{{}\colonequals{}}l}a_{163}&\smash[t]{2\cdot3^2\cdot 7\cdot11\cdot19\cdot127\cdot163,}
\\b_{163}&13\cdot 1045493,\\X_{163}&- \smash[b]{\dfrac{1}{2^{18} \cdot3^3 \cdot5^3 \cdot23^3 \cdot29^3}}
\end{array} \right.
\end{align}(resp.\ its counterpart\begin{align}\begin{split}{}&
\mathsf\Sigma_{-1/6}\left(a_{163},b_{163};3(2 \mathsf H_{6 k}-\mathsf H_{3 k}-\mathsf H_k)^{};X_{163}\right)+\mathsf\Sigma_{-1/6}\left(0,a_{163};1;X_{163}\right)\\\colonequals{}&\sum_{k=0}^\infty\binom{2k}k\binom{3 k}{k}\binom{6 k}{3k}\left[3\left(a_{163}k+b_{163}\right)(2 \mathsf H_{6 k}-\mathsf H_{3 k}-\mathsf H_k)+a_{163}\right](X_{163})^k\\={}&-\frac{2^7 \sqrt{3\cdot5^3 \cdot23^3 \cdot29^3}}{\pi }\log| X_{163}|\end{split}
\end{align}in Sun's context  \cite[Conjecture 29(iv)]{Sun2022}) follows from an application of the first (resp.\ second) entry on the right-hand side of \eqref{eq:Rnu_canc} to an analytic continuation of  \eqref{eq:P+P+} [resp.\  \eqref{eq:P+P-}] for a particular choice of parameters\begin{align}
\nu=-\frac{1}{6},\quad t=\frac{1-\xi_{163}}{2}=\frac{1}{2}-\frac{1}{2}\sqrt{\frac{3^3\cdot 7^2\cdot 11^2\cdot 19^2\cdot 127^2 \cdot163}{2^{12} \cdot 5^3 \cdot 23^3 \cdot 29^3}}\label{eq:163}
\end{align} that is associated with a modular invariant \begin{align} X_{163}=\frac{1-\xi_{163}^2}{1728}=\dfrac{1}{j\left(\frac{1+i\sqrt{163}}{2}\right)}.\end{align} Here, to compute \begin{align} R_{-1/6}(\xi_{163})=\frac{13\cdot 1045493}{2^6\sqrt{3\cdot 5^{3}\cdot 23^{3}\cdot 29^{3}\cdot 163}}\end{align} for the specific $\xi_{163}\in(1,\infty)$ given  in \eqref{eq:163}, one may set $ z=\frac{1+i\sqrt{163}}{2}$ in \cite[(2.2.20)]{AGF_PartI}\begin{align}
R_{-1/6}\left( \sqrt{\frac{j(z)-1728}{j(z)}}
\right)\sqrt{\frac{j(z)-1728}{j(z)}}={}&-\frac{E_{6}(z)}{3[E_{4}(z)]^{2}}\left[ E_2(z)-\frac{ E_6(z)}{E_4(z)} \right]
\end{align} for \begin{align}
\I z>0,\quad |\R z|\leq\frac{1}{2},\quad |z|\geq1,
\end{align}where the  Eisenstein series are given by \begin{align}
\begin{cases}E_2(z)\colonequals\displaystyle \smash[t]{1-{\frac{3}{\pi\I z}-24\sum_{n=1}^\infty\frac{ne^{2\pi inz}}{1-e^{2\pi inz}} }},\\
E_4(z)\colonequals \displaystyle 1+240\nsum\limits_{n=1}^\infty\dfrac{n^3e^{2\pi inz}}{1-e^{2\pi inz}},\\
E_6(z)\colonequals \displaystyle \smash[b]{1-{504\nsum\limits_{n=1}^\infty\dfrac{n^5e^{2\pi inz}}{1-e^{2\pi inz}}}}.\\
\end{cases}\label{eq:E2E4E6}\\[-12pt]\notag
\end{align}Using the same arithmetic data given above, one can also establish \begin{align}
\begin{split}
&\mathsf\Sigma_{-1/6}\left(a_{163},b_{163};c_{-1/6,k}^{2}+d_{-1/6,k}^{}=9(2 \mathsf H_{6 k}-\mathsf H_{3 k}-\mathsf H_k)^{2}+3\left[-12\mathsf H_{6k}^{(2)}+3\mathsf H_{3k}^{(2)}+\mathsf H_k^{(2)}
 \right];X_{163}\right)\\{}&+\mathsf\Sigma_{-1/6}\left(0,2a_{163};3(2 \mathsf H_{6 k}-\mathsf H_{3 k}-\mathsf H_k);X_{163}\right)\\\colonequals{}&\sum_{k=0}^\infty\binom{2k}k\binom{3 k}{k}\binom{6 k}{3k}\left[\left(a_{163}k+b_{163}\right)\left( c_{-1/6,k}^{2}+d_{-1/6,k}^{} \right)+6a_{163}(2 \mathsf H_{6 k}-\mathsf H_{3 k}-\mathsf H_k)\right](X_{163})^k\\={}&\frac{2^7 \sqrt{3\cdot5^3 \cdot23^3 \cdot29^3}}{\pi }\left( 163\pi^2-\log^2|X_{163}|\right)
\end{split}
\end{align}via an analytic continuation of \eqref{eq:Au_log_sqr}. It is a well-known fact  that $ \log|X_{163}|$ is very close to $ -\sqrt{163}\pi$, so the series above evaluates to a fairly small number.

Likewise, for $ \nu\in\left\{-\frac14,-\frac13,-\frac12\right\}$ and $ N=4\sin^2(\nu\pi)\in\{2,3,4\}$, the following relations \cite[(2.2.21)]{AGF_PartI}
\begin{align}
R_{\nu}(1-2\alpha_N(z))={}&-\frac{N-1}{6}\frac{N^{2}[E_{2}(Nz)]^2-N^{2}E_{4}(N z)-[E_{2}(z)]^2+E_{4}(z)}{[NE_{2}(Nz)-E_{2}(z)]^{2}}\label{eq:LegendreRamanujan_E}
\end{align}will be useful in the analytic continuation of the corresponding Ramanujan--Sun series.
\subsection{Some series descending from Clausen couplings\label{subsec:corClausen}}
Before closing this section, we mention some  simple consequences of Lemma \ref{lm:CFZ} with regard to binomial harmonic sums that lie outside the scope of Theorem \ref{thm:RamanujanSun}.
\subsubsection{Derivatives and differential equations for products of  two Legendre functions}
One may construct Table  \ref{tab:central_deriv_FZ} by evaluating \eqref{eq:Pnu0deriv} with   \eqref{eq:P+P+}--\eqref{eq:P+P-}. The bulkier  identities associated with \eqref{eq:P+P+_+_P-P-} are omitted from this table.

 \begin{table}[t]\caption{Selected series involving $ \binom{2k}k^3$, where $ h_k^{(r)}\colonequals\mathsf H_{2k}^{(r)}-\frac{1}{2^{r}}\mathsf H_k^{(r)}$, $ \lambda\colonequals\log2$,  $\zeta(s)\colonequals\sum_{k=1}^\infty\frac{1}{k^s} $, $ \beta(s)\colonequals\sum_{k=0}^\infty\frac{(-1)^k}{(2k+1)^{s}}$, and $ G\colonequals\beta(2)$ \label{tab:central_deriv_FZ}}

\begin{scriptsize}\begin{align*}\begin{array}{r@{}l}\hline\hline
\displaystyle \vphantom{\frac{\frac{\frac11}{1}}{1}}\sum_{k=0}^\infty h_k^{(2)} \frac{\binom{2k}k^3}{2^{6k}}={}&\displaystyle-\frac{\big[\Gamma\big(\frac{1}{4}\big)\big]^4}{4\pi^{3}}\left(G-\frac{\pi^2}8\right)  \\
\displaystyle \sum_{k=0}^\infty\left\{  h_k^{(4)}-\left[ h_k^{(2)}\right]^{2}\right\} \frac{\binom{2k}k^3}{2^{6k}}={}&\displaystyle -\frac{\big[\Gamma\big(\frac{1}{4}\big)\big]^4}{4\pi^{3}}\left[ \beta(4)+\left(G-\frac{\pi^2}8\right)^{2}-\frac{\pi^{4}}{96} \right] \\\displaystyle \sum_{k=0}^\infty \left\{ h_k^{(6)} -\frac{3h_k^{(4)}-\left[ h_k^{(2)}\right]^{2}}{2}h_k^{(2)}\right\} \frac{\binom{2k}k^3}{2^{6k}}={}&\displaystyle -\frac{\big[\Gamma\big(\frac{1}{4}\big)\big]^4}{4\pi^{3}} \left\{ \beta (6)+\frac{3}{2} \left[ \beta (4)-\frac{\pi^{4}}{96}\right] \left(G-\frac{\pi^2}8\right)\vphantom{\big(^2}\right.\\{}&\displaystyle\left.{}+\frac{1}{2} \left(G-\frac{\pi^2}8\right)^3 -\frac{\pi ^6}{960}\right\}    \\\hline \displaystyle \vphantom{\frac{\frac{\frac11}{1}}{1}}\sum_{k=0}^\infty\left[ h_k^{(3)} -\frac{3}{2}\big( \mathsf H_{2k}-\mathsf H_k-\lambda \big)h_k^{(2)}\right] \frac{\binom{2k}k^3}{2^{6k}}={}&\displaystyle \frac{\big[\Gamma\big(\frac{1}{4}\big)\big]^4}{32\pi^{3}}[7\zeta (3)-2\pi G]  \\\displaystyle \sum_{k=0}^\infty \left\{ h_{k} ^{(5)}-h_k^{(2)} h_k^{(3)}+3\big( \mathsf H_{2k}-\mathsf H_k-\lambda \big)\frac{\left[ h_k^{(2)}\right]^{2}-h_k^{(4)}}{4}\right\}\frac{\binom{2k}k^3}{2^{6k}}={}&\displaystyle \frac{\big[\Gamma\big(\frac{1}{4}\big)\big]^4}{32\pi^{3}}\left\{ \frac{31\zeta(5)}{4}-\pi\big[\beta(4)+G^2\big]+7\zeta(3)\left(G-\frac{\pi^2}8\right)\right\}\\\hline\hline
\end{array}\end{align*}\end{scriptsize}
\end{table}

There is a class of infinite series similar to  \eqref{eq:P+P+} that can be evaluated in closed form, despite its absence from Sun's recent conjectures  \cite{Sun2022}. Consider the following third-order differential equation \cite[(16)]{Zhou2013Pnu}
 \begin{align}
\frac{\D}{\D \xi}\left\{\big(1-\xi^2 \big)\frac{\D}{\D \xi}\left[ \big(1-\xi^2 \big)\frac{\D f(\xi)}{\D \xi}\right]+4\nu(\nu+1)\big(1-\xi^2 \big) f(\xi)\right\}+4\nu(\nu+1)
\xi f(\xi)=0\label{eq:LegendreSqrDiff_nu}
\end{align}satisfied by $ f(\xi)=[P_\nu(\xi)]^2$, which allows us to find the anti-derivative of $-\frac{\xi}{2}[P_\nu(\xi)]^2 $ in the context of \eqref{eq:P+P+}, namely\begin{align}\begin{split}&
\frac{1-\xi^2 }{4\nu(\nu+1)}\frac{\D}{\D \xi}\left[ \big(1-\xi^2 \big)P_\nu(\xi)\frac{\D P_\nu(\xi)}{\D \xi}\right]+\frac{1-\xi^2 }{2}[P_\nu(\xi)]^2\\={}&\frac{\big(1-\xi^2\big)^{2}}{4\nu(\nu+1)}\left[ \frac{\D P_\nu(\xi)}{\D \xi} \right]^{2}+\frac{1-\xi^2 }{4}[P_\nu(\xi)]^2\\={}&\sum_{k=0}^\infty\binom{2k}k\frac{(-\nu)_k(\nu+1)_k}{k!(k+1)!}\left(\frac{1-\xi^2}{4}\right)^{k+1}=\frac{1-\xi^2}{4}{_3F_2}\left( \left.\begin{array}{@{}c@{}}
\frac12,-\nu,\nu+1\ \\
1,2\\
\end{array}\right|1 -\xi^{2}\right)\end{split}
\end{align}for $ \nu\in(-1,0)$, $ \R \xi>0$, and  $ \big|1-\xi^2\big|<1$. In the $ \xi\to 0$ limit, this brings us (cf.\ \cite[(2.10)]{CampbellCantariniAurizio2022} for the case where $\nu=-1/2$) \begin{align}
_3F_2\left( \left.\begin{array}{@{}c@{}}
\frac12,-\nu,\nu+1\ \\[4pt]
1,2\\
\end{array}\right|1 \right)=\sum_{k=0}^\infty\frac{\binom{2k}k}{2^{2k}}\frac{(-\nu)_k(\nu+1)_k}{k!(k+1)!}=\frac{  \nu  (\nu +1)\pi}{4\big[ \Gamma \big(\frac{2-\nu }{2}\big) \Gamma \big(\frac{\nu +3}{2}\big)\big]^2}+\frac{\pi }{\big[\Gamma \big(\frac{1-\nu }{2}\big) \Gamma \big(\frac{\nu+2 }{2}\big)\big]^2}
\end{align}for $ \nu\in(-1,0)$.

\subsubsection{Integrals over products of  two Legendre functions}

Thanks to the closed-form evaluations (cf.\ \cite[items 7.113.1, 7.113.3, and 7.112.3]{GradshteynRyzhik}){\allowdisplaybreaks\begin{align}
\int_{0}^1[P_\nu(\xi)]^2\D\xi={}&\frac{1}{2 \nu +1}\left\{1+\frac{\sin (\nu\pi   )}{\pi} \left[ \psi ^{(0)}\left(\frac{\nu+2 }{2}\right)-\psi ^{(0)}\left(\frac{\nu+1 }{2}\right )\right]\right\},\\\int_{0}^1P_\nu(\xi)P_{\nu}(-\xi)\D\xi={}&\frac{\cos (\nu \pi  )}{2 \nu +1},\\\int_{-1}^1[P_\nu(\xi)]^2\D\xi={}&\frac{2}{2\nu+1}\left[1-\frac{2\sin^2(\nu\pi)}{\pi^{2}}\psi^{(1)}(\nu+1)\right],
\end{align}}it is possible to produce some series identities by manipulating \eqref{eq:P+P+}--\eqref{eq:P+P+_+_P-P-} in a similar fashion
as Corollaries \ref{cor:binomLS}--\ref{cor:HkLS}, as shown below.\begin{corollary}[Integral couplings of two Legendre functions]\label{cor:Int2Pnu}We have \begin{align}\begin{split}{}&
2\int_{0}^{1/2}\sigma_\nu(4t(1-t))\D t=\sum_{k=0}^\infty\frac{(-\nu)_k(\nu+1)_k}{(2k+1)(k!)^{2}}\\={}&\begin{cases}\frac{1}{2 \nu +1}\left\{1+\frac{\sin (\nu\pi   )}{\pi} \left[ \psi ^{(0)}\big(\frac{\nu+2 }{2}\big)-\psi ^{(0)}\big(\frac{\nu+1 }{2}\big)\right]\right\},&\nu\neq-\frac12,\\\frac{4G}{\pi},&\nu=-\frac12,\end{cases}\end{split}\end{align}\begin{align}\begin{split}
&2\int_{0}^{1/2}\left[\varsigma_\nu(4t(1-t))-\sigma_\nu(4t(1-t))\log\frac{1}{4t(1-t)}\right]\D t\\={}&\sum_{k=0}^\infty\frac{(-\nu)_k(\nu+1)_k}{(2k+1)(k!)^{2}}\left[ 2\gamma_{0}-2\mathsf H_k-\frac{2}{2k+1}+\psi ^{(0)}(k-\nu )+\psi ^{(0)}(k+\nu +1)\right]\\={}&\begin{cases}\frac{\pi\cot(\nu\pi)}{2\nu+1}, & \nu\neq-\frac12, \\
-\frac{\pi^{2}}{2} ,& \nu=-\frac12, \\
\end{cases}\end{split}\end{align}{\allowdisplaybreaks\begin{align}\begin{split}&
2\int_{0}^{1/2}\left[\tau_\nu(4t(1-t))-2\varsigma_\nu(4t(1-t))\log\frac{1}{4t(1-t)}+\sigma_\nu(4t(1-t))\log^{2}\frac{1}{4t(1-t)}\right]\D t\\={}&\sum_{k=0}^\infty\frac{(-\nu)_k(\nu+1)_k}{(2k+1)(k!)^{2}}\left\{ \tau_{\nu,k} -4\left(\mathsf H_{2k}-\mathsf H_k+\frac{1}{2k+1}-\log 2 \right)\varsigma_{\nu,k}\right.\\{}&\left.{}+4\left(\mathsf H_{2k}-\mathsf H_k+\frac{1}{2k+1}-\log 2 \right)^{2}+4\mathsf H_{2k}^{(2)}-2\mathsf H_k^{(2)}+\frac{4}{(2k+1)^{2}}-\frac{\pi^{2}}{3}\right\}\\={}&\begin{cases}\frac{2}{2\nu+1} \left[\frac{\pi ^2}{\sin ^2(\nu\pi   )}-2 \psi ^{(1)}(\nu +1)\right],& \nu\neq-\frac12, \\
28\zeta(3),\ & \nu=-\frac12. \\
\end{cases}\end{split}
\end{align}}Here, for $ \nu=-1/2$, the corresponding infinite  series can be evaluated by independent methods \cite[Corollary 3.16]{Zhou2022mkMpl}.\qed\end{corollary}

In  \S\ref{sec:CG_coupling}, we will encounter more sophisticated corollaries to Clausen couplings.

\section{Green--Epstein variations on Sun's series\label{sec:GreenSun}}The automorphic Green's function (cf.\  \cite[p.~207]{GrossZagier1985}, \cite[pp.~238--239]{GrossZagierI} and \cite[p.~544]{GrossZagierII}) of weight $k\in2\mathbb Z_{>1}$ and level $N\in\mathbb Z_{>0}$ is defined by\begin{align}
G_{k/2}^{\mathfrak H/\overline{\varGamma}_0(N)}(z_1,z_2)\colonequals{}&-\sum_{\substack{a,b,c,d\in\mathbb Z\\ ad-Nbc=1}}Q_{\frac{k}{2}-1}
\left( 1+\frac{\left\vert z_{1} -\frac{a z_2+b}{Ncz_{2}+d}\right\vert ^{2}}{2\I z_1\I\frac{a z_2+b}{Ncz_{2}+d}} \right),\label{eq:wtkAGF}
\end{align}with  $ Q_{\nu}(X):=\int_0^\infty{\big(X+\sqrt{X^2-1}\cosh u\big)^{-\nu-1}}\D u$  being the Legendre function of the second kind for $ X>1,\nu>-1$.

In this section, we will construct analogs of  Sun's series that can be evaluated via automorphic Green's functions of weight $4$ and levels $N\in\{1,2,3\}$, corresponding to  $ Q_1(X)\colonequals-1+\frac{X}{2}\log\frac{X+1}{X-1},X>1$ in \eqref{eq:wtkAGF}. Here, the infinite series in \eqref{eq:wtkAGF} evaluates to a finite number if the two points $ z_1,z_2\in\mathfrak H\colonequals\{w\in\mathbb C|\I w>0\}$ satisfy $ \alpha_N(z_1)\neq\alpha_N(z_2)\in\mathbb C\cup\{\infty\}$ when $N\in\{2,3\}$ and $ j(z_1)\neq j(z_2)$ when $ N=1$, with modular invariants defined in \eqref{eq:alpha_N}--\eqref{eq:jKlein}.
\begin{theorem}[Green--Sun series]\label{thm:GreenSun}
 For    $|t|\leq 1$ and $ t\notin [-1,0]\cup\{1\}$, we have  the following integral representations of infinite series:{\allowdisplaybreaks\begin{align}\begin{split}
\mathsf S^G_{-1/3}(t)\colonequals{}&-\frac{3\sqrt{3}}{\pi}\sum_{k=1}^\infty\binom{2 k}{k}\binom{3 k}{k} \left(\frac{t}{3^{3}}\right)^k\left[ \psi ^{(0)}\left(k+\frac{2}{3}\right)-\psi ^{(0)}\left(k+\frac{1}{3}\right)-\frac{\pi }{\sqrt{3}}\right]\\={}&\int^1_{1-2t}[P_{-1/3}(-\xi)P_{-1/3}(1-2t)-P_{-1/3}(\xi)P_{-1/3}(2t-1)]P_{-1/3}(\xi)\D\xi,\end{split}\label{eq:GS3}\\\begin{split}
\mathsf S^G_{-1/4}(t)\colonequals{}&-\frac{2\sqrt{2}}{\pi}\sum_{k=1}^\infty\binom{2 k}{k}\binom{4 k}{2k} \left(\frac{t}{2^{6}}\right)^k\left[ \psi ^{(0)}\left(k+\frac{3}{4}\right)-\psi ^{(0)}\left(k+\frac{1}{4}\right) -\pi\right]\\={}&\int^1_{1-2t}[P_{-1/4}(-\xi)P_{-1/4}(1-2t)-P_{-1/4}(\xi)P_{-1/4}(2t-1)]P_{-1/4}(\xi)\D\xi,\end{split}\label{eq:GS4}\\\begin{split}
\mathsf S^G_{-1/6}(t)\colonequals{}&-\frac{3}{2\pi}\sum_{k=1}^\infty\binom{3 k}{k}\binom{6 k}{3k} \left(\frac{t}{2^{4}3^3}\right)^k\left[ \psi ^{(0)}\left(k+\frac{5}{6}\right)-\psi ^{(0)}\left(k+\frac{1}{6}\right)-\sqrt{3}\pi\right]\\={}&\int^1_{1-2t}[P_{-1/6}(-\xi)P_{-1/6}(1-2t)-P_{-1/6}(\xi)P_{-1/6}(2t-1)]P_{-1/6}(\xi)\D\xi.\end{split}\label{eq:GS6}
\end{align}}These series are related to automorphic Green's functions by the following sum rules for $ |t|<1$ and $ |1-t|<1$:\begin{align}\begin{split}&\R\left[\frac{\sin(\nu \pi )}{\I\frac{iP_{\nu}(1-2t)}{P_{\nu}(2t-1)}}\frac{\mathsf S^G_{\nu}(t)}{P_{\nu}(2t-1)}+\frac{\sin(\nu \pi )}{\I\frac{iP_{\nu}(2t-1)}{P_{\nu}(1-2t)}}\frac{\mathsf S^G_{\nu}(1-t)}{P_{\nu}(1-2t)}\right]-\frac{\sin (2\nu\pi   )}{2 \nu +1}
\\={}&\frac{1-2 \sin ^2(2\nu \pi   )}{4 (2 \nu +1)^2 \pi\cos ^2(\nu\pi   )}G^{\mathfrak H/\overline\varGamma_0(4\sin^2(\nu\pi))}_2\left( \frac{P_{\nu}(2t-1)}{2i\sin(\nu \pi )P_{\nu}(1-2t)},\frac{1-i\cot(\nu\pi)}{2} \right)\\={}&\frac{1-2 \sin ^2(2\nu \pi   )}{4 (2 \nu +1)^2 \pi\cos ^2(\nu\pi   )}G^{\mathfrak H/\overline\varGamma_0(4\sin^2(\nu\pi))}_2\left( \frac{P_{\nu}(1-2t)}{2i\sin(\nu \pi )P_{\nu}(2t-1)},\frac{1-i\cot(\nu\pi)}{2} \right),\end{split}
\end{align}where $N=4\sin^2(\nu\pi)\in\{1,2,3\} $ for $ \nu\in\left\{-\frac16,-\frac14,-\frac13\right\}$.\end{theorem}
\begin{remark}As we may recall from the proof of Corollary \ref{cor:LS_CM}, for $ N=4\sin^2(\nu\pi)\in\{2,3\}$, the statements in the theorem above can be paraphrased by Ramanujan's modular parametrization\begin{align}
z=\frac{iP_\nu(2\alpha_{N}(z)-1)}{\sqrt{N}P_\nu(1-2\alpha_{N}(z))},\label{eq:z_Pnu_ratio}
\end{align}which is valid for $\alpha_{N}(z)\in(\mathbb C\smallsetminus\mathbb R)\cup(0,1) $  and $ z$ in the range specified by \eqref{eq:D234}; for $N=1$, one has \begin{align}
z=\lim_{\varepsilon\to0^+}\frac{iP_{-1/6}\big(\!-\!\sqrt{[j(z+i\varepsilon)-1728]/j(z+i\varepsilon)}\big)}{P_{-1/6}\big(\sqrt{[j(z+i\varepsilon)-1728]/j(z+i\varepsilon)}\big)}\quad
\end{align}when $z$ satisfies \eqref{eq:D1}. Furthermore, extending the definition of harmonic numbers   $   \mathsf H_n\colonequals\sum_{k=1}^n\frac1{k},n\in\mathbb Z_{\geq0}$ to   $  \mathsf H_{w}\colonequals\sum_{k=1}^\infty\big(\frac{1}{k} -\frac{1}{k+w}\big)=\psi ^{(0)}(w+1)+\gamma_{0},w\in\mathbb C\smallsetminus \mathbb Z_{<0}$, one may view  \eqref{eq:GS3}--\eqref{eq:GS6} as generalizations of the binomial harmonic sums  in \eqref{eq:binomH}.  \eor\end{remark}\begin{remark}The original formulation of automorphic Green's function by Gross--Zagier   \cite{GrossZagierI} involves localization of the N\'eron--Tate heights over non-archimedean places. For future work, it is perhaps worthwhile to explore some $p$-adic components of Sun's conjectures \cite{Sun2022} using the non-archimedean  local heights
of Gross--Zagier    \cite[I.4  and III]{GrossZagierI} and/or Disegni's construction of the   $p$-adic Gross--Zagier formula \cite[Theorem B and
\S5]{Disegni2022GZp}.
\eor\end{remark}

 For each  Hecke congruence group \begin{align} \varGamma_0(N):=\left\{ \left.\left(\begin{smallmatrix}a&b\\c&d\end{smallmatrix}\right)\right| a,b,c,d\in\mathbb Z; ad-bc=1;c\equiv0\pmod N\right\}\end{align}of level $N\in\mathbb Z_{>1}$, define the Epstein zeta function on $\varGamma_0(N)$     as  \cite[Chap.\ II, (2.14)]{GrossZagierI} \begin{align}
E^{\varGamma_0(N)}(z,s):=\sum_{\hat\gamma\in\left.\left(\begin{smallmatrix}*&*\\0&*\end{smallmatrix}\right)\right\backslash\varGamma_0(N)}[\I(\hat\gamma z )]^{s},\quad z\in\mathfrak H,\R s>1.
\end{align} For $N=1$, we have \cite[p.~207]{GrossZagier1985}\begin{align}
E^{\varGamma_0(1)}(z,s)=\frac{1}{2\zeta(2s)}\sum_{\substack{m,n\in\mathbb Z\\m^2+n^2\neq0}}\frac{(\I z)^s}{|mz+n|^{2s}},
\end{align}a form that is very similar to the Eisenstein series. The Epstein zeta function is related to the automorphic Green's function by the following relation (see \cite[p.~240, (2.19)]{GrossZagierI} or \cite[p.~39, (6.5)]{Hejhal1983}):\begin{align}
E^{\varGamma_0(N)}(z,s)=-\lim_{z'\to i\infty}\frac{(2s-1)(\I z')^{s-1}}{4\pi}G_{s}^{\mathfrak H/\overline{\varGamma}_0(N)}(z,z')\label{eq:EZF_HeckeN}
\end{align} for $ s\in\mathbb Z_{>1}$  and $ N\in\mathbb Z_{>0}$.

Another goal of the current section is to establish  series representations for certain  $ E^{{\varGamma}_0(4)}(z,2)$, as stated in the next theorem.  \begin{theorem}[Epstein--Sun series]\label{thm:EpsteinSun}For    $|t|\leq 1$ and $ t\notin [-1,0]\cup\{1\}$, we have\begin{align}
\begin{split}
\mathsf S^E_{-1/2}(t)\colonequals{}&\frac{8}{\pi}\sum_{k=1}^\infty\binom{2k}k^2\left( \frac{t}{2^4} \right)^{k}\left[\mathsf H_{2k}^{(2)}-\frac14 \mathsf H_k^{(2)}\right]\\={}&\frac{16}{\pi^{3}}\int_0^t\mathbf K\big(\sqrt{s}\big)\big[\mathbf K\big(\sqrt{1-s}\big)\mathbf K\big(\sqrt{t}\big)-\mathbf K\big(\sqrt{s}\big)\mathbf K\big(\sqrt{1-t}\big)\big]\D s,\end{split}
\label{eq:SE_int_repn}\end{align}where  $\mathbf K\big(\sqrt{t}\big)=\frac{\pi}{2}P_{-1/2}(1-2t)$ is the complete elliptic integral of the first kind. There is a sum rule  for $ |t|<1$ and $ |1-t|<1$:

\begin{align}\begin{split}
&2-\R\left[\frac{1}{\I\frac{i\mathbf K(\sqrt{t})}{\mathbf K(\sqrt{1-t})}}\frac{\mathsf S_{-1/2}^E(t)}{\mathbf K\big(\sqrt{1-t}\big)}+\frac{1}{\I\frac{i\mathbf K(\sqrt{1-t})}{\mathbf K(\sqrt{t})}}\frac{\mathsf S_{-1/2}^E(1-t)}{\mathbf K\big(\sqrt{t}\big)}\right]\\={}&\frac{32}{3}E^{{\varGamma}_0(4)}\left(-\frac{\mathbf K\big(\sqrt{t}\big)}{2\big[i\mathbf K\big(\sqrt{1-t}\big) -\mathbf K\big(\sqrt{t}\big) \big]},2\right),\end{split}\label{eq:SE_sum}
\end{align}where the Epstein zeta function $ E^{\varGamma_0(4)}(z,2)$ is given by \eqref{eq:EZF_HeckeN}. \end{theorem}

The proofs and applications of the last two theorems will occupy \S\S\ref{subsec:nu_deriv}--\ref{subsec:GreenEpstein}.

In \S\ref{subsec:4F3ud}, we will rework the Guillera--Rogers theory \cite{GuilleraRogers2014} via automorphic representations of the Epstein zeta functions $ E^{\varGamma_0(N)}(z,2)$ for $ N\in\{2,3,4\}$, based on  the modular invariants $ \alpha_N(z),N\in\{2,3,4\}$  in \eqref{eq:alpha_N},    the Eisenstein series in \eqref{eq:E2E4E6}, as well as  the Legendre--Ramanujan functions $ R_\nu(1-2\alpha_N(z))$ in \eqref{eq:LegendreRamanujan_E}.\begin{theorem}[Guillera--Rogers summation formulae]\label{thm:GR_sum} If $ |4\alpha_N(z)[1-\alpha_{N}(z)]|\geq1$, $ \alpha_{N}(z)\neq\frac12$,  and $z$ satisfies the inequality constraints in \eqref{eq:D234}, then we have    (cf.\ \cite[Proposition 2]{GuilleraRogers2014}) \begin{align}\begin{split}&\sum_{k=1}^\infty\frac{(k!)^{2}[1-2\alpha_N(z)]\left[2k-\frac{R_\nu(1-2\alpha_N(z))}{1-2\alpha_N(z)}\right]}{k^3\binom{2k}k(-\nu)_k(\nu+1)_k\{\alpha_N(z)[1-\alpha_{N}(z)]\}^k}
\\={}&-\frac{8\pi^{2}}{3}\left[ E^{\varGamma_0(N)}(z,2)-E^{\varGamma_0(N)}\left(-\frac{1}{Nz},2\right) \right]\\{}&+\frac{4\pi^{2}i}{\I z}\R\int_z^{\frac{\R z}{2|\R z|}+\frac{i}{2}\sqrt{\frac{4-N}{N}}}\frac{E_4(w)-N^2E_4(Nw)}{N^{2}-1}(z-w)(\overline z-w)\D w\end{split}\label{eq:GR_Epstein}
\end{align} for  $\nu\in\left\{-\frac{1}{4},-\frac{1}{3},-\frac{1}{2}\right\}$ and $ N=4\sin^2(\nu\pi)\in\{2,3,4\}$. The same is true (by continuity) when the conditions  $|\R z|<\frac12$, $\big|z+\frac1N\big|>\frac1N $, and $\big|z-\frac1N\big|>\frac1N $  in  \eqref{eq:D234} are relaxed to $|\R z|\leq\frac12$,  $\big|z+\frac1N\big|\geq\frac1N $, and $\big|z-\frac1N\big|\geq\frac1N $.

Here, the real part  of \eqref{eq:GR_Epstein} can be reformulated with (see \cite[p.\ 240, (2.16)]{GrossZagierI}) \begin{align}
E^{\varGamma_0(N)}(z,2)-E^{\varGamma_0(N)}\left(-\frac{1}{Nz},2\right)=\frac{E^{\varGamma_0(1)}(Nz,2)-E^{\varGamma_0(1)}(z,2)}{N^{2}-1}\label{eq:EZF_add}
\end{align} for $ N\in\{2,3,4\}$. Meanwhile, the imaginary part of \eqref{eq:GR_Epstein}  is equal to  \begin{align}\begin{split}&
\frac{4\pi^{2}}{3\I z}\left( \R z-\frac{\R z}{2|\R z|} \right)\left[ \left( \R z-\frac{\R z}{2|\R z|} \right)^2+3(\I z)^2+\frac{12-N}{4N} \right]\\{}&+\frac{4\pi^{2}\R\Big[  \widetilde{\mathscr E}_4(Nz)-N\widetilde{\mathscr E}_4(z) \Big]}{N(N^{2}-1)\I z},\end{split}\label{eq:ImGRsum}
\end{align}where the real part of     \begin{align} \widetilde{\mathscr E}_4(z)\colonequals\int_z^{i\infty}[1-E_4(w)](z-w)(\overline z-w)\D w\label{eq:EichlerE4}\end{align}  satisfies (cf.\  \cite[Proposition 3]{GuilleraRogers2014})\begin{align}\R\widetilde{\mathscr E}_4(z)={}&\begin{cases}0, & 2\R z\in\mathbb Z, \\
-\frac{\R z}{3}\left[ \frac{|z|^{2}+2(\I z)^2}{|z|^{4}} +|z|^{2}+2(\I z)^2-5\right], & 2\R \frac{1}{z}\in\mathbb Z. \\
\end{cases}\label{eq:GR_ReE}
\end{align}    
\end{theorem}
 \subsection{Derivatives of Legendre functions with respect to their degrees\label{subsec:nu_deriv}}For any complex-valued degree $ \nu\in\mathbb C$, the Legendre function   $ P_\nu(x),-1<x\leq1$ is the unique $ C^2(-1,1]$ solution to the Legendre differential equation \begin{align}\widehat L_\nu f(x)\colonequals\frac{\D}{\D x}\left[ (1-x^2 )\frac{\D f(x)}{\D x}\right]+\nu(\nu+1) f(x)=0 \tag{\ref{eq:LegendreODE}$'$}\label{eq:LegendreODE'}\end{align}equipped with a natural boundary condition $ f(1)=1$. [Note that \eqref{eq:LegendreODE} and \eqref{eq:LegendreODE'} differ by a reparametrization $x=1-2t$.]

The next lemma exploits the differential equation \eqref{eq:LegendreODE'} and its boundary condition in the construction of an integral representation for $ \partial P_\nu(x)/\partial\nu$.

\begin{lemma}[Variation of parameters for $ \partial P_\nu(x)/\partial\nu$]\label{lm:derivPnu}For $\nu\in(-1,0)$ and $x\in\mathbb C\smallsetminus(-\infty,-1]$, we have \begin{align}
\frac{\partial P_\nu(x)}{\partial\nu}={}&\frac{(2\nu+1)\pi}{2\sin(\nu\pi)}\int^1_x[P_\nu(-\xi)P_\nu(x)-P_\nu(\xi)P_\nu(-x)]P_\nu(\xi)\D\xi,\label{eq:Pnu_deriv_nu}
\end{align}where the path of integration evades the branch cuts of the integrand.\end{lemma}\begin{proof}Differentiating the Legendre differential equation in \eqref{eq:LegendreODE'}  with respect to the degree $\nu$, we get\begin{align}
0=\frac{\partial}{\partial\nu}\left[\widehat L_\nu P_\nu(x)\right]=\widehat L_\nu\frac{\partial P_\nu(x)}{\partial\nu}+(2\nu+1) P_\nu(x).
\end{align}In view of the Wro\'nskian determinant \eqref{eq:W_Pnu}, we may use the standard variation of parameters \cite[\S5.23]{Ince1956ODE} for inhomogeneous differential equations to
check that \begin{align}
\widehat L_\nu\left\{ c_{+}P_\nu(x)+c_{-}P_\nu(-x)-\frac{\pi}{2\sin(\nu\pi)}\int^1_x[P_\nu(-\xi)P_\nu(x)-P_\nu(\xi)P_\nu(-x)]h(\xi)\D\xi \right\}=h(x)\label{eq:Lnu_inhom}
\end{align}for two constants $ c_\pm$, along with a suitably regular function $ h(x),x\in(-1,1)$.

As the natural boundary condition $ P_\nu(x)=1+O(x-1)$  stipulates that\begin{align}\lim_{x\to1-0^+}
\frac{\partial P_\nu(x)}{\partial\nu}=0,
\end{align}while\begin{align}
\quad P_\nu(-x)=\frac{\sin(\nu\pi)}{\pi}\left[ \psi ^{(0)}(\nu +1)+\psi ^{(0)}(-\nu )+2\gamma_{0}+\log\frac{1-x}{2}\right]+O((1-x)\log(1-x))\label{eq:Pnu_log_beh}
\end{align} as $ x\to1-0^+$, we must pick $ c_+=c_-=0$ and $ h(\xi)=-(2\nu+1)P_\nu(\xi)$ in the operand of \eqref{eq:Lnu_inhom} for a correct integral representation of $ \partial P_\nu(x)/\partial\nu$, as declared in \eqref{eq:Pnu_deriv_nu}.
\end{proof}

For  $\nu\in\left\{-\frac{1}{3},-\frac{1}{4},-\frac{1}{6}\right\} $,      $|t|\leq 1$ and $ t\notin [-1,0]\cup\{1\}$, spelling out\begin{align}\begin{split}
\mathsf S_\nu^{G}(t)\colonequals{}&\frac{2\sin(\nu\pi)}{(2\nu+1)\pi}\frac{\partial P_\nu(1-2t)}{\partial\nu}\\={}&\frac{2\sin(\nu\pi)}{(2\nu+1)\pi}\sum_{k=0}^\infty\frac{(-\nu)_{k}(\nu+1)_k}{(1)_{k}}\left[ \pi  \cot (\nu\pi   )-\psi ^{(0)}(k-\nu )+\psi ^{(0)}(k+\nu +1) \right]\frac{t^{k}}{k!}\\={}&\int^1_{1-2t}[P_\nu(-\xi)P_\nu(1-2t)-P_\nu(\xi)P_\nu(2t-1)]P_\nu(\xi)\D\xi,\end{split}
\end{align} we get \eqref{eq:GS3}--\eqref{eq:GS6} in Theorem \ref{thm:GreenSun}.

From either the symmetry $ P_\nu(x)=P_{-\nu-1}(x)$ or the $\nu=-1/2$ case of Lemma \ref{lm:derivPnu}, we know that \begin{align}
\left.\frac{\partial P_\nu(x) }{\partial \nu}\right|_{\nu=-1/2}=0.
\end{align} Differentiating \eqref{eq:Pnu_deriv_nu} again in $\nu$, we get \begin{align}
\left.\frac{\partial ^{2}P_{\nu}(x)}{\partial \nu^{2}}\right|_{\nu=-1/2}=-\pi\int^1_x[P_{-1/2}(-\xi)P_{-1/2}(x)-P_{-1/2}(\xi)P_{-1/2}(-x)]P_{-1/2}(\xi)\D\xi,
\end{align}which entails \eqref{eq:SE_int_repn} in Theorem \ref{thm:EpsteinSun}.

\subsection{Automorphic Green's functions and Epstein zeta functions of weight $4$\label{subsec:GreenEpstein}}The integral forms of  Green--Sun series $ \mathsf S_\nu^{G}(t)$ are pertinent to the representation of $ G_2^{\mathfrak H/\overline\varGamma_0(N)}(z,z')$ [defined in \eqref{eq:wtkAGF}], where $z'=\frac{1}{2}+\frac{i}{2}\sqrt{\frac{4-N}{N}} $ is a special CM point satisfying $ j(z')=0$ (for $ N=1$) or $ \alpha_N(z')=\infty$ (for $N\in\{2,3\} $).

\begin{proposition}[Some automorphic Green's functions of weight $4$]\label{prop:G2}The following identities are true:{\allowdisplaybreaks\begin{align}\begin{split}&
G_2^{\mathfrak H/\overline\varGamma_0(1)}\left(z,\frac{1+i\sqrt{3}}{2}\right)=G_2^{\mathfrak H/\overline\varGamma_0(1)}\left(-\frac{1}{z},\frac{1+i\sqrt{3}}{2}\right)\\={}&-\frac{4\pi}{3\I z}\R\int_{-\sqrt{\frac{j(z)-1728}{j(z)}}}^1[P_{-1/6}(x)]^2\D x+\frac{8\pi \R z}{3\I z}\R\int_{\sqrt{\frac{j(z)-1728}{j(z)}}}^1iP_{-1/6}(x)P_{-1/6}(-x)\D x\\{}&-\frac{4\pi|z|^{2}}{3\I z}\R\int_{\sqrt{\frac{j(z)-1728}{j(z)}}}^1[P_{-1/6}(x)]^2\D x,\text{ where }\I z>0,|z|\geq1,|\R z|<\frac{1}{2},\end{split}\label{eq:G2PSL2Zint}\\\begin{split}&G_2^{\mathfrak H/\overline\varGamma_0(2)}\left(z,\frac{1+i}{2}\right)=G_2^{\mathfrak H/\overline\varGamma_0(2)}\left(-\frac{1}{2z},\frac{1+i}{2}\right)\\={}&-\frac{\pi}{4\I z}\R\int_{2\alpha_{2}(z)-1}^1[P_{-1/4}(x)]^2\D x+\frac{\pi \R z}{\sqrt{2}\I z}\R\int_{1-\alpha_2(z)}^1iP_{-1/4}(x)P_{-1/4}(-x)\D x\\{}&-\frac{\pi|z|^{2}}{2\I z}\R\int_{1-\alpha_2(z)}^1[P_{-1/4}(x)]^2\D x,\text{ where }\I z>0,|\R z|<\frac{1}{2},\left\vert z+\frac{1}{2} \right\vert>\frac{1}{2},\left\vert z-\frac{1}{2} \right\vert>\frac{1}{2},\end{split}\\\begin{split}&G_2^{\mathfrak H/\overline\varGamma_0(3)}\left(z,\frac{3+i\sqrt{3}}{6}\right)=G_2^{\mathfrak H/\overline\varGamma_0(3)}\left(-\frac{1}{3z},\frac{3+i\sqrt{3}}{6}\right)\\={}&-\frac{\pi}{9\I z}\R\int_{2\alpha_{3}(z)-1}^1[P_{-1/3}(x)]^2\D x+\frac{\pi \R z}{3\sqrt{3}\I z}\R\int_{1-\alpha_3(z)}^1iP_{-1/3}(x)P_{-1/3}(-x)\D x\\{}&-\frac{\pi|z|^{2}}{3\I z}\R\int_{1-\alpha_3(z)}^1[P_{-1/3}(x)]^2\D x,\text{ where }\I z>0,|\R z|<\frac{1}{3},\left\vert z+\frac{1}{3} \right\vert>\frac{1}{3},\left\vert z-\frac{1}{3} \right\vert>\frac{1}{3}.\end{split}\label{eq:G2Hecke3int}
\end{align}}\end{proposition}\begin{proof}These are specializations of \cite[(2.2.14) and (2.2.15)]{AGF_PartI}.

One can also verify \eqref{eq:G2PSL2Zint}--\eqref{eq:G2Hecke3int} from first principles, as follows. Check that the difference between the series representation [cf.\ \eqref{eq:wtkAGF}] of $G_2^{\mathfrak H/\overline\varGamma_0(N)}\Big(z,\frac{1}{2}+\frac{i}{2}\sqrt{\frac{4-N}{N}}\Big)$ and its claimed integral representation extends to a bounded and smooth function on  $ Y_0(N)\colonequals \varGamma_0(N)\backslash\mathfrak H$ that is annihilated by the differential operator\begin{align}
\Delta_z^{\mathfrak H}-2\equiv(\I z)^2\left[ \frac{\partial^2}{(\partial\R z)^2} +\frac{\partial^2}{(\partial\I z)^2}\right]-2,
\end{align} and tends to zero as $z$ approaches any cusp in $ \varGamma_0(N)\backslash(\mathbb Q\cup\{i\infty\})$. Since the Laplacian $ \Delta_z^{\mathfrak H}$ in this context does not possess a positive eigenvalue \cite[\S4.1]{IwaniecGSM17}, the aforementioned difference must be identically vanishing on $\varGamma_0(N)\backslash\mathfrak H$.       \end{proof}

At this point, it is clear that Theorem \ref{thm:GreenSun} issues directly from Lemma \ref{lm:derivPnu} and Proposition \ref{prop:G2}.

According to some particular cases of the conjectures by Gross--Zagier  {\cite[p.~317]{GrossZagierI}} and Gross--Kohnen--Zagier {\cite[p.~556]{GrossZagierII}}, we have \begin{align}
\exp \left[\I z\I z'G^{\mathfrak H/\overline {\varGamma}_0(N)}_{2}(z,z')\right]\in\overline{\mathbb Q}\quad \text{if }[\mathbb Q(z):\mathbb Q]=[\mathbb Q(z'):\mathbb Q]=2 \text{ and }N\in\{1,2,3,4\},
\end{align}
with the understanding that $ \exp(-\infty)=0\in\overline{\mathbb Q}$ whenever the automorphic Green's function $ G^{\mathfrak H/\overline {\varGamma}_0(N)}_{2}(z,z')$ diverges to $-\infty$ for either $ j(z)=j(z'), N=1$ or $ \alpha_N(z)=\alpha_N(z'),N\in\{2,3,4\}$. This long-standing algebraicity conjecture (in its full generality) has been recently settled by Bruinier--Li--Yang \cite{BruinierLiYang2022}---a remarkable feat that builds upon the milestones of Zhang \cite{SWZhang1997}, Mellit \cite{MellitThesis}, Viazovska \cite{ViazovskaThesis,Viazovska2015,Viazovska2012,Viazovska2011}, Bruinier--Ehlen--Yang \cite{BruinierEhlenYang2019}, and Li \cite{Li2018,Li2021}.

Therefore, we have an advanced analog of Corollary \ref{cor:LS_CM}, as given below.

\begin{corollary}[Some arithmetic properties of Green--Sun series]\label{cor:GreenSun}For each CM point $ \frac{P_{\nu}(2t-1)}{2i\sin(\nu \pi )P_{\nu}(1-2t)}$ satisfying $|t|<1 $,  $ |1-t|<1$, and $\nu\in\left\{-\frac16,-\frac14,-\frac13\right\}$, we have an explicitly computable algebraic number\begin{align}
\exp\left\{\frac{\pi}{\tan(\nu\pi)}\R\left[\left\vert \frac{P_{\nu}(2t-1)}{P_{\nu}(1-2t)} \right\vert^2\frac{\mathsf S^G_{\nu}(t)}{P_{\nu}(2t-1)}+\frac{\mathsf S^G_{\nu}(1-t)}{P_{\nu}(1-2t)}\right]-\frac{2\pi\cos^{2} (\nu\pi   )}{2 \nu +1}\frac{\I\frac{iP_{\nu}(2t-1)}{P_{\nu}(1-2t)}}{\sin(\nu\pi)}\right\}\in\overline{\mathbb Q},
\end{align} whose factorization follows from \cite[Theorem 1.2]{BruinierLiYang2022}.\qed\end{corollary}\begin{remark}Table \ref{tab:Green} supplies some concrete examples for Theorem \ref{thm:GreenSun} and Corollary \ref{cor:GreenSun}. Here, to keep the table within manageable size, we have picked a small fraction of the modular invariants listed in \cite[Tables 14.3--14.5]{Cooper2017Theta}.\eor\end{remark}

\begin{table}[h]\caption{Selected representations of $ G_2^{\mathfrak H/\overline\varGamma_0(N)}\Big(z,\frac{1}{2}+\frac{i}{2}\sqrt{\frac{4-N}{N}}\Big)$ for $ N\in\{1,2,3\}$\label{tab:Green}}

\begin{scriptsize}\begin{align*}\begin{array}{l|l|@{}r@{}l}\hline\hline
z\vphantom{\frac{\frac{\frac11}1}{\frac{\frac11}1}}&\displaystyle\frac{1}{j(z)} & \multicolumn{2}{l}{\dfrac{\sqrt{3}\I z}{2}G_2^{\mathfrak H/\overline\varGamma_0(1)}\bigg(z,\displaystyle\frac{1+i\sqrt{3}}{2}\bigg)} \\\hline
i\sqrt{2} &\displaystyle\frac{1}{2^{6}\cdot5^3}\vphantom{\dfrac{\frac{11}1}{}}& {}&3\log\dfrac{\sqrt{6}-1}{\sqrt{6}+1}-6\log\big(\sqrt{2}+\sqrt{3}\big)\\&&{}={}&\dfrac{8\pi^{2}\sqrt{\pi}}{\sqrt{3}\Gamma\big(\frac18\big)\Gamma\big(\frac38\big)}\sqrt[4]{\dfrac{2}{5}}\bigg[ \sqrt{2}\mathsf S^G_{-1/6} \bigg( \dfrac{1}{2}-\dfrac{7}{10}\sqrt{\dfrac{2}{5}} \bigg)+\mathsf S^G_{-1/6} \bigg( \dfrac{1}{2}+\dfrac{7}{10}\sqrt{\dfrac{2}{5}} \bigg)\bigg]-3\sqrt{2}\pi \\[12pt]
i\sqrt{3} &\dfrac{1}{2^4\cdot3^3\cdot5^3} &{}&-6\log5\\&&{}={}&\dfrac{8\pi^{3}}{3\big[\Gamma\big(\frac13\big)\big]^3}\dfrac{\sqrt[3]2}{\sqrt[4]5}\left[ \sqrt{3}\mathsf S^G_{-1/6} \left( \dfrac{1}{2}-\dfrac{11}{10\sqrt{5}} \right)+\mathsf S^G_{-1/6} \left( \dfrac{1}{2}+\dfrac{11}{10\sqrt{5}} \right)\right]-3\sqrt{3}\pi \\[12pt]\hline\hline \multicolumn{4}{c}{\vphantom{a}}\\\hline\hline z\vphantom{\frac{\frac{\frac11}1}{\frac{\frac11}1}}&\dfrac{\alpha_{2}(z)[1-\alpha_2(z)]}{2^{6}} & \multicolumn{2}{l}{\dfrac{\I z}{2}G_2^{\mathfrak H/\overline\varGamma_0(2)}\bigg(z,\dfrac{1+i}{2}\bigg)} \\\hline
i &\dfrac{1}{2^3\cdot3^4}\vphantom{\dfrac{\frac{11}1}{}}&{}&-\log3\\&&{} ={}& \displaystyle\frac{\pi^{2}\sqrt{\pi}}{2\sqrt{3}\big[\Gamma\big(\frac14\big)\big]^2}\left[ \sqrt{2} \mathsf S^G_{-1/4}\left( \frac{1}{9} \right)+ \mathsf S^G_{-1/4}\left( \frac{8}{9} \right)\right]-\frac{\pi}{2}\\[12pt]
i\sqrt{\dfrac{3}{2}} &\dfrac{1}{2^8\cdot3^2} &{}&-\log\big(\sqrt{2}+\sqrt{3}\big)\\&&{}={}&\displaystyle\frac{\pi^2\sqrt{\pi}}{2\Gamma\big(\frac{1}{24}\big)\Gamma\big(\frac{11}{24}\big)}\sqrt{\frac{3}{\sqrt{3}-1}}\bigg[ \sqrt{3} \mathsf S^G_{-1/4}\bigg( \frac{1}{2} -\frac{\sqrt{2}}{3}\bigg)+ \mathsf S^G_{-1/4}\bigg( \frac{1}{2} +\frac{\sqrt{2}}{3} \bigg)\bigg]-{}\frac{3\pi}{2\sqrt{6}} \\[12pt]\hline\hline\multicolumn{4}{c}{\vphantom{a}}\\\hline\hline z\vphantom{\frac{\frac{\frac11}1}{\frac{\frac11}1}}&\dfrac{\alpha_{3}(z)[1-\alpha_3(z)]}{3^{3}} & \multicolumn{2}{l}{\dfrac{\sqrt{3}\I z}{6}G_2^{\mathfrak H/\overline\varGamma_0(3)}\bigg(z,\dfrac{3+i\sqrt{3}}{6}\bigg)} \\\hline
i\sqrt{\dfrac{2}{3}} &\dfrac{1}{2^3\cdot3^3}\vphantom{\dfrac{\frac{\frac11}1}{}}& {}&-\dfrac{\log\big(\sqrt{2}+1\big)}{3}\\&&{}={}&\displaystyle\frac{2\sqrt{2}\pi^2\sqrt{\pi}}{9\sqrt{3}\Gamma\big(\frac{1}{24}\big)\Gamma\big(\frac{11}{24}\big)}\sqrt{\frac{\sqrt{2}}{\sqrt{3}-1}}\left[ \sqrt{2} \mathsf S^G_{-1/3}\left( \frac{1}{2} -\frac{1}{2\sqrt{2}}\right)+ \mathsf S^G_{-1/3}\left( \frac{1}{2} +\frac{1}{2\sqrt{2}} \right)\right]-\frac{\pi }{3 \sqrt{6}} \\[15pt]
\dfrac{2i}{\sqrt{3}} &\dfrac{1}{2\cdot3^6} &{}&\dfrac{\log 2}{3}-\dfrac{\log 3}{2}\\&&{}={}& \displaystyle\frac{4\sqrt[3]2\pi^{3}}{81\big[\Gamma\big(\frac13\big)\big]^3}\left[ {2} \mathsf S^G_{-1/3}\left( \frac{1}{2} -\frac{5}{6 \sqrt{3}}\right)+ \mathsf S^G_{-1/3}\left( \frac{1}{2} +\frac{5}{6 \sqrt{3}} \right)\right]-\frac{\pi }{3 \sqrt{3}}\\[12pt]\hline\hline\end{array}\end{align*}\end{scriptsize}
\end{table}

For all $ z\in\mathfrak H$, the modular invariant $\alpha_4(z)$
 evaluates to a  complex number with finite magnitude. Therefore, a natural $N=4$ analog of  Proposition \ref{prop:G2} will invoke asymptotic behavior of automorphic Green's functions [cf.\ \eqref{eq:EZF_HeckeN}], which in turn, ushers in the Epstein zeta function.

\begin{table}
\caption{Selected representations of $ E^{\varGamma_0(4)}\big(\!-\!\frac{1}{4z},2\big)$, where $ \Li_2(w)\colonequals\sum_{k=1}^\infty\frac{w^k}{k^2}$ and $ G=\I \Li_2(i)$\label{tab:EZF4}}

\begin{scriptsize}\begin{align*}\begin{array}{l|l|@{}r@{}l}\hline\hline
\vphantom{\frac{\frac{\frac\int1}1}{\frac{\frac11}1}}z+\dfrac{1}{2}&\dfrac{\alpha_4\big(z+\frac12\big)\big[1-\alpha_4\big(z+\frac12\big)\big]}{2^{4}} & \multicolumn{2}{l}{E^{\varGamma_0(4)}\left(-\dfrac{1}{4z},2\right)} \\\hline \dfrac i2&\dfrac1{2^6}\vphantom{\dfrac{\frac{11}1}{}}&&\dfrac{3G}{2\pi^2}\\&&={}&\displaystyle\frac{3}{16}-\frac{3\sqrt{\pi}}{4\big[\Gamma\big(\frac14\big)\big]^{2}}\mathsf S_{-1/2}^E\left( \frac{1}{2} \right)\\[12pt] \dfrac{i\sqrt{2}}{2}&\dfrac{ \big(\sqrt2-1\big)^{3} }{2^{3}} &&\dfrac{3}{2  \sqrt{2}\pi ^2}\left[ 4\I \Li_2\big( e^{\pi i/4} \big)-3G \right]\\&&={}&\displaystyle\frac{3}{16}-\frac{3\sqrt{\pi}}{4\sqrt{2+\sqrt{2}}\Gamma\big(\frac18\big)\Gamma\big(\frac38\big)}\left[ \sqrt{2} \mathsf S_{-1/2}^E\big( \big(\sqrt2-1\big)^{2} \big)+ \mathsf S_{-1/2}^E\big(2\big(\sqrt2-1\big) \big)\right]\\[12pt]
\dfrac{i\sqrt{3}}{2}&\dfrac1{2^8}&{}&\dfrac{45}{16\sqrt{3}\pi^{2}}\I\Li_2\big(e^{2\pi i/3}\big)\\&&{}={}&\displaystyle\frac{3}{16}-\frac{\sqrt[3]2\sqrt[4]3\pi}{8\big[\Gamma\big(\frac13\big)\big]^3}\left[ \sqrt{3} \mathsf S_{-1/2}^E\left( \frac{2-\sqrt{3}}{2^{2}} \right)+ \mathsf S_{-1/2}^E\left(\frac{2+\sqrt{3}}{2^{2}} \right)\right]\\[12pt]i&\dfrac{\big(\sqrt{2}-1\big)^6}{2 \sqrt{2}}&&\dfrac{3 }{2 \pi ^2}\left[5G-4 \I\Li_2\big(e^{\pi i/4}\big)\right]\\&&={}&\displaystyle\frac{3}{16}-\frac{3\sqrt{\pi}}{4\big(2+\sqrt{2}\big)\big[\Gamma\big(\frac14\big)\big]^{2}}\left[2 \mathsf S_{-1/2}^E\big( \big(\sqrt2-1\big)^{4} \big)+ \mathsf S_{-1/2}^E\big(2^{2}\sqrt{2}\big(\sqrt2-1\big)^{2} \big)\right]\\[12pt]\dfrac{i\sqrt{5}}{2}&\dfrac{1}{2^6}\bigg( \dfrac{\sqrt{5}-1}{2} \bigg)^6&&\dfrac{3}{4 \pi ^2 \sqrt{5}}\left[ 5 \I\Li_2\big(e^{{ 3\pi i }/{10}}\big)+5 \I\Li_2\big(e^{{ 7\pi i}/{10}}\big)-6G\right]\\&&={}&\displaystyle\frac{3}{16}-\frac{3\sqrt[8]5\Gamma\big(\frac1{10}\big)}{4\sqrt[5]{2^{3}}\big[\Gamma\big(\frac1{20}\big)\big]^2}\frac{\sqrt[4]{\sqrt{5}-2}}{\sqrt{1+\sqrt{5}-\sqrt{5+2 \sqrt{5}}}}\left[ \sqrt{5}\mathsf S_{-1/2}^E\left( \frac{\Big(1-\sqrt{\sqrt{5}-2}\Big)^4}{2^3 \left(\frac{\sqrt{5}-1}{2}\right)} \right) \right.\\&&&\displaystyle\left.{}+ \mathsf S_{-1/2}^E\left( \frac{2 \left(\frac{\sqrt{5}-1}{2}\right)^7}{\Big(1-\sqrt{\sqrt{5}-2}\Big)^4} \right)\right]\\[12pt]\dfrac{i\sqrt{6}}{2}&\dfrac{\big(2-\sqrt{3}\big)^3 \big(\sqrt{3}-\sqrt{2}\big)^3}{2^3 \big(\sqrt{2}-1\big)^2}&&\dfrac{3}{4 \pi ^2 \sqrt{6}}\left[ 7 \I\Li_2\big(e^{{ \pi i }/{3}}\big)+36 \I\Li_2\big(e^{{ \pi i}/{12}}\big)-12 \Li_2\big(e^{{ 5\pi i}/{12}}\big)-16G\right]\\&&={}&\displaystyle\frac{3}{16}-\frac{3\sqrt{\pi}\sqrt[4]3\sqrt[4]{2-\sqrt{3}}}{4\Gamma\big(\frac{1}{24}\big)\Gamma\big(\frac{11}{24}\big)}\frac{\sqrt{\sqrt{3}-\sqrt{2}}}{\sqrt{\sqrt{2}-1}}\left[\sqrt{6} \mathsf S_{-1/2}^E\big( \big(\sqrt{3}-\sqrt{2}\big)^2 \big(2-\sqrt{3}\big)^2 \big)\vphantom{\left(\frac{\big(\big)^2}{\big(\big)^2}\right)}\right.\\[10pt]{}&&&\displaystyle{}\left.{}+ \mathsf S_{-1/2}^E\left( \frac{2 \big(\sqrt{3}-\sqrt{2}\big) \big(2-\sqrt{3}\big)}{\big(\sqrt{2}-1\big)^2} \right)\right]\\[12pt]
\dfrac{i\sqrt{7}}{2}&\dfrac1{2^{12}}&{}&\dfrac{21}{16\sqrt{7}\pi^2}\left[ \I\Li_2\big(e^{2\pi i/7}\big)+\I\Li_2\big(e^{4\pi i/7}\big)-\I\Li_2\big(e^{6\pi i/7}\big) \right]\\&&{}={}&\displaystyle\frac{3}{16}-\frac{3\pi}{8\sqrt[4]7\Gamma\big(\frac17\big)\Gamma\big(\frac27\big)\Gamma\big(\frac47\big)}\left[ \sqrt{7} \mathsf S_{-1/2}^E\left( \frac{8-3\sqrt{7}}{2^{4}} \right)+ \mathsf S_{-1/2}^E\left(\frac{8+3\sqrt{7}}{2^{4}} \right)\right]\\[12pt]\hline\hline\end{array}\end{align*}\end{scriptsize}\end{table}

\begin{proposition}[Integral and series representations of Epstein zeta functions]If $z\in\mathfrak H$ satisfies the following inequalities:\begin{align}
\left|\R\left(z+\frac{1}{2}\right)\right|<\frac{1}{2},\quad\left\vert z+\frac{1}{4} \right\vert>\frac{1}{4},\quad \left\vert z+\frac{3}{4} \right\vert>\frac{1}{4},\label{eq:lambda_range_Epstein_Hecke4}
\end{align} then we have\begin{align}\begin{split}
&E^{\varGamma_0(4)}\left( -\frac{1}{4z},2 \right)\\={}&\frac{3}{4\pi^{3}\I z }\R\int_0^{1-\alpha_4\left( z+\frac{1}{2} \right)}\big[\mathbf K\big(\sqrt{t}\big)\big]^2\D t-\frac{3\R\big(z+\frac12\big)}{\pi^{3}\I z }\R\int_0^{\alpha_4\left( z+\frac{1}{2} \right)}i\mathbf K\big(\sqrt{t}\big)\mathbf K\big(\sqrt{1-t}\big)\D t\\{}&+\frac{3\big|z+\frac12\big|^2}{\pi^{3}\I z }\R\int_0^{\alpha_4\left( z+\frac{1}{2} \right)}\big[\mathbf K\big(\sqrt{t}\big)\big]^2\D t,\end{split}\label{eq:Epstein4int}
\end{align}hence \eqref{eq:SE_sum} in Theorem \ref{thm:EpsteinSun}. \end{proposition}\begin{proof}The integral representation in \eqref{eq:Epstein4int} is a reformulation of \cite[(1.14)]{EZF}. Its proof draws on the spectral analysis of the Laplacian $ \Delta_z^{\mathfrak H}$ on $ Y_0(4)\colonequals\varGamma_0(4)\backslash\mathfrak H$, akin to the situation in Proposition  \ref{prop:G2}. \end{proof}\begin{remark}We note that (see \cite[p.\ 240, (2.16)]{GrossZagierI} or \cite[(3.4)]{EZF})\begin{align}
E^{\varGamma_0(4)}(z,2)={}&\frac{1}{15}\left[ E^{\varGamma_0(1)}(4z,2)-\frac{E^{\varGamma_0(1)}(2z,2)}{4} \right],
\end{align}where  solvable CM values of the Epstein zeta function \cite[p.\ 207]{GrossZagier1985}\begin{align}
E^{\varGamma_0(1)}(z,2)=\frac{45}{\pi^{4}}\sum_{\substack{m,n\in\mathbb Z\\m^2+n^2\neq0}}\frac{(\I z)^2}{|mz+n|^{4}}
\end{align}follow from  standard procedures \cite{LatticeSum2013,GlasserZucker1980,WeiGuo2022,Williams1999,Zagier1986,ZagierDilog2007,ZuckerRobertson1984}. This explains the entries in Table \ref{tab:EZF4} that involve special values of the dilogarithm function $ \Li_2(w)$.\eor\end{remark}

\subsection{Integral representations of certain $ _4F_3$ series\label{subsec:4F3ud}}
We give a new proof for  the main identities of Guillera--Rogers  \cite[Propositions 2 and 3]{GuilleraRogers2014}.\begin{proposition}[Ramanujan series upside-down]For $ \nu\in(-1,0)$, $ |4t(t-1)|\geq1$, and $ t\neq\frac12$, the following formula holds true:\begin{align}\begin{split}&
\sum_{k=1}^\infty\frac{(k!)^{2}}{k^3\binom{2k}k(-\nu)_k(\nu+1)_k[t(1-t)]^k}\\={}&\frac{\pi^2}{2\sin^2(\nu\pi)}\int_0^1[P_\nu(1-2s)]^2\left( \frac{1}{s-t}+ \frac{1}{s-1+t}\right)\D s,\end{split}\label{eq:GR_int}
\end{align}where the path of integration runs along the unit interval on the real axis.
Particular cases of  \eqref{eq:GR_int}  allow us to recover the real  part of   \eqref{eq:GR_Epstein}.

For   $\nu\in\left\{-\frac{1}{4},-\frac{1}{3},-\frac{1}{2}\right\}$ and  $ N=4\sin^2(\nu\pi)\in\{2,3,4\}$, we have a modular parametrization\begin{align}\begin{split}{}&
\frac{\pi^2}{2\sin^2(\nu\pi)}\int_0^1\frac{[P_\nu(1-2s)]^2}{[P_\nu(1-2\alpha_{N}(z))]^2}\left[ \frac{1}{s-\alpha_{N}(z)}+ \frac{1}{s-1+\alpha_{N}(z)}\right]\D s\\={}&4\pi^{3}i\int_{z}^{\frac{\R z}{2|\R z|}+\frac{i}{2}\sqrt{\frac{4-N}{N}}}\frac{N^{2}E_{4}(Nw)-E_4(w)}{N^{2}-1}(z-w)^2\D w\end{split}\tag{\ref{eq:GR_int}$^{*}$}\label{eq:GR_int_mod}
\end{align}  for $ z$ satisfying \eqref{eq:D234} and $ \R z\neq0$,
hence the imaginary part of   \eqref{eq:GR_Epstein}. \end{proposition}\begin{proof}Consider the homogeneous Appell differential equation \cite{Appell1881}\begin{align}\begin{split}
\widehat {\mathsf A}_{\nu,t} f(t)\colonequals\frac{\partial }{\partial t}\left\{\left[t (1-t)\frac{\partial }{\partial t} \right]^2+4 \nu  (\nu +1)t(1-t)\right\}f(t)-2\nu  (\nu +1)(1-2t)f(t)=0\end{split}\label{eq:AppellODE}
\end{align}whose solution space is spanned by $ [P_\nu(2t-1)]^2$, $P_\nu(2t-1)P_\nu(1-2t)$, and $[P_\nu(1-2t)]^2$. The left-hand side of \eqref{eq:GR_int}, being  equal to [cf.\ \eqref{eq:defn_pFq}]\begin{align}
\mathscr F_\nu(t)\colonequals\frac{{_4F_3}\left(\left.\begin{array}{@{}c@{}}
1, 1, 1, 1 \\
\frac{3}{2},1-\nu ,\nu +2 \\
\end{array}\right| \frac{1}{4t(1-t)      }\right)}{-2t(1-t)\nu(\nu+1)},\label{eq:GR4F3}
\end{align} is annihilated by \begin{align}\begin{split}\widehat {\mathsf B}_{\nu,t}\colonequals{}& t^3(1-t)^3 (1-2 t) \frac{\partial^{4} }{\partial t^{4}}+2 t^2 (1-t)^2 [3-11 t(1-t) ] \frac{\partial^{3} }{\partial t^{3}}\\{}&+ t(1-t) (1-2 t) \{7+2 [2 \nu  (\nu +1)-15] t(1-t)\}\frac{\partial^{2} }{\partial t^{2}}\\{}&+\left(1+2t (1-t) \left\{5 \nu  (\nu +1)-7-2 [8 \nu  (\nu +1)-9]t (1-t) \right\}\right) \frac{\partial }{\partial t}+2 \nu  (\nu +1) (1-2 t)^3
\\={}&\left\{  t (1-t)(1-2 t)\frac{\partial }{\partial t} +[1-2t(1-t)]\right\} \widehat {\mathsf A}_{\nu,t} .\end{split}\label{eq:BfacAppell}
\end{align} In other words, there is a constant $C$ such that\begin{align}
\widehat {\mathsf A}_{\nu,t}\mathscr F_\nu(t)=C\frac{1-2t}{t(1-t)}.
\end{align}One can find $ C=1$ by comparing the asymptotic expansions of both sides in the $t\to\infty$ regime.

From \eqref{eq:AppellODE}, one can easily check that the operator $\widehat{ \mathsf A}_{\nu,t} ^{\vphantom*}=-\widehat{ \mathsf A}_{\nu,t}^* $ is anti-symmetric. Direct computations also show that $ \widehat {\mathsf A}_{\nu,t}\frac{1}{s-t}=-\widehat {\mathsf A}_{\nu,s}\frac{1}{s-t}$ and  $ \widehat {\mathsf A}_{\nu,t}\frac{1}{s-1+t}=\widehat {\mathsf A}_{\nu,s}\frac{1}{s-1+t}$. Integrating by parts with respect to $s$, and exploiting the asymptotic behavior of $ P_\nu(1-2s),s\to1^-$ [cf.\ \eqref{eq:Pnu_log_beh}],  we can verify that\begin{align}
\widehat {\mathsf A}_{\nu,t} \int_0^1[P_\nu(1-2s)]^2\left( \frac{1}{s-t}+ \frac{1}{s-1+t}\right)\D s=\frac{2\sin^2(\nu\pi)}{\pi^2}\frac{1-2t}{t(1-t)}.
\end{align}  Thus, for  $|4t(1-t)|>1 $ and $\I t\neq0$, the two sides of  \eqref{eq:GR_int} differ by the restriction of a function\footnote{Recall that $ P_\nu(1-2t),t\in\mathbb C\smallsetminus[1,\infty)$ is holomorphic and $ P_\nu(1-2t+i0^+)-P_\nu(1-2t-i0^+)=2i\sin(\nu\pi)P_\nu(2t-1)$ for $t>1$. Therefore, if the function $ \varphi_\nu(t)$ in \eqref{eq:Fnu} satisfies $ \varphi_{\nu}(t+i0^+)=\varphi_\nu(t-i0^+)$ for $ t\in\mathbb R\smallsetminus[0,1]$, then $\varphi_\nu(t)$ has to be a constant multiple of $[P_\nu(2t-1)]^2-[P_\nu(1-2t)]^2-\frac{2i\I t}{|\I t|} \sin(\nu\pi)P_\nu(1-2t)P_\nu(2t-1)$ for $\I t\neq0$. If one further knows  that $ \varphi_\nu(t)=\varphi_\nu(1-t)$ is an even function of $1-2t$, then $ \varphi_\nu(t)$ must vanish identically.}\begin{align}
\varphi_{\nu}(t)=\begin{cases}c_{1}^+[P_\nu(2t-1)]^2+c_{2} ^{+}P_\nu(1-2t)P_\nu(2t-1)+c_{3} ^{+}[P_\nu(1-2t)]^2, & \I t>0, \\
c_{1}^-[P_\nu(2t-1)]^2+c_{2} ^{-}P_\nu(1-2t)P_\nu(2t-1)+c_{3} ^{-}[P_\nu(1-2t)]^2, & \I t<0, \\
\end{cases}\label{eq:Fnu}
\end{align}  where $ c_j^\pm,j\in\{1,2,3\}$ are constants. Since both sides of   \eqref{eq:GR_int}  are even functions of $ 1-2t$ that remain holomorphic for $|4t(1-t)|>1 $,  the function $\varphi_{\nu}(t)$ must be identically vanishing, which establishes   \eqref{eq:GR_int} in full.

To verify the real part of \eqref{eq:GR_Epstein} for $z $ subject to the constraints in  \eqref{eq:D234}, simply note that the identities\begin{align}\left\{\begin{array}{@{}r@{{}={}}r}
E^{\varGamma_0(N)}(z,2)&+\smash[t]{\frac{3}{16\sin^2(\nu\pi)}\R\left\{\left.\left[ 2t(1-t)\frac{\D}{\D t}+R_\nu(1-2t) \right]\int_0^1\frac{[P_\nu(1-2s)]^2}{s-1+t}\D s\right|_{t=\alpha_N(z)}\right\}}\\[10pt]E^{\varGamma_0(N)}\smash[b]{\left(-\frac{1}{Nz },2\right)}&\smash[b]{-\frac{3}{16\sin^2(\nu\pi)}\R\left\{\left.\left[ 2t(1-t)\frac{\D}{\D t}+R_\nu(1-2t) \right]\int_0^1\frac{[P_\nu(1-2s)]^2}{s-t}\D s\right|_{t=\alpha_N(z)}\right\}}\end{array}\right.
\end{align} follow from the integral representations of  automorphic Green's functions $ G_2^{\mathfrak H/\overline{\varGamma}_0(N)}(z,z')$ (cf.\ \cite[(2.2.13) and (2.2.15)]{AGF_PartI})  along with their asymptotic behavior in the $ z'\to i \infty$ limit [cf.\ \eqref{eq:EZF_HeckeN}].

For $ \nu\in(-1,0)$  and $ t\notin\mathbb R$, one can prove \begin{align}\begin{split}0={}&
\int_0^1[P_\nu(1-2s)]^2\left( \frac{1}{s-t}+ \frac{1}{s-1+t}\right)\D s\\{}&-\int_{t}^\infty[P_\nu(1-2s)P_\nu(2t-1)-P_\nu(2s-1)P_\nu(1-2t)]^2\frac{2s-1}{s(1-s)}\D s\end{split}\label{eq:Appell_alt}
\end{align}  (where the path of integration in the second integral does not intersect the real axis) by checking that the right-hand side enjoys three properties: (1)~It is an even function of $1-2t$; (2)~It is annihilated by   $\widehat{\mathsf A}_\nu $; (3) It lacks jump discontinuities\footnote{For two fixed real numbers $ \R s,\R t>1$, we have $ \lim_{\I s\to0^+,\I t\to0^+}[P_\nu(1-2s)P_\nu(2t-1)-P_\nu(2s-1)P_\nu(1-2t)]^2=\lim_{\I s\to0^-,\I t\to0^-}[P_\nu(1-2s)P_\nu(2t-1)-P_\nu(2s-1)P_\nu(1-2t)]^2\in\mathbb R$.  One has a similar argument for two fixed real numbers $ \R s,\R t<0$. }  for $ t\in\mathbb R\smallsetminus[0,1]$. To relate \eqref{eq:Appell_alt} to  \eqref{eq:GR_int_mod}, simply differentiate   Ramanujan's  modular parametrization in \eqref{eq:z_Pnu_ratio} with the help from the Wro\'nskian determinant in  \eqref{eq:W_Pnu}:  \begin{align}\frac{
\D z}{\D\alpha_N(z)}=\frac{1}{2\pi i\alpha_N(z)[1-\alpha_N(z)][P_\nu(1-2\alpha_N(z))]^2},
\end{align} and observe that (cf.\ \cite[Theorem 4.31]{Cooper2017Theta})\begin{align}
[1-2\alpha_{N}(z)][P_\nu(1-2\alpha_N(z))]^4=\frac{N^{2}E_4(Nz)-E_4(z)}{N^{2}-1}.\label{eq:4Pnu}
\end{align}      On the right-hand side of  \eqref{eq:GR_int_mod}, one can pick a path connecting $ z$ to  $ \frac{\R z}{2|\R z|}+\frac{i}{2}\sqrt{\frac{4-N}{N}}$ that does not cross the positive $ \I z$-axis, on which $ \alpha_N(z)\in(0,1)$ corresponds to the branch cut for the left-hand side of  \eqref{eq:GR_int_mod}.

To derive the imaginary part of \eqref{eq:GR_Epstein} from   \eqref{eq:GR_int_mod}, note that\begin{align}
\left.\left[ 2t(1-t)\frac{\D}{\D t}+R_\nu(1-2t) \right]f(t)\right|_{t=\alpha_N(z)}=-\frac{\I z}{\pi}\frac{\partial}{\partial\I z}\left\{\frac{1}{\I z}\frac{f(\alpha_N(z))}{[P_\nu(1-2\alpha_N(z))]^2}\right\}
\end{align}for  $z$ satisfying \eqref{eq:D234}.

The  imaginary part of \eqref{eq:GR_Epstein}   splits into three summands:\begin{itemize}
\item
an elementary contribution\begin{align}\begin{split}&
-\frac{4\pi^{2}}{\I z}\R\int_z^{\frac{\R z}{2|\R z|}+\frac{i}{2}\sqrt{\frac{4-N}{N}}}(z-w)(\overline z-w)\D w\\={}&\frac{4\pi^{2}}{3\I z}\left( \R z-\frac{\R z}{2|\R z|} \right)\left[ \left( \R z-\frac{\R z}{2|\R z|} \right)^2+3(\I z)^2-\frac{3(4-N)}{4N} \right];\end{split}
\end{align}\item an expression \begin{align}\begin{split}&
\frac{4\pi^{2}}{\I z}\R\int_z^{i\infty}\frac{E_4(w)-N^2E_4(Nw)-1+N^2}{N^{2}-1}(z-w)(\overline z-w)\D w\\={}&\frac{4\pi^{2}\R\Big[  \widetilde{\mathscr E}_4(Nz)-N\widetilde{\mathscr E}_4(z) \Big]}{N(N^{2}-1)\I z}\end{split}
\end{align}involving Eichler-type integrals defined in \eqref{eq:EichlerE4};\item an integral{\allowdisplaybreaks\begin{align}\begin{split}&
\frac{4\pi^{2}}{\I z}\R\int^{\frac{\R z}{2|\R z|}+\frac{i}{2}\sqrt{\frac{4-N}{N}}}_{i\infty}\frac{E_4(w)-N^2E_4(Nw)-1+N^2}{N^{2}-1}(z-w)(\overline z-w)\D w\\={}&\frac{8\pi^{2}}{\I z}\left( \R z-\frac{\R z}{2|\R z|} \right)\R\int^{\frac{\R z}{2|\R z|}+\frac{i}{2}\sqrt{\frac{4-N}{N}}}_{i\infty}\frac{E_4(w)-N^2E_4(Nw)-1+N^2}{N^{2}-1}\left( w-\frac{\R z}{2|\R z|} \right)\D w\\={}&\frac{8}{\I z}\left( \R z-\frac{\R z}{2|\R z|} \right)\int_1^\infty\frac{x[P_\nu(x)]^{4}-1}{(1-x^{2})[P_\nu(x)]^{2}}\left[\frac{\pi\cot(\nu\pi)}{2}- \frac{Q_\nu(x)}{P_\nu(x)} \right]\D x\\={}&\frac{8\pi^{2}}{\I z}\left( \R z-\frac{\R z}{2|\R z|} \right)\frac{6-N}{6N}\end{split}
\end{align}where the identities\begin{align}\begin{split}&
\int_{1}^\infty\frac{x[P_\nu(x)]^{4}-1}{(1-x^{2})[P_\nu(x)]^{2}}\D x=\lim_{\delta\to0^+}
\left\{\int_{1+\delta}^\infty\frac{x[P_\nu(x)]^{2}\D x}{1-x^{2}}+\frac{Q_\nu(1+\delta)}{P_\nu(1+\delta)}\right\}\\={}&\lim_{\delta\to0^+}
\left\{\int_{1+\delta}^\infty\frac{x[P_\nu(x)]^{2}\D x}{1-x^{2}}-\frac{2\gamma_{0}+2\psi^{(0)}(\nu)+\log\frac{\delta}{2}}{2}\right\}=\frac{\pi\cot(\nu\pi)^{}}{2},\end{split}\\\begin{split}&
\int_{1}^\infty\frac{x[P_\nu(x)]^{4}-1}{(1-x^{2})[P_\nu(x)]^{2}}\frac{Q_\nu(x)}{P_\nu(x)}\D x=\lim_{\delta\to0^+}
\left\{\int_{1+\delta}^\infty\frac{xP_\nu(x)Q_\nu(x)\D x}{1-x^{2}}+\frac{[Q_\nu(1+\delta)]^{2}}{2[P_\nu(1+\delta)]^{2}}\right\}\\={}&\lim_{\delta\to0^+}
\left\{\int_{1+\delta}^\infty\frac{xP_\nu(x)Q_\nu(x)\D x}{1-x^{2}}+\frac{\big[2\gamma_{0}+2\psi^{(0)}(\nu)+\log\frac{\delta}{2}\big]^{2}}{8}\right\}=-\frac{\pi^{2}}{12},\end{split}
\end{align}}are deducible from the methods in \cite[\S3.1]{AGF_PartI}.
\end{itemize} These  can be rearranged into \eqref{eq:ImGRsum}.
When $ 2\R z\in\mathbb Z$, the integrand of   $ \widetilde{\mathscr E}_4(z)$ is real-valued along an integration path parallel to the $\I z$-axis, thus making   $ \widetilde{\mathscr E}_4(z)$ purely imaginary; when $ 2\R \frac1z\in\mathbb Z$, we may divide Ramanujan's reflection formula \cite[p.\ 276]{RN2}\begin{align}
\mathscr E_4(z)-z^2\mathscr E_4\left( -\frac{1}{z} \right)=-\frac{z^4-5 z^2+1}{3 z}+\frac{30 i\zeta (3) \left(z^2-1\right) }{\pi ^3}
\end{align} for   $\mathscr E_4(z)\colonequals\int_{z}^{i\infty} [1-E_4(w)](z-w)^2\D w=\frac{60 i}{\pi^3}\sum_{k=1}^\infty\frac{1}{k^3(e^{-2\pi ik z}-1)}$ by $ \I z$ before differentiating with respect to $ \I z$, so that  the resulting  reflection formula\begin{align}\frac{|z|^{2}}{(\I z)^2}\R\widetilde {\mathscr E}_4\left( -\frac{1}{z} \right)-
\frac{\R\widetilde {\mathscr E}_4(z)}{(\I z)^2}=\frac{\R z}{3(\I z)^2}\left[ \frac{|z|^{2}+2(\I z)^2}{|z|^{4}} +|z|^{2}+2(\I z)^2-5\right]
\end{align}brings us a closed-form evaluation of $ \R\widetilde{\mathscr E}_4(z)$ in \eqref{eq:GR_ReE}.
  \end{proof}\begin{remark}As pointed out by Guillera--Rogers \cite[\S4.1]{GuilleraRogers2014} and Sun \cite[\S3]{Sun2022},  particular cases of  \eqref{eq:GR_Epstein}  have been  investigated by several authors, including Zeilberger \cite{Zeilberger1993pun}, Guillera \cite{Guillera2008ext10,Guillera2013div}, and Hessami Pilehrood--Hessami Pilehrood \cite{Pilehroods2011Hurwitz}.\eor\end{remark}\begin{remark}As noted by Guillera--Rogers \cite[around (43)]{GuilleraRogers2014}, there are no automorphic interpretations for the $\nu=-\frac16$ case. This comes as no surprise, since the case violates $ \varGamma_0(1)$ symmetry in two aspects: (1)~The following integral representation (cf.\ \cite[(2.2.13)--(2.2.14)]{AGF_PartI}) \begin{align}\begin{split}
E^{\varGamma_0(1)}(z,2)={}&-\frac{3}{4}\R\left[ 2t(1-t)\frac{\D}{\D t}+R_{-1/6}(1-2t) \right]\int_0^1[P_{-1/6}(1-2s)]^2\\{}&\left.\times\left( \frac{1}{s-t}- \frac{1}{s-1+t}\right)\D s\right|_{t=\frac12-\frac12\sqrt{\frac{j(z)-1728}{j(z)}}}\end{split}
\end{align}is incompatible with the real part for the right-hand side of  \eqref{eq:GR_int}; (2)~When $ t=\frac12-\frac12\sqrt{\frac{j(z)-1728}{j(z)}}$ for a certain $z$ satisfying $ \I z>0,|\R z|<\frac12,|z|>1$, the expression [cf.\ \eqref{eq:4Pnu}]\begin{align}
(1-2t)[P_{-1/6}(1-2t)]^4=\sqrt{\frac{j(z)-1728}{j(z)}} E_4(z)
\end{align}  disqualifies  the right-hand side of  \eqref{eq:GR_int} from being an Eichler integral for a holomorphic modular form of weight $4$. \eor\end{remark}

\section{Clebsch--Gordan couplings and binomial harmonic sums\label{sec:CG_coupling}}
Using the Mehler--Dirichlet integral representation for the Legendre functions \begin{align}\label{eq:MD_Pnu}P_\nu(\cos\theta):={}&\frac{2}{\pi}\int_0^\theta\frac{\cos\frac{(2\nu+1)\beta}{2}\D\beta}{\sqrt{2(\cos\beta-\cos\theta)}},& \theta\in(0,\pi)\quad \end{align}of arbitrary complex degrees $ \nu\in\mathbb C$, we define the generalized Clebsch--Gordan integrals  \begin{align}
T_{\mu,\nu}\colonequals{}&\int_{-1}^1 P_{\mu}(x)P_\nu(x)P_\nu(-x)\D x,&\mu,\nu\in{}&\mathbb{C},\\\widetilde T_\nu\colonequals{}&\int_{-1}^1 \frac{[P_\nu(x)]^{2}P_\nu(-x)\D x}{(1-x^{2})^{\frac{1-\nu}{2}}},&\R\nu{}>&-1.
\end{align} About a decade ago, closed forms were found \cite{Zhou2013Pnu,Zhou2013Int3Pnu} for $ T_{\nu,\nu}$, $ T_{2\nu-1,\nu}$, and $ \widetilde T_\nu$, along with some generalized hypergeometric representations  for $T_{\mu,\nu}$. Recently, interest in these Clebsch--Gordan integrals has been revived by  Cantarini \cite{Cantarini2022} and Campbell \cite{Campbell2023CCG,Campbell2023mma,Campbell2024lemn}.

In \S\S\ref{subsec:CG}--\ref{subsec:BM_CG}, we review and re-derive the main results from  \cite{Zhou2013Pnu,Zhou2013Int3Pnu}, using automated manipulations of generalized hypergeometric series. Some series identities therein can be reinterpreted as closed-form evaluations of binomial harmonic sums, as in  \eqref{eq:CGeg1}--\eqref{eq:CGeg3}.

In \S\ref{subsec:Pmnnu}, we extend the Clebsch--Gordan couplings to associated Legendre functions [see \eqref{eq:EulerPnumu}], and study yet another type of integral over the products of three Legendre functions, namely\begin{align}
U_{\mu,\nu}\colonequals{}&\int_{-1}^1 \frac{P_{\mu}(x)+P_{\mu}(-x)}{2}[P_\nu(x)]^{2}\D x,&\mu,\nu\in{}&\mathbb{C}.\label{eq:SigmaCG_defn}
\end{align}These efforts will culminate in a new theorem that is not found in  \cite{Zhou2013Pnu,Zhou2013Int3Pnu}.\begin{theorem}[Correspondence between $T_{\mu,\nu}$ and $U_{\mu,\nu}$]\label{thm:TU}
If we define \begin{align}
c_{\mu,\nu}(\varepsilon)\colonequals\frac{\Gamma \big(\frac{1}{2}+\varepsilon\big) \Gamma (-\nu+\varepsilon) \Gamma (\nu+1+\varepsilon)}{\Gamma (1+\varepsilon) \Gamma \big(\frac{2-\mu }{2}+\varepsilon\big) \Gamma \big(\frac{\mu+3 }{2}+\varepsilon \big)},
\end{align}then we have \begin{align}T_{\mu,\nu}={}&\frac{2\sin(\mu\pi)\sin(\nu\pi)}{\mu(\mu+1)\pi^{2}}
\left.\frac{\partial}{\partial\varepsilon}\right|_{\varepsilon=0}\left[\frac{c_{\mu,\nu}(\varepsilon)}{c_{\mu,\nu}(0)}{_4F_3}\left(\left.\begin{array}{@{}c@{}}
1,\frac{1}{2}+\varepsilon,-\nu+\varepsilon,\nu+1+\varepsilon \\[4pt]
1+\varepsilon,\frac{2-\mu }{2}+\varepsilon,\frac{\mu+3 }{2}+\varepsilon\ \\
\end{array}\right| 1\right)\right],\label{eq:SigmaCG4F3'}\\U_{\mu,\nu}={}&\frac{\sin(\mu\pi)\sin^2(\nu\pi)}{\mu(\mu+1)\pi^{3}}
\left.\frac{\partial^{2}}{\partial\varepsilon^{2}}\right|_{\varepsilon=0}\left[\frac{c_{\mu,\nu}(\varepsilon)}{c_{\mu,\nu}(0)}{_4F_3}\left(\left.\begin{array}{@{}c@{}}
1,\frac{1}{2}+\varepsilon,-\nu+\varepsilon,\nu+1+\varepsilon \\[4pt]
1+\varepsilon,\frac{2-\mu }{2}+\varepsilon,\frac{\mu+3 }{2}+\varepsilon\ \\
\end{array}\right| 1\right)\right],\label{eq:SigmaCG4F3''}
\end{align}when $ \mu\notin\mathbb Z$.\end{theorem} \subsection{Clebsch--Gordan couplings of Legendre functions\label{subsec:CG}}
In this subsection, we recall some results from our previous works \cite{Zhou2013Pnu,Zhou2013Int3Pnu}, and compare them to certain consequences of  Lemma \ref{lm:CFZ}.
  \begin{proposition}[Evaluations of $T_{\mu,\nu} $]\label{prop:2013Pnu}If $ 2\mu\notin\mathbb Z$ and $2\nu\notin\mathbb Z$, then we have the following representations\begin{subequations} {\allowdisplaybreaks\begin{align}
\begin{split}T_{\mu,\nu}={}&\frac{2}{\pi}\frac{\sin(\mu\pi)\cos(\nu\pi)}{2\nu+1}\left[ \, \frac{1}{\mu}\,{_5F_4}\left(\left.\begin{array}{@{}c@{}}
\frac{1}{2},\frac{1}{2},-\frac{\mu }{2},-\nu ,\nu+1 \\[4pt]
1,\frac{2-\mu }{2},\frac{1-2\nu}{2} ,\frac{2\nu+3}{2} \\
\end{array}\right| 1\right) -\frac{1}{\mu+1}\, {_5F_4}\left(\left.\begin{array}{@{}c@{}}
\frac{1}{2},\frac{1}{2},\frac{1+\mu}{2},-\nu ,\nu+1\ \\[4pt]
1,\frac{3+\mu}{2},\frac{1-2\nu}{2} ,\frac{2\nu+3}{2}
\end{array}\right| 1\right) \right]\end{split}\label{eq:SmunuCGa}\\\begin{split}={}&\frac{2}{\pi^2}\frac{\sin(\mu\pi)\sin(\nu\pi)}{\mu(\mu+1)}\left[ \, \frac{1}{\nu}\,{_4F_3}\left(\left.\begin{array}{@{}c@{}}
1,\frac{1-\mu}{2},\frac{\mu+2 }{2},-\nu \\[4pt]
\frac{2-\mu }{2},\frac{\mu+3 }{2},1-\nu \\
\end{array}\right| 1\right) -\frac{1}{\nu+1} \,{ _4F_3}\left(\left.\begin{array}{@{}c@{}}
1,\frac{1-\mu}{2},\frac{\mu+2 }{2},\nu+1\ \\[4pt]
\frac{2-\mu }{2},\frac{\mu+3 }{2},\nu+2 \\
\end{array}\right| 1\right) \right]\end{split}\label{eq:SmunuCGb}\\\begin{split}={}&\frac{2}{\pi^2}\frac{\sin(\mu\pi)\sin(\nu\pi)}{\mu(\mu+1)}\left\{\left.\frac{\partial}{\partial\varepsilon}\right|_{\varepsilon=0}{_4F_3}\left(\left.\begin{array}{@{}c@{}}
1,\frac{1}{2}+\varepsilon,-\nu+\varepsilon,\nu+1+\varepsilon\ \\[4pt]
1+\varepsilon,\frac{2-\mu }{2}+\varepsilon,\frac{\mu+3 }{2}+\varepsilon\ \\
\end{array}\right| 1\right)\right.\\{}&\left.{}-\left[ \psi ^{(0)}\left(\frac{\mu+3 }{2}\right)+\psi ^{(0)}\left(\frac{2-\mu }{2}\right)-\psi ^{(0)}(\nu +1)-\psi ^{(0)}(-\nu )+2\log 2 \right]\right.\\&\left.{}\times{_3F_2}\left(\left.\begin{array}{@{}c@{}}
\frac{1}{2},-\nu ,\nu+1\ \\[4pt]
\frac{2-\mu }{2},\frac{\mu+3 }{2} \\
\end{array}\right| 1\right)\right\}
\end{split}\label{eq:SmunuCGc}\end{align}}\end{subequations}involving  generalized hypergeometric series [cf.\ \eqref{eq:defn_pFq}]. In particular, we have\begin{align}
T_{\nu,\nu}={}&\frac{1+2\cos(\nu\pi )}{3}\frac{\pi\Gamma\big(\frac{\nu+1 }{2}\big) \Gamma\big(\frac{3 \nu+2 }{2}\big)}{\left[ \Gamma\big(\frac{1-\nu}{2}\big)\right]^2\left[ \Gamma\big(\frac{\nu+2 }{2}\big)\right]^3 \Gamma\big(\frac{3 \nu+3 }{2}\big)},\label{eq:Snunu}\\T_{2\nu-1,\nu}={}&\frac{\sin(\nu\pi )\sin(2\nu\pi )}{\nu^{2}\pi^{2}},\label{eq:S2nu-1nu}
\end{align}while interpreting indeterminate forms  through l'H\^opital's rule.\end{proposition}
 \begin{proof}The $ _5F_4$ (resp.\ $ _4F_3$) representation in \eqref{eq:SmunuCGa} [resp.\ \eqref{eq:SmunuCGb}] comes from \cite[(19$^*$)]{Zhou2013Pnu} (resp.\ \cite[(19)]{Zhou2013Pnu}).

To prove  \eqref{eq:SmunuCGc} [which is equivalent to \eqref{eq:SigmaCG4F3'}], use   \eqref{eq:P+P+}--\eqref{eq:P+P-} for\begin{align}\begin{split}
T_{\mu,\nu}={}&2\int_0^1P_\mu(1-2t)P_\nu(1-2t)P_\nu(2t-1)\D t\\={}&\frac{2\sin(\nu\pi)}{\pi}\int_0^1P_\mu(1-2t)\sum_{k=0}^\infty[t(1-t)]^ks_{\nu,k}\left[\varsigma_{\nu,k}-\log\frac{1}{4t(1-t)}\right]\D t,\end{split}
\end{align} and then perform termwise integration with the assistance from \eqref{eq:Pnu_avg}.

One can find the closed forms of $ T_{\nu,\nu}$ and $ T_{2\nu-1,\nu}$  in \cite[(19$_{(\nu,\nu)}$)]{Zhou2013Pnu} and \cite[(19$_{(2\nu-1,\nu)}$)]{Zhou2013Pnu}.\end{proof}
\begin{remark}The generalized hypergeometric series [cf.\ \eqref{eq:defn_pFq}]
\begin{align}{_pF_{p-1}}\left(\left.\begin{array}{@{}c@{}}
a_{1},\dots,a_p \\[4pt]
b_{1},\dots,b_{p-1} \\
\end{array}\right| 1\right):=1+\sum_{n=1}^\infty\frac{\prod_{j=1}^p(a_{j})_n}{\prod_{k=1}^{p-1}(b_{k})_n }\frac{1}{n!}\label{eq:defn_pFq(1)}
\end{align} converges when $(\{a_1,\dots,a_p\}\cup \{b_1,\dots,b_{p-1}\})\cap\mathbb Z_{\leq 0}=\varnothing$ and \begin{align}
\R\sum_{j=1}^pa_{j}<\R\sum_{k=1}^{p-1}b_{k}.\label{eq:pFq_ind_ineq}
\end{align}Therefore, all the generalized hypergeometric series appearing in \eqref{eq:SmunuCGa}--\eqref{eq:SmunuCGc}  are  indeed convergent.  \eor\end{remark}
\begin{remark}Laporta \cite[(48)]{Laporta2008} and Wan \cite[(54)--(55)]{Wan2012} independently discovered \begin{align}
\int_0^1\big[\mathbf K\big(\sqrt{t}\big)\big]^2\mathbf K\big(\sqrt{1-t}\big)\D t=\frac{\big[\Gamma\big(\frac14\big)\big]^8}{384\pi^2}
\end{align}through numerical experiments, where $   \mathbf K\big(\sqrt{t}\big)=\frac{\pi}{2}P_{-1/2}(1-2t)$. The same was proved by the present author \cite{Zhou2013Pnu} (via a special case  $T_{-1/2,-1/2}$) and by Rogers--Wan--Zucker \cite{RogersWanZucker2015} (as one of various lattice sums representable through modular  $L$-values).\eor\end{remark}
\begin{table}[t]\caption{Selected series related to $ T_{\nu,\nu}$ \label{tab:CG1}}

\begin{scriptsize}\begin{align*}\begin{array}{c|r@{}c@{}l}\hline\hline
\nu & \multicolumn{3}{l}{T_{\nu,\nu}}\\\hline
-\dfrac{1}{2} &\vphantom{\frac{\frac{\frac1 \int}1}1}\dfrac{\big[\Gamma\big(\frac14\big)\big]^8}{24\pi^{5}}&{}= {}&\displaystyle4\sum_{k=0}^\infty\frac{\binom{2k}k^4}{2^{8k}}\frac{1}{4k+1}
\\[10pt]-\dfrac{1}{3}&\dfrac{3 \sqrt{3}\big[\Gamma\big(\frac13\big)\big]^9}{32\pi^{5}}&{}= {}&\displaystyle\frac{27\sqrt{3}}{4\pi}\sum_{k=0}^\infty\frac{\binom{2k}k^3}{\binom{6k}{3k}}\frac{4 k+1}{(6 k+1)^2 (3 k+1)}
\\[10pt]-\dfrac{1}{4}&\dfrac{\big[\Gamma\big(\frac18\big)\Gamma\big(\frac38\big)\big]^2}{3\pi^3}&{}= {}&\displaystyle\frac{32}{\pi}\sum_{k=0}^\infty\frac{\binom{2k}k^2}{2^{4k}}\frac{1}{(8 k+1) (8 k+3)}
\\[10pt]-\dfrac{1}{6}&\dfrac{\big[\Gamma\big(\frac14\big)\big]^4}{\sqrt{2}\sqrt[4]3\pi^{3}}&{}= {}&\displaystyle\frac{27\sqrt{3}}{\pi}\sum_{k=0}^\infty\frac{\binom{2k}k\binom{6k}{3k}}{2^{8k}}\frac{4k+1}{(3 k+1) (12 k+1) (12 k+5)}\\[12pt]\hline\hline\end{array}\end{align*}\end{scriptsize}

\end{table}

\begin{table}[t]\caption{Selected series related to $ T_{2\nu-1,\nu}$\label{tab:CG2}}

\begin{scriptsize}\begin{align*}\begin{array}{r|r@{}c@{}l}\hline\hline
n & \multicolumn{3}{l}{\left.\dfrac{1}{\pi}\dfrac{\partial^{n}T_{2\nu-1,\nu}}{\partial\nu^{n}}\right|_{\nu=-1/2}\vphantom{\frac{\frac\int1}1}}\\\hline
1 &\dfrac{8}{\pi^2}&{}= {}&\displaystyle\sum_{k=0}^\infty \frac{\binom{2k}k^4}{2^{8k}}\frac{4k+1}{(k+1)(1-2k)}\vphantom{\frac{\frac11}1}\\
2 & \dfrac{64}{\pi ^2} &{}= {}&\displaystyle6\sum_{k=0}^\infty \frac{\binom{2k}k^4}{2^{8k}}\frac{4k+1}{(k+1)^{2}(1-2k)^{2}}\\
3 & \dfrac{576}{\pi ^2}-56 &{}={}&\displaystyle6\sum_{k=0}^\infty \frac{\binom{2k}k^4}{2^{8k}}\frac{(4k+1)\Big\{2 (k+1) (1-2 k) \left[\mathsf H_k^{(2)}-2\mathsf  H_{2 k}^{(2)}-\frac{5 \pi ^2}{12}\right]+\frac{9}{(k+1) (1-2 k)}-2\Big\}}{(k+1)^{2}(1-2k)^{2}}\\4&\dfrac{6144}{\pi ^2}-896&{}={}&\displaystyle 144\sum_{k=0}^\infty \frac{\binom{2k}k^4}{2^{8k}}\frac{(4k+1)\Big\{2 (k+1)^{2} \left[\mathsf H_k^{(2)}-2\mathsf  H_{2 k}^{(2)}-\frac{5 \pi ^2}{12}\right]+\frac{4 (k+1)^{2}}{(1-2 k)^{2}}+1\Big\}}{(k+1)^{4}(1-2k)^{2}}\\\hline\hline
\end{array}\end{align*}\end{scriptsize}
\end{table}

\begin{remark}
In Table \ref{tab:CG1} (resp.\  Table \ref{tab:CG2}), we  showcase the series representations based on \eqref{eq:SmunuCGa} for   $T_{\nu,\nu} $ (resp.\  low-order derivatives of $T_{2\nu-1,\nu} $), with limit procedures $ \mu=\nu\to-\frac12$ (resp.\   $\nu=\frac{\mu+1}{2}\to-\frac12$) wherever necessary. As one may note, \begin{align}
\left.\dfrac{1}{\pi}\dfrac{\partial T_{2\nu-1,\nu}}{\partial\nu}\right|_{\nu=-1/2},\quad \frac{T_{-1/2,-1/2}}{4},\quad\frac{1}{8}\left.\frac{\partial^{2}T_{\nu,\nu}}{\partial\nu^{2}}\right|_{\nu=-1/2}
\end{align}can be  paraphrased into \eqref{eq:CGeg1}--\eqref{eq:CGeg3} mentioned in the introduction.

To spell out the summands for series  \eqref{eq:SmunuCGb} and \eqref{eq:SmunuCGc}, one needs generalized binomial coefficients $\binom n m\colonequals\frac{\Gamma(n+1)}{\Gamma(m+1)\Gamma(n-m+1)}$, as well as values of the digamma function at fractional arguments, which do not fit the pattern in \eqref{eq:binomH}.  \eor\end{remark}\begin{remark}Some series identities  in Tables \ref{tab:CG1} and  \ref{tab:CG2} involve  $ \binom{2k}k^4$, which  agree in part with a result of Cantarini and D'Aurizio \cite[Corollaries 1 and 2]{CantariniAurizio2019}: the relations\begin{align}
\pi^{2}\sum_{k=0}^\infty\frac{\binom{2k}k^4}{2^{8k}}\frac{4k+1}{p(k)}\in\mathbb Q,\quad \pi^{2}\sum_{k=0}^\infty\frac{\binom{2k}k^4}{2^{8k}}\frac{4k+1}{[p(k)]^{2}}\in\mathbb Q
\end{align} hold for infinitely many polynomials $ p(k)\in\mathbb Z[k]$ with integer coefficients. Formulae for infinite series involving $ \binom{2k}k^4$ are otherwise rare, being notably absent from Sun's recent conjectures \cite[\S3]{Sun2022}.   \eor\end{remark}\begin{remark}Cantarini--D'Aurizio \cite{CantariniAurizio2019}, Campbell--Cantarini--D'Aurizio \cite{CampbellCantariniAurizio2022},  Cantarini \cite{Cantarini2022}, and Campbell \cite{Campbell2023mma}  have evaluated a wide variety of binomial harmonic sums using Fourier--Legendre expansions, which are independent of the approach we employ in the present article. Campbell has also published a recent proof \cite{Campbell2024lemn} of \eqref{eq:CGeg2} by Fourier--Legendre expansions.    \eor\end{remark}
  \begin{proposition}[Evaluations of $\widetilde{T}_{\nu} $]\label{prop:2013Int3Pnu}If  $\R\nu>-1$, then \begin{align}
\widetilde{T}_{\nu}={}&\frac{1}{\pi}\left(\frac{\cos\frac{\nu\pi}{2}}{2^{\nu}}\right)^3\left[ \frac{\Gamma\big(\frac{1+\nu}{2}\big)}{\Gamma\big(\frac{2+\nu}{2}\big)} \right]^{4}\label{eq:SnuCGa}\\\begin{split}={}&\frac{ \sin (\nu\pi   ) \cos \frac{\nu\pi}{2} }{2^{\nu }\pi  }\left[ \frac{\Gamma\big(\frac{1+\nu}{2}\big)}{\Gamma\big(\frac{2+\nu}{2}\big)} \right]^{2}\sum_{k=0}^\infty\frac{\big(\frac{1+\nu}{2}\big)_k^2(-\nu)_k}{(k!)^{3}}\\{}&\times\left[ 2 \psi ^{(0)}\left(k+\frac{\nu+1 }{2}\right)+\psi ^{(0)}(k-\nu )-3 \psi ^{(0)}(k+1) \right].\end{split}\label{eq:SnuCGb}
\end{align}\end{proposition}
 \begin{proof}The last equality in \cite[(1.1)]{Zhou2013Int3Pnu} implies \eqref{eq:SnuCGa}. To show  \eqref{eq:SnuCGb}, apply   \eqref{eq:P+P+}--\eqref{eq:P+P-} to\begin{align}\begin{split}
\widetilde T_\nu={}&2\int_0^1\frac{[P_\nu(1-2t)]^{2}P_\nu(2t-1)\D t}{[4t(1-t)]^{\frac{1-\nu}{2}}}\\={}&\frac{2\sin(\nu\pi)}{\pi}\int_0^1\frac{P_\nu(1-2t)}{[4t(1-t)]^{\frac{1-\nu}{2}}}\sum_{k=0}^\infty[t(1-t)]^ks_{\nu,k}\left[\varsigma_{\nu,k}-\log\frac{1}{4t(1-t)}\right]\D t,\end{split}
\end{align} and then integrate termwise  using \eqref{eq:Pnu_avg}.
 \end{proof}

\subsection{Bailey--Meijer formulations of Clebsch--Gordan couplings\label{subsec:BM_CG}}In this subsection, we~give self-contained proofs of \eqref{eq:Snunu}, \eqref{eq:S2nu-1nu} and  \eqref{eq:SnuCGa}, without going through the heavy computations  in existent literature \cite{Zhou2013Pnu,Zhou2013Int3Pnu}.

A major tool in this subsection is Bailey's reduction \cite[(3.4)]{Bailey1932}  of a convergent and well-poised $ _7F_6 $ series into (what is currently known as) a special class of the Meijer $G $-function (cf.\  \cite[Proposition 2]{Zudilin2004odd_zeta} and \cite[Figure 3]{BSW2013}):\begin{align}
\begin{split}&
{_7F_6}\left(\left. \begin{array}{@{}r@{}r@{,{}}r@{,{}}r@{,{}}r@{,{}}r@{,{}}r@{}}
a,{}&1+\frac{a}{2}&b&c&d&e&f \\[4pt]
&\frac{a}{2}&1+a-b&1+a-c&1+a-d&1+a-e&1+a-f \\
\end{array} \right|1\right)\\={}&\frac{\Gamma (1+a-b) \Gamma (1+a-c) \Gamma (1+a-d) \Gamma (1+a-e) \Gamma (1+a-f)}{\Gamma (1+a) \Gamma (b) \Gamma (c) \Gamma (d) \Gamma (1+a-b-c) \Gamma (1+a-b-d) \Gamma (1+a-c-d) \Gamma (1+a-e-f)}\\{}&\times G_{4,4}^{2,4}\left(1\left|
\begin{array}{@{}c@{}}
 e+f-a,1-b,1-c,1-d \\[4pt]
 0,1+a-b-c-d,e-a,f-a \\
\end{array}
\right.\right).
\end{split}   \label{eq:Bailey7F6}
\end{align}Here, on the left-hand side of \eqref{eq:Bailey7F6}, hypergeometric parameters  sitting in the same column\footnote{In this article, we typeset each well-poised series with aligned columns for its hypergeometric parameters, following the practice of Slater \cite[\S2.4]{Slater}. } all add up to the same number $1+a$ (hence the name ``well-poised''); on the right-hand side of \eqref{eq:Bailey7F6}, the Meijer $G$-function is defined by an integral of Mellin--Barnes type:\begin{align}
G^{m,n}_{p,q}\left(z\left|
\begin{array}{@{}c@{}}
 a_{1},\dots ,a_p \\[4pt]
 b_{1},\dots,b_q \\
\end{array}
\right.\right):=\frac{1}{2\pi i}\resizebox{1.25\width}{1.25\height}{$\displaystyle\int$}_{\hspace{-0.8em}C}\frac{\prod_{j=1}^n\Gamma(1-a_{j}-s)\prod_{k=1}^m\Gamma(b_{k}+s)}{\prod_{j=n+1}^p\Gamma(a_{j}+s)\prod_{k=m+1}^q\Gamma(1-b_{k}-s)}\frac{\D s}{z^s},\label{eq:MeijerG_defn}
\end{align}where the canonical design of Barnes' contour $C$ ensures that   the right-hand side of \eqref{eq:MeijerG_defn} equals the sum over the residues of \begin{align}
-\frac{\prod_{j=1}^n\Gamma(1-a_{j}-s)\prod_{k=1}^m\Gamma(b_{k}+s)}{\prod_{j=n+1}^p\Gamma(a_{j}+s)\prod_{k=m+1}^q\Gamma(1-b_{k}-s)}\frac{1}{z^{s}}
\end{align}at all the poles in  $ \prod_{j=1}^n\Gamma(1-a_{j}-s)$, when the contour closes rightwards. By convention, all the empty products are treated as unity.

\begin{proposition}[Bailey--Meijer formulations of $ T_{\mu,\nu}$]\begin{enumerate}[leftmargin=*,  label=\emph{(\alph*)},ref=(\alph*),
widest=d, align=left] \item If  $ 2\mu\notin\mathbb Z$ and $2\nu\notin\mathbb Z$, then we have{\allowdisplaybreaks\begin{subequations}\begin{align}
T_{\mu,\nu}={}&\frac{2\sin (\mu\pi   )\cos (\nu\pi   )}{  \mu  (\mu +1) (2 \nu +1)\pi}{_7F_{6}}\left(\left.\begin{array}{@{}r@{}r@{,{}}r@{,{}}r@{,{}}r@{,{}}r@{,{}}r@{}}
\frac{1}{2},{}&\frac{5}{4}&\frac{1}{2}&-\frac{\mu }{2}&\frac{\mu +1}{2}&-\nu &\nu +1 \\[4pt]
&\frac{1}{4}&1&\frac{\mu +3}{2}&\frac{2-\mu }{2}&\frac{2\nu+3}{2}&\frac{1-2\nu}{2}  \\
\end{array}\right| 1\right)\label{eq:SmunuCG_BMa}\\={}{}&\frac{2\sin (\mu\pi   )\sin(\nu\pi   )}{ \mu  (\mu +1) \nu  (\nu +1)\pi ^2}{_6F_{5}}\left(\left.\begin{array}{@{}r@{}r@{,{}}r@{,{}}r@{,{}}r@{,{}}r@{}}
1,{}&\frac{3}{2}&\frac{1-\mu }{2}&\frac{\mu +2}{2}&-\nu &\nu +1\\[4pt]
&\frac{1}{2}&\frac{\mu +3}{2}&\frac{2-\mu }{2}&\nu +2&1-\nu  \\
\end{array}\right| 1\right);\label{eq:SmunuCG_BMb}
\intertext{if  $ 2\mu\notin\mathbb Z$, $2\nu\notin\mathbb Z$, and $ \R \nu>-1$, then we have}\begin{split}
T_{\mu,\nu}={}&\frac{2 (1-2 \nu )\sin (\mu\pi   )\sin(\nu\pi   )}{\mu  (\mu +1) \nu\pi ^2  }{_7F_{6}}\left(\left.\begin{array}{@{}r@{}r@{,{}}r@{,{}}r@{,{}}r@{,{}}r@{,{}}r@{}}
\frac{1-2\nu}{2},{}&\frac{5-2 \nu}{4}&1&\frac{1}{2} &\frac{\mu -2 \nu +1}{2} &-\frac{\mu+2\nu }{2} &-\nu\\[4pt]
&\frac{1-2 \nu }{4}&\frac{1-2\nu}{2}&1-\nu&\frac{2-\mu }{2}&\frac{\mu +3}{2} &\frac{3}{2} \\
\end{array}\right| 1\right)\\={}&\frac{2 (1-2 \nu )\sin (\mu\pi   )\sin(\nu\pi   )}{ \mu  (\mu +1) \nu\pi ^2 }{_6F_{5}}\left(\left.\begin{array}{@{}r@{}r@{,{}}r@{,{}}r@{,{}}r@{,{}}r@{}}
1,{}&\frac{5-2 \nu}{4} &\frac{1}{2}&\frac{\mu -2 \nu +1}{2} &-\frac{\mu+2\nu }{2} &-\nu\\[4pt]
&\frac{1-2 \nu }{4}&1-\nu&\frac{2-\mu }{2}&\frac{\mu +3}{2} &\frac{3}{2} \\
\end{array}\right| 1\right);\end{split}\label{eq:SmunuCG_BMc}
\intertext{if   $ 2\mu\notin\mathbb Z$, $2\nu\notin\mathbb Z$, and $ \R \nu<0$,  then we have}\begin{split}
T_{\mu,\nu}={}&-\frac{2 (2 \nu +3)\sin (\mu\pi   )\sin(\nu\pi   )}{\mu  (\mu +1) (\nu +1)\pi ^2 }{_7F_{6}}\left(\left.\begin{array}{@{}r@{}r@{,{}}r@{,{}}r@{,{}}r@{,{}}r@{,{}}r@{}}
\frac{2\nu+3}{2},{}&\frac{2 \nu +7}{4}&1&\frac{1}{2}&\frac{\mu+2\nu +3}{2} &\frac{2-\mu+2\nu }{2}&\nu +1 \\[4pt]&\frac{2\nu+3}{4}&\frac{2\nu+3}{2}&
\nu +2&\frac{2-\mu }{2} &\frac{\mu +3}{2}&\frac{3}{2} \\
\end{array}\right| 1\right)\\={}&-\frac{2 (2 \nu +3)\sin (\mu\pi   )\sin(\nu\pi   )}{ \mu  (\mu +1) (\nu +1)\pi ^2}{_6F_{5}}\left(\left.\begin{array}{@{}r@{}r@{,{}}r@{,{}}r@{,{}}r@{,{}}r@{}}
1,{}&\frac{2 \nu +7}{4}&\frac{1}{2}&\frac{\mu+2\nu +3}{2} &\frac{2-\mu+2\nu }{2}&\nu +1 \\[4pt]&\frac{2\nu+3}{4}&
\nu +2&\frac{2-\mu }{2} &\frac{\mu +3}{2}&\frac{3}{2} \\
\end{array}\right| 1\right).\end{split}\label{eq:SmunuCG_BMd}\intertext{Furthermore, for   $ \mu\notin\mathbb Z$ and $2\nu\notin\mathbb Z$,  we have }\begin{split}
T_{\mu,\nu}={}&\sin (\mu\pi   )\sin(\nu\pi   )\left(\frac{1}{\nu  (\nu +1)\pi ^2 }{_4F_{3}}\left(\left.\begin{array}{@{}c@{}}
1,1,\frac{1-\mu }{2},\frac{\mu+2 }{2} \\[4pt]
\frac{3}{2},1-\nu ,\nu +2 \\
\end{array}\right| 1\right)\right.\\{}&+\frac{\Gamma \big(\frac{\mu+1 }{2}\big) \Gamma \big(\!-\!\frac{\mu }{2}\big)}{\pi}\left\{\frac{ \tan (\nu\pi   )}{2 \nu +1 }\frac{ 1}{\Gamma \big(\frac{1-\mu }{2}\big) \Gamma \big(\frac{\mu+2 }{2}\big)}{_3F_2}\left(\left.\begin{array}{@{}c@{}}
\frac{1}{2},\frac{1+\mu }{2},-\frac{\mu }{2} \\[4pt]
\frac{1-2\nu}{2} ,\frac{2\nu+3}{2} \\
\end{array}\right| 1\right)\right.\\{}&-\frac{ \sin (  \nu\pi ) }{2^{2 \nu +1}\pi ^2 }\frac{[\Gamma (-\nu ) \Gamma (2 \nu +1)]^2}{\Gamma \big(\frac{2-\mu+2\nu }{2}\big) \Gamma \big(\frac{\mu+2\nu +3}{2} \big)}{_3F_2}\left(\left.\begin{array}{@{}c@{}}
-\frac{\mu+2\nu+1 }{2},\frac{\mu-2\nu }{2} ,-\nu \\[4pt]
\frac{1-2\nu}{2} ,-2\nu\ \\
\end{array}\right| 1\right)\\{}&-\left.\left.\frac{2^{2 \nu +1}\sin (  \nu\pi ) }{\pi ^2}\frac{[\Gamma (\nu +1) \Gamma (-2 \nu -1)]^2}{ \Gamma \big(\frac{\mu -2 \nu +1}{2} \big) \Gamma \big(\!-\!\frac{\mu+2\nu }{2} \big)}{_3F_2}\left(\left.\begin{array}{@{}c@{}}
\frac{1-\mu+2\nu }{2},\frac{2+\mu+2\nu }{2} ,\nu+1 \\[4pt]
\frac{2\nu+3}{2} ,2\nu+2\ \\
\end{array}\right| 1\right)\right\}\right).\end{split}\label{eq:SmunuCG_BMe}
\end{align}\end{subequations}}\item The expression \begin{align}
\varphi(\nu)\colonequals\frac{T_{\nu,\nu} -\frac{1+2\cos(\nu\pi )}{3}\frac{\pi\Gamma\left(\frac{\nu+1 }{2}\right) \Gamma\left(\frac{3 \nu+2 }{2}\right)}{\left[ \Gamma\left(\frac{1-\nu}{2}\right)\right]^2\left[ \Gamma\left(\frac{\nu+2 }{2}\right)\right]^3 \Gamma\left(\frac{3 \nu+3 }{2}\right)}}{[1+2\cos(\nu\pi )]\cos\frac{\nu\pi}{2}}\label{eq:varphi_defn}
\end{align} extends to a    holomorphic function for $ \R\nu\geq0$ that vanishes for every $\nu\in\mathbb Z_{\geq0} $ and has at most $ O\big(e^{\pi|\nu|/2}\big)$ growth rate for large $|\nu|$, hence $\varphi(\nu) \equiv0$.\item For $\nu\in(-1,0)\smallsetminus\left\{ -\frac{1}{2} \right\} $, we have {\allowdisplaybreaks\begin{subequations}\begin{align}T_{2\nu-1,\nu}={}&\frac{\sin(\nu\pi )\sin(2\nu\pi )}{\nu^{2}\pi^{2}}\frac{\nu \pi\cot(\nu\pi)}{(1-2\nu)(1+2\nu)}{_5F_{4}}\left(\left.\begin{array}{@{}r@{}r@{,{}}r@{,{}}r@{,{}}r@{}}
 \frac12,{}&\frac54&\frac12&-\nu&\nu\\[4pt]&\frac14&1&\frac{3}{2}+\nu&\frac{3}{2}-\nu
 \\
\end{array}\right| 1\right)\label{eq:CG2nuFa}\\={}&\frac{\sin(\nu\pi )\sin(2\nu\pi )}{\nu^{2}\pi^{2}}\frac{1}{(1+\nu)(1-2\nu)}{_5F_{4}}\left(\left.\begin{array}{@{}r@{}r@{,{}}r@{,{}}r@{,{}}r@{}}
 1,{}&\frac32&1&-\nu&\frac{1}{2}+\nu\\[4pt]&\frac12&1&2+\nu&\frac{3}{2}-\nu
 \\
\end{array}\right| 1\right)\label{eq:CG2nuFb}\\={}&\frac{\sin(\nu\pi )\sin(2\nu\pi )}{\nu^{2}\pi^{2}}\label{eq:CG2nuFc}\\={}&\frac{\sin(\nu\pi )\sin(2\nu\pi )}{\nu^{2}\pi^{2}}\frac{\nu  (2 \nu +3)}{(\nu+1)(2\nu-1)}{_5F_{4}}\left(\left.\begin{array}{@{}r@{}r@{,{}}r@{,{}}r@{,{}}r@{}}
 \frac{2\nu+3}{2},{}&\frac{2\nu+7}{4}&1&\frac{1}{2}&1+2\nu\\[4pt]&\frac{2\nu+3}{4}&\frac{2\nu+3}{2}&2+\nu&\frac{3}{2}-\nu
 \\
\end{array}\right| 1\right)\label{eq:CG2nuFd}.\end{align}\end{subequations}}\end{enumerate}\end{proposition}\begin{proof}\begin{enumerate}[leftmargin=*,  label={(\alph*)},ref=(\alph*),
widest=d, align=left] \item We can rewrite \eqref{eq:SmunuCGc} as\begin{align}\begin{split}T_{\mu,\nu}={}&\frac{2 \sin (\mu \pi  ) \sin (\nu\pi   )}{ \mu  (\mu +1)\pi ^2}\sum_{k=0}^\infty\frac{\big(\frac{1}{2}\big)_k (-\nu )_k (\nu +1)_k}{ \big(\frac{2-\mu }{2}\big)_k \big(\frac{\mu +3}{2}\big)_kk!}\\{}&{}\times\left[ -\psi ^{(0)}\left(k+\frac{\mu+3 }{2}\right)-\psi ^{(0)}\left(k-\frac{\mu-2 }{2}\right)+\psi ^{(0)} (k+\nu +1)\right.\\{}&\left.{}+\psi ^{(0)} (k-\nu )-\psi ^{(0)} (k+1)+\psi ^{(0)} \left(k+\frac{1}{2}\right) \right].\end{split}\end{align}For $ \nu\in(-1,0)$, this is equal to \begin{align}\begin{split}&-\frac{ \Gamma \big(\frac{\mu+1 }{2}\big) \Gamma \big(\!-\!\frac{\mu }{2}\big) \sin (\mu\pi   )\sin ^2(\nu\pi   )}{2 \pi ^{3}\sqrt{\pi}}\lim_{\varepsilon\to0^+}\int_{\varepsilon-i\infty}^{\varepsilon+i\infty}g_{\mu,\nu}(s)\frac{\D s}{2\pi i}\\={}&
-\frac{ \Gamma \big(\frac{\mu+1 }{2}\big) \Gamma \big(\!-\!\frac{\mu }{2}\big) \sin (\mu\pi   )\sin ^2(\nu\pi   )}{2 \pi ^{3}\sqrt{\pi}}G_{4,4}^{2,4}\left(1\left|
\begin{array}{@{}c@{}}
 \frac{1}{2},0,\nu +1,-\nu  \\[4pt]
 0,0,\frac{\mu }{2},-\frac{\mu+1}{2}  \\
\end{array}
\right.\right)\end{split}\label{eq:gmunu}
\end{align} where \begin{align}g_{\mu,\nu}(s)\colonequals
\frac{[\Gamma (s)]^2 \Gamma \big(\frac{1}{2}-s\big) \Gamma (1-s) \Gamma (-\nu -s) \Gamma (\nu +1-s)}{\Gamma \big(\frac{2-\mu }{2}-s\big) \Gamma \big(\frac{\mu +3}{2}-s\big)}
\end{align}and the contour closes leftwards.

Due to  the invariance of the generalized hypergeometric series and the Meijer  $G$-functions under certain permutations of their parameters, there are four (ostensibly) different outcomes after we apply Bailey's reduction formula \eqref{eq:Bailey7F6} to \eqref{eq:gmunu}, as listed in \eqref{eq:SmunuCG_BMa}--\eqref{eq:SmunuCG_BMd}.  One can analytically continue these identities beyond the regime where  $ \nu\in(-1,0)$,  so  long as the generalized hypergeometric series in question remain convergent, according to the condition set forth in \eqref{eq:pFq_ind_ineq}.

Alternatively, we may close the contour rightwards in \eqref{eq:gmunu}, after which residue calculus brings us  \eqref{eq:SmunuCG_BMe} for $ \nu\in(-1,0)$. Subsequent analytic continuation confirms \eqref{eq:SmunuCG_BMe}  in its entirety.

       \item By Dougall's theorem (see \cite[(2.3.4.1)]{Slater} or \cite[(2.2.9)]{AAR}), the well-poised series in  \eqref{eq:SmunuCG_BMa} reduces to a ratio of gamma values when $ -\frac{\mu }{2},\frac{\mu +1}{2},-\nu ,\nu +1\in\mathbb Z_{\leq0}$. These are precisely the cases of \cite[(19$ _{(\mu,n)}$) and (19$ _{(2m,\nu)}$)]{Zhou2013Pnu}, which were previously derived by totally different means.

In particular, the last paragraph allows us to confirm that $ \varphi(\nu)=0$ for all $ \nu\in2\mathbb Z_{\geq0}$.

To show that $ \varphi(2n+1)=0$ for  $ n\in\mathbb Z_{\geq0}$, we need to verify that the numerator in  \eqref{eq:varphi_defn} vanishes to the order of $ O((\nu-2n-1)^2)$ as $\nu\to2n+1$. This can be done via a particular case of   \eqref{eq:SmunuCG_BMb}: \begin{align}\begin{split}
T_{\nu,\nu}={}&\frac{2\sin^{2}(\nu\pi   )}{  \nu^{2}  (\nu +1)^{2}\pi ^2}{_6F_{5}}\left(\left.\begin{array}{@{}r@{}r@{,{}}r@{,{}}r@{,{}}r@{,{}}r@{}}
1,{}&\frac{3}{2}&\frac{1-\nu }{2}&\frac{\nu+2}{2}&-\nu &\nu +1\\[4pt]
&\frac{1}{2}&\frac{\nu +3}{2}&\frac{2-\nu }{2}&\nu +2&1-\nu  \\
\end{array}\right| 1\right)\\={}&\frac{2\sin^{2}(\nu\pi   )}{\nu^{2}  (\nu +1)^{2}\pi ^2}\sum_{k=0}^\infty\frac{\big(\frac32\big)_k\big(\frac{1-\nu }{2}\big)_k\big(\frac{\nu+2 }{2}\big)_k}{\big(\frac12\big)_k\big(\frac{\nu+3 }{2}\big)_k\big(\frac{2-\nu }{2}\big)_k}\frac{\nu(\nu+1)}{(\nu-k)(\nu+k+1)},\end{split}\tag{\ref{eq:SmunuCG_BMb}$'$}\label{eq:SmunuCG_BMb'}
\end{align}   where the infinite series terminates to a bounded quantity, as  $\nu$ approaches an  odd positive integer.

To demonstrate that $ \varphi(\nu)$ has a holomorphic  extension for $ \R\nu\geq0$, we still have to check that the numerator of   \eqref{eq:varphi_defn} vanishes when $ \nu=2n+1\pm\frac13$ for   $ n\in\mathbb Z_{\geq0}$.  For this purpose, we specialize    \eqref{eq:SmunuCG_BMc}  to two sequences of terminating  series \begin{align}T_{{2n+\frac{4}{3}},2n+\frac{4}{3}}={}&-\frac{27 (12 n+5){_6F_{5}}\left(\left.\begin{array}{@{}r@{}r@{,{}}r@{,{}}r@{,{}}r@{,{}}r@{}}
1,&\frac{1}{2}&-n+\frac{7}{12} &-n-\frac{1}{6} &-3n-2 &-2n-\frac{4}{3}\\[4pt]
&-2n-\frac{1}{3}&-n-\frac{5}{12}&-n+\frac{1}{3}&n+\frac{13}{6} &\frac{3}{2} \\
\end{array}\right| 1\right)}{8 \pi ^2 (3 n+2)^2 (6 n+7)},
\tag{\ref{eq:SmunuCG_BMc}$'_+$}\\T_{{2n+\frac{2}{3}},2n+\frac{2}{3}}={}&-\frac{27 (12 n+1){_6F_{5}}\left(\left.\begin{array}{@{}r@{}r@{,{}}r@{,{}}r@{,{}}r@{,{}}r@{}}
1,&\frac{1}{2}&-n+\frac{11}{12} &-n+\frac{1}{6} &-3n-1 &-2n-\frac{2}{3}\\[4pt]
&-2n+\frac{1}{3}&-n-\frac{1}{12}&-n+\frac{2}{3}&n+\frac{11}{6} &\frac{3}{2} \\
\end{array}\right| 1\right)}{8 \pi ^2 (3 n+1)^2 (6 n+5)},\tag{\ref{eq:SmunuCG_BMc}$'_-$}
\end{align} which can be computed   through the recursions\footnote{These recursions can be generated automatically from the \texttt{Annihilator} command in Christoph Koutschan's \texttt{HolonomicFunctions} package \cite{Koutschan2013} for \textit{Mathematica}, which implements the Wilf--Zeilberger algorithm \cite{WZ1990,Zeilberger1990a,Zeilberger1991}. It turns out that both recursions in \eqref{eq:SCG_rec} are valid for non-integer values of $n$\ as well [cf.\ \eqref{eq:Snunu}].}\begin{align}
\left\{\begin{array}{@{}r@{{}={}}l}T_{{2(n+1)+\frac{4}{3}},2(n+1)+\frac{4}{3}}&\smash[t]{\frac{(n+1) (3 n+4) (6 n+7)^2}{(2 n+3) (3 n+5)^2 (6 n+11)} }T_{{2n+\frac{4}{3}},2n+\frac{4}{3}}, \\[8pt]
 \smash[b]{T_{{2(n+1)+\frac{2}{3}},2(n+1)+\frac{2}{3}}}&\smash[b]{\frac{(n+1) (3 n+2) (6 n+5)^2}{(2 n+3) (3 n+4)^2 (6 n+7)} T_{{2n+\frac{2}{3}},2n+\frac{2}{3}}},
\end{array}
\right.\label{eq:SCG_rec}\end{align}   with initial conditions\begin{align}
\left\{\begin{array}{@{}r@{{}={}}l}{T_{{\frac{4}{3}},\frac{4}{3}}}&0,\\[8pt]
\smash[b]{T_{{\frac{2}{3}},\frac{2}{3}}}&0.  \\
\end{array}
\right.\end{align}Thus far, we have verified that  $ \varphi(\nu)$  remains finite as  $ \nu\to2n+1\pm\frac13$ for   $ n\in\mathbb Z_{\geq0}$.

To prove that  $ \varphi(\nu)$ has at most  $ O\big(e^{\pi|\nu|/2}\big)$ growth rate for large $|\nu|$, apply Stirling's formula for Euler's gamma function to the right-hand sides of \eqref{eq:Snunu} and \eqref{eq:SmunuCG_BMb'}.

By Carlson's theorem \cite[Theorem 2.8.1]{AAR}, we conclude that $\varphi(\nu) $ is identically vanishing.     \item Setting $\mu=2\nu-1$ in \eqref{eq:SmunuCG_BMa}--\eqref{eq:SmunuCG_BMd}, we arrive at \eqref{eq:CG2nuFa}--\eqref{eq:CG2nuFd}.

For \eqref{eq:CG2nuFc}, the hypergeometric series terminates after the leading constant term. [Plugging $ \mu=2\nu-1$ into \eqref{eq:SmunuCG_BMe},  while employing some standard reduction formulae for $ _3F_2$ and $_2F_1$, one  gets the same result as \eqref{eq:CG2nuFc}.]

One may also turn \eqref{eq:CG2nuFa}, \eqref{eq:CG2nuFb}, and  \eqref{eq:CG2nuFd}  into \eqref{eq:CG2nuFc} by a reduction formula (see \cite[(2.3.4.5)]{Slater} or \cite[Corollary 3.5.2]{AAR}) \begin{align}
\begin{split}&{_5F_{4}}\left(\left.\begin{array}{@{}r@{}r@{,{}}r@{,{}}r@{,{}}r@{}}
 a,{}&1+\frac{a}{2}&b&c&d\\[4pt]&\frac{a}{2}&1+a-b&1+a-c&1+a-d \\
\end{array}\right| 1\right)\\={}&\frac{\Gamma(1+a-b)\Gamma(1+a-c)\Gamma(1+a-d)\Gamma(1+a-b-c-d)}{\Gamma(1+a)\Gamma(1+a-b-c)\Gamma(1+a-b-d)\Gamma(1+a-c-d)}\end{split}\label{eq:5F4red}
\end{align}for a special class of convergent and well-poised $_5F_4$.
   \qedhere\end{enumerate}  \end{proof}
\begin{proposition}[Bailey--Meijer formulations of $ \widetilde T_\nu$]For $ \R\nu>-1$ and $ \nu\notin\mathbb Z$, the following identities hold true:\begin{align}
\widetilde T_\nu={}&\frac{1}{\pi}\left(\frac{\cos\frac{\nu\pi}{2}}{2^{\nu}}\right)^3\left[ \frac{\Gamma\big(\frac{1+\nu}{2}\big)}{\Gamma\big(\frac{2+\nu}{2}\big)} \right]^{2}\frac{2^{2 \nu +1} (1-\nu )}{\nu  (\nu +1) \cot\frac{\nu\pi   }{2}}{_5F_{4}}\left(\left.\begin{array}{@{}r@{}r@{,{}}r@{,{}}r@{,{}}r@{}}
 \frac{1-\nu}{2},{}&\frac{5-\nu }{4}&\frac{1-\nu}{2}&-\nu&\frac{1+\nu}{2}\\[4pt]&\frac{1-\nu}{4}&1&\frac{3+\nu}{2}&1-\nu
 \\
\end{array}\right| 1\right)\label{eq:Snu5F4}\\={}&\frac{1}{\pi}\left(\frac{\cos\frac{\nu\pi}{2}}{2^{\nu}}\right)^3\left[ \frac{\Gamma\big(\frac{1+\nu}{2}\big)}{\Gamma\big(\frac{2+\nu}{2}\big)} \right]^{4}.\label{eq:Snu5F4red}
\end{align}\end{proposition}
\begin{proof}For $\nu\in(-1,0)$, one may convert \eqref{eq:SnuCGb} into\begin{align}\begin{split}
 \widetilde T_\nu={}&-\frac{ \sin (\pi  \nu ) \cos \frac{\nu\pi   }{2}}{2^{\nu }\pi  \big[\Gamma \big(\frac{2+\nu }{2}\big)\big]^2 \Gamma (-\nu )}\lim_{\varepsilon\to0^+}\int_{\varepsilon-i\infty}^{\varepsilon+i\infty}\widetilde g_\nu(s)\frac{\D s}{2\pi i}\\={}&-\frac{ \sin (\pi  \nu ) \cos \frac{\nu\pi   }{2}}{2^{\nu }\pi  \big[\Gamma \big(\frac{2+\nu }{2}\big)\big]^2 \Gamma (-\nu )}G_{4,4}^{2,4}\left(1\left|
\begin{array}{@{}c@{}}
 0,\frac{1-\nu }{2},\frac{1-\nu }{2},\nu +1 \\[4pt]
 0,0,0,0
\end{array}
\right.\right)\\={}&-\frac{ \sin (\pi  \nu ) \cos \frac{\nu\pi   }{2}}{2^{\nu }\pi  \big[\Gamma \big(\frac{2+\nu }{2}\big)\big]^2 \Gamma (-\nu )}G_{3,3}^{2,3}\left(1\left|
\begin{array}{@{}c@{}}
 \frac{1-\nu }{2},\frac{1-\nu }{2},\nu +1 \\[4pt]
 0,0,0
\end{array}
\right.\right),\end{split}\label{eq:MeijerSnuCG}
\end{align}where\begin{align}
\widetilde g_\nu(s)\colonequals\frac{[\Gamma (s)]^2 \Gamma (-s-\nu )\big[ \Gamma \big(\frac{\nu +1}{2}-s\big)\big]^2}{\Gamma (1-s)}.
\end{align}Applying Bailey's reduction \eqref{eq:Bailey7F6} to the second line in \eqref{eq:MeijerSnuCG}, we receive two (ostensibly) different outputs. One of them is given in \eqref{eq:Snu5F4},  where the  $_5F_4 $ series converges for  $\R\nu>-1$; the other is
\begin{align}
-\frac{4 \sin (\nu\pi   ) \cos \frac{\nu\pi}{2}   }{\pi ^{3/2} }\frac{\big[\Gamma \big(\frac{\nu+1 }{2}\big)\big]^3}{\Gamma\big(\frac{2+\nu}{2}\big)\Gamma (\nu +2)}
{_5F_{4}}\left(\left.\begin{array}{@{}r@{}r@{,{}}r@{,{}}r@{,{}}r@{}}
 1+\nu,{}&\frac{3+\nu}{2}&\frac{1+\nu}{2}&1+\nu&\frac{1+\nu}{2}\\[4pt]&\frac{1+\nu}{2}&\frac{3+\nu}{2}&1&\frac{3+\nu}{2}
 \\
\end{array}\right| 1\right),\tag{\ref{eq:Snu5F4}$'$}\label{eq:Snu5F4'}
\end{align}
where the ${_5F_4} $ series converges for  $\R\nu<0$.
In both \eqref{eq:Snu5F4} and \eqref{eq:Snu5F4'}, the well-poised
 ${_5F_4} $ series are reducible by \eqref{eq:5F4red}. This proves \eqref{eq:Snu5F4red}.   \end{proof}
\subsection{Generalizations to associated Legendre functions\label{subsec:Pmnnu}}
\begin{lemma}[Clebsch--Gordan coupling of associated Legendre functions]When $ 2\varepsilon\notin\mathbb Z_{<0}$ and $ \R(\beta+\varepsilon)>\frac{|\R\delta|}{2}-1$, we have \begin{align}\begin{split}&
\int_0^1 \big(1-x^{2}\big)^\beta\left[P_\mu^\delta(x)+P_\mu^\delta(-x)\right] \left[P_\nu^{-\varepsilon}(x)\right]^2\D x\\={}&\frac{  2^{\delta -2 \varepsilon } \pi\Gamma \big(1+\beta -\frac{\delta }{2}+\varepsilon \big) \Gamma \big(1+\beta +\frac{\delta }{2}+\varepsilon\big)  }{[\Gamma (1+\varepsilon )]^2 \Gamma \big(\frac{1-\delta-\mu}{2}\big) \Gamma \big(\frac{2-\delta+\mu }{2}\big) \Gamma \big(\beta +\varepsilon -\frac{\mu-2 }{2}\big) \Gamma \big(\beta +\varepsilon +\frac{\mu+3 }{2}\big)}\\{}&\times {_5F_4}\left(\left.
\begin{array}{@{}c@{}}\frac{1}{2}+\varepsilon ,1+\beta -\frac{\delta }{2}+\varepsilon ,1+\beta +\frac{\delta }{2}+\varepsilon ,\varepsilon -\nu ,1+\varepsilon +\nu \\[4pt]1+\varepsilon,1+2 \varepsilon ,\beta +\varepsilon -\frac{\mu-2 }{2},\beta +\varepsilon +\frac{\mu+3 }{2}\end{array}\right|1\right).\end{split}\label{eq:PmunuCGa}
\end{align}When $ \varepsilon\notin\mathbb Z$ and $ \R\beta>\frac{|\R\delta|}{2}-1$, we have \begin{align}\begin{split}&
\int_0^1 \big(1-x^{2}\big)^\beta\left[P_\mu^\delta(x)+P_\mu^\delta(-x)\right] P_\nu^\varepsilon(x)P_\nu^{-\varepsilon}(x)\D x\\={}&\frac{\sin(\varepsilon\pi)}{\varepsilon\pi}\frac{2^{\delta } \pi \Gamma \big(1+\beta -\frac{\delta }{2} \big) \Gamma \big(1+\beta +\frac{\delta }{2} \big)}{\Gamma \big(\frac{1-\delta-\mu}{2}\big) \Gamma \big(\frac{2-\delta+\mu }{2}\big) \Gamma \big(\beta  -\frac{\mu-2 }{2}\big) \Gamma \big(\beta  +\frac{\mu+3 }{2}\big)} {_5F_4}\left(\left.
\begin{array}{@{}c@{}}\frac{1}{2},1+\beta -\frac{\delta }{2} ,1+\beta +\frac{\delta }{2},-\nu ,1+\nu \\[4pt]1-\varepsilon,1+\varepsilon,\beta  -\frac{\mu-2 }{2},\beta +\frac{\mu+3 }{2}\end{array}\right|1\right).\end{split}\label{eq:PmunuCGb}
\end{align}\end{lemma}\begin{proof}Instead of \eqref{eq:Pnu_avg}, we need   a  more general formula (see \cite[item 7.132.1]{GradshteynRyzhik}, which corrects \cite[\S18.1(16)]{ET2}) \begin{align}\begin{split}&
\int_0^1 \big(1-x^{2}\big)^\sigma\left[P_\mu^\delta(x)+P_\mu^\delta(-x)\right]\D x\\={}&\frac{2^{\delta}\pi\Gamma\big(1+\sigma+\frac{\delta}{2}\big)\Gamma\big(1+\sigma-\frac{\delta}{2}\big)}{\Gamma\big(\frac{\mu +3}{2}+\sigma\big)\Gamma\big(\sigma -\frac{\mu-2}{2}\big)\Gamma\big(\frac{2+\mu-\delta}{2}\big)\Gamma\big(\frac{1-\mu-\delta }{2}\big)},\quad \R \sigma>\frac{|\R\delta|}{2}-1.\end{split}\label{eq:Pnu_avg'}\tag{\ref{eq:Pnu_avg}$'$}
\end{align} Joining this integral identity with \eqref{eq:Clausen} and \eqref{eq:Clausen'}, while switching the order of summation and integration, we obtain  \eqref{eq:PmunuCGa} and  \eqref{eq:PmunuCGb}. \end{proof}
\begin{remark}In \eqref{eq:PmunuCGa}, upon trigonometric substitution to the scenario where $ \beta=\nu=-\frac12$ and  $\delta=\varepsilon=\mu=0$, we have \begin{align}
\int_0^{\pi/2}\left[\mathbf K\left( \sin\frac{\theta}{2} \right)\right]^2\D      \theta={}&\frac{\pi^{3}}{8}{_4F_3}\left( \left.\begin{array}{@{}c@{}}
\frac12, \frac12,\frac12, \frac12\\[4pt]
1,1,1
\end{array}\right|1 \right).
\intertext{One may compare this to}
\int_0^{\pi/2}\left[\mathbf K\left( \sin\theta \right)\right]^2\D      \theta={}&\frac{\pi^{3}}{4}{_4F_3}\left( \left.\begin{array}{@{}c@{}}
\frac12, \frac12,\frac12, \frac12\\[4pt]
1,1,1
\end{array}\right|1 \right),
\end{align} a formula proved by Wan \cite[(35)]{Wan2012}. This specific $_4F_3 $ series is related to a critical $L$-value (see \cite[32]{RogersWanZucker2015} or \cite[\S7]{Zagier2018ATDE}) of a weight-4 cusp form attached to the rigid Calabi--Yau threefold  $ \sum_{j=1}^4\big(x_j+\frac{1}{x_j}\big)=0$ \cite[\S2]{Zudilin2018hypCY}.    \eor\end{remark}

Now, we are ready to complete the proof of Theorem \ref{thm:TU} by putting a finishing touch to \eqref{eq:SigmaCG4F3''}.
\begin{proof}[Proof of \eqref{eq:SigmaCG4F3''}] With \begin{align}
a_{\mu}(\varepsilon)\colonequals{}&\frac{2^{2 \varepsilon +1} [\sin (\mu\pi   )-\sin (2\varepsilon \pi   )]}{\pi ^{3/2} (\mu-2 \varepsilon  ) (\mu +2 \varepsilon +1) }\frac{\Gamma \big(\frac{1}{2}+\varepsilon\big) }{\Gamma (1+\varepsilon)},\\b_{\mu}(\varepsilon)\colonequals{}&\frac{2^{2\varepsilon}\pi  }{\Gamma \big(\frac{2-\mu }{2}\big) \Gamma \big(\frac{\mu +3}{2}\big) \Gamma \big(\frac{1-\mu-2\varepsilon}{2} \big) \Gamma \big(\frac{\mu+2-2\varepsilon}{2} \big)},
\end{align} we consider the following particular cases of   \eqref{eq:PmunuCGa}  and  \eqref{eq:PmunuCGb}:\begin{align}
\int_0^1 \left[P_\mu^{2\varepsilon}(x)+P_\mu^{2\varepsilon}(-x)\right] \left[P_\nu^{-\varepsilon}(x)\right]^2\D x={}&a_{\mu}(\varepsilon){_4F_3}\left(\left.\begin{array}{@{}c@{}}
1,\frac{1}{2}+\varepsilon,-\nu+\varepsilon,\nu+1+\varepsilon \\[4pt]
1+\varepsilon,\frac{2-\mu }{2}+\varepsilon,\frac{\mu+3 }{2}+\varepsilon\ \\
\end{array}\right| 1\right),\tag{\ref{eq:PmunuCGa}$'$}\\\int_0^1 \left[P_\mu^{2\varepsilon}(x)+P_\mu^{2\varepsilon}(-x)\right] P_\nu^\varepsilon(x)P_\nu^{-\varepsilon}(x)\D x={}&b_{\mu}(\varepsilon){_3F_2}\left(\left.\begin{array}{@{}c@{}}
\frac{1}{2},-\nu ,\nu+1\ \\[4pt]
\frac{2-\mu }{2},\frac{\mu+3 }{2} \\
\end{array}\right| 1\right),
\tag{\ref{eq:PmunuCGb}$'$}\label{eq:PmunuCGb'}
\end{align} whose right-hand sides involve two generalized hypergeometric series that showed up previously in \eqref{eq:SmunuCGc}.

Simplifying  \begin{align} \lim_{\varepsilon\to0}\frac{\partial^{2}}{\partial\varepsilon^{2}}
\int_0^1 \left[P_\mu^{2\varepsilon}(x)+P_\mu^{2\varepsilon}(-x)\right]\left\{ \frac{\left[P_\nu^{-\varepsilon}(x)\right]^2}{a_{\mu}(\varepsilon)}\frac{c_{\mu,\nu}(\varepsilon)}{c_{\mu,\nu}(0)}-\frac{P_\nu^\varepsilon(x)P_\nu^{-\varepsilon}(x)}{b_{\mu}(\varepsilon)}\right\}\D x
\end{align}with the assistance from the first-order derivative of any
 associated Legendre function with respect to its order  (cf.\ \cite[(3)]{Brychkov2010Pnumu}) \begin{align}
\left.\frac{\partial}{\partial\varepsilon}\right|_{\varepsilon=0}P_\nu^\varepsilon(x)=\frac{\psi ^{(0)}(\nu +1)+\psi ^{(0)}(-\nu )}{2}  P_{\nu }(x )-\frac{\pi  P_{\nu }(-x )}{2 \sin (\nu\pi   )},
\end{align} and a pair of digamma identities{\begin{align}
\psi ^{(0)}\left(\frac{\mu+3 }{2}\right)={}&\frac{2}{\mu +1}+\frac{\pi}{\tan \frac{\mu\pi  }{2}}  -\psi ^{(0)}\left(\frac{2-\mu }{2}\right)+2 \psi ^{(0)}(\mu )-2\log 2,\\\psi ^{(0)}\left(\frac{1-\mu }{2}\right)={}&\frac{2}{\mu}-\pi  \tan \frac{\mu\pi  }{2}-\psi ^{(0)}\left(\frac{\mu+2 }{2}\right)+2 \psi ^{(0)}(\mu )-2\log2,
\end{align}}we get \eqref{eq:SigmaCG4F3''}.    \end{proof}

\subsection*{Acknowledgments}A considerable amount of the materials developed in this work were based on my unpublished notes prepared at Princeton University during 2011--2013, for Prof.\ Weinan E's working seminar on mathematical physics.

This paper is a write-up of my talk ``Sun's series at vanishing and positive genera'' presented during the 2nd Workshop on Supercongruences,  $\pi$ Series, and Related Fields,  at Nanjing University on Nov.\ 17--19, 2023. I thank the organizers of this workshop for their hospitality and all the participants for their thoughtful feedback during and after this event.

I am indebted to two anonymous referees, whose  thoughtful feedback helped me improve the presentation of this note.


\begin{thebibliography}{1000}

\bibitem{Ablinger2017}
J.~Ablinger.
\newblock Discovering and proving infinite binomial sums identities.
\newblock {\em Exp. Math.}, 26(1):62--71, 2017.

\bibitem{Ablinger2022}
J.~Ablinger.
\newblock Discovering and proving infinite {P}ochhammer sum identities.
\newblock {\em Exp. Math.}, 31(1):309--323, 2022.

\bibitem{ABRS2014}
J.~Ablinger, J.~Bl\"{u}mlein, C.~G. Raab, and C.~Schneider.
\newblock Iterated binomial sums and their associated iterated integrals.
\newblock {\em J. Math. Phys.}, 55(11):112301, 57 pages, 2014.

\bibitem{ABS2011}
J.~Ablinger, J.~Bl\"{u}mlein, and C.~Schneider.
\newblock Harmonic sums and polylogarithms generated by cyclotomic polynomials.
\newblock {\em J. Math. Phys.}, 52(10):102301, 52 pages, 2011.

\bibitem{ABS2013}
J.~Ablinger, J.~Bl\"{u}mlein, and C.~Schneider.
\newblock Analytic and algorithmic aspects of generalized harmonic sums and
  polylogarithms.
\newblock {\em J. Math. Phys.}, 54(8):082301, 74 pages, 2013.

\bibitem{AZ1998}
T.~Amdeberhan and D.~Zeilberger.
\newblock Hypergeometric series acceleration via the {WZ} method, 1998.
\newblock \url{arXiv:math/9804121} [math.CO].

\bibitem{AAR}
G.~E. Andrews, R.~Askey, and R.~Roy.
\newblock {\em Special Functions}, volume~71 of {\em Encyclopedia of
  Mathematics and Its Applications}.
\newblock Cambridge University Press, Cambridge, UK, 1999.

\bibitem{Appell1881}
P.~Appell.
\newblock \selectlanguage{french}{M}\'emoire sur les \'equations
  diff\'erentielles lin\'eaires\selectlanguage{english}.
\newblock {\em Ann. sci. \'Ec. Norm. Sup\'er.}, 10:391--424, 1881.

\bibitem{Au2025a}
K.~C. Au.
\newblock Multiple zeta values, {WZ}-pairs and infinite sums computations.
\newblock {\em Ramanujan J.}, 66(1):Paper No. 3, 35, 2025.

\bibitem{Au2025b}
K.~C. Au.
\newblock Wilf-{Z}eilberger seeds and non-trivial hypergeometric identities.
\newblock {\em J. Symbolic Comput.}, 130:Paper No. 102421, 34, 2025.

\bibitem{Bailey1932}
W.~N. Bailey.
\newblock Some transformations of generalized hypergeometric series, and
  contour integrals of {B}arnes's type.
\newblock {\em Quart. J. Math.}, 3:168--182, 1932.

\bibitem{HTF1}
H.~Bateman.
\newblock {\em Higher Transcendental Functions}, volume~I.
\newblock McGraw-Hill, New York, NY, 1953.
\newblock (compiled by staff of the Bateman Manuscript Project: Arthur
  Erd{\'e}lyi, Wilhelm Magnus, Fritz Oberhettinger, Francesco G. Tricomi, David
  Bertin, W. B. Fulks, A. R. Harvey, D. L. Thomsen, Jr., Maria A. Weber and E.
  L. Whitney).

\bibitem{ET2}
H.~Bateman.
\newblock {\em Table of Integral Transforms}, volume~II.
\newblock McGraw-Hill, New York, NY, 1954.
\newblock (compiled by staff of the Bateman Manuscript Project: Arthur
  Erd{\'e}lyi, Wilhelm Magnus, Fritz Oberhettinger, Francesco G. Tricomi, David
  Bertin, W. B. Fulks, A. R. Harvey, D. L. Thomsen, Jr., Maria A. Weber and E.
  L. Whitney).

\bibitem{RN2}
B.~C. Berndt.
\newblock {\em Ramanujan's Notebooks (Part II)}.
\newblock Springer-Verlag, New York, NY, 1989.

\bibitem{RN5}
B.~C. Berndt.
\newblock {\em Ramanujan's Notebooks (Part V)}.
\newblock Springer-Verlag, New York, NY, 1998.

\bibitem{BorweinBroadhurstKamnitzer2001}
J.~M. Borwein, D.~J. Broadhurst, and J.~Kamnitzer.
\newblock Central binomial sums, multiple {C}lausen values, and zeta values.
\newblock {\em Experiment. Math.}, 10(1):25--34, 2001.

\bibitem{LatticeSum2013}
J.~M. Borwein, M.~L. Glasser, R.~C. McPhedran, J.~G. Wan, and I.~J. Zucker.
\newblock {\em Lattice Sums: Then and Now}, volume 150 of {\em Encyclopedia of
  Mathematics and Its Applications}.
\newblock Cambridge University Press, 2013.

\bibitem{BSW2013}
J.~M. Borwein, A.~Straub, and J.~Wan.
\newblock Three-step and four-step random walk integrals.
\newblock {\em Exp. Math.}, 22(1):1--14, 2013.

\bibitem{Broadhurst1996}
D.~J. Broadhurst.
\newblock On the enumeration of irreducible $k$-fold {E}uler sums and their
  roles in knot theory and field theory.
\newblock \url{arXiv:hep-th/9604128}, 1996.

\bibitem{Broadhurst1999}
D.~J. Broadhurst.
\newblock Massive 3-loop {F}eynman diagrams reducible to {$\rm SC^*$}
  primitives of algebras of the sixth root of unity.
\newblock {\em Eur. Phys. J. C Part. Fields}, 8(2):313--333, 1999.

\bibitem{Brown2009a}
F.~Brown.
\newblock The massless higher-loop two-point function.
\newblock {\em Comm. Math. Phys.}, 287(3):925--958, 2009.
\newblock \url{arXiv:0804.1660v1} [math.AG].

\bibitem{Brown2009arXiv}
F.~Brown.
\newblock On the periods of some {F}eynman integrals.
\newblock \url{arXiv:0910.0114v2} [math.AG], 2010.

\bibitem{Brown2009b}
F.~C.~S. Brown.
\newblock Multiple zeta values and periods of moduli spaces
  {$\overline{\mathfrak M}_{0,n}$}.
\newblock {\em Ann. Sci. \'{E}c. Norm. Sup\'{e}r. (4)}, 42(3):371--489, 2009.
\newblock \url{arXiv:math/0606419v1} [math.AG].

\bibitem{BruinierEhlenYang2019}
J.~H. Bruinier, S.~Ehlen, and T.~Yang.
\newblock C{M} values of higher automorphic {G}reen functions for orthogonal
  groups.
\newblock {\em Invent. Math.}, 225(3):693--785, 2021.
\newblock \url{arXiv:1912.12084} [math.NT].

\bibitem{BruinierLiYang2022}
J.~H. Bruinier, Y.~Li, and T.~Yang.
\newblock Deformations of theta integrals and a conjecture of
  {G}ross--{Z}agier.
\newblock \url{arXiv:2204.10604v2} [math.NT], 2022.

\bibitem{Brychkov2010Pnumu}
Y.~A. Brychkov.
\newblock On the derivatives of the {L}egendre functions {$P^\mu_\nu(z)$} and
  {$Q^\mu_\nu(z)$} with respect to {$\mu$} and {$\nu$}.
\newblock {\em Integral Transforms Spec. Funct.}, 21(3--4):175--181, 2010.

\bibitem{Campbell2023CCG}
J.~M. Campbell.
\newblock Applications of {C}aputo operators in the evaluation of
  {C}lebsch--{G}ordan-type multiple elliptic integrals.
\newblock {\em Integral Transforms Spec. Funct.}, 34(5):371--383, 2023.

\bibitem{Campbell2023mma}
J.~M. Campbell.
\newblock New {C}lebsch--{G}ordan-type integrals involving threefold products
  of complete elliptic integrals.
\newblock {\em Math. Methods Appl. Sci.}, 46(17):17774--17787, 2023.
\newblock \url{arXiv:2302.05819v1} [math.CA].

\bibitem{Campbell2023Sun}
J.~M. Campbell.
\newblock On a conjecture on a series of convergence rate $\frac12$.
\newblock \url{arXiv:2304.00360} [math.NT], 2023.

\bibitem{Campbell2024lemn}
J.~M. Campbell.
\newblock On a higher-order version of a formula due to {R}amanujan.
\newblock {\em Integral Transforms Spec. Funct.}, 35(4):260--269, 2024.

\bibitem{CampbellCantariniAurizio2022}
J.~M. Campbell, M.~Cantarini, and J.~D'Aurizio.
\newblock Symbolic computations via {F}ourier--{L}egendre expansions and
  fractional operators.
\newblock {\em Integral Transforms Spec. Funct.}, 33(2):157--175, 2022.

\bibitem{Cantarini2022}
M.~Cantarini.
\newblock A note on {C}lebsch--{G}ordan integral, {F}ourier--{L}egendre
  expansions and closed form for hypergeometric series.
\newblock {\em Ramanujan J.}, 59(2):549--557, 2022.

\bibitem{CantariniAurizio2019}
M.~Cantarini and J.~D'Aurizio.
\newblock On the interplay between hypergeometric series, {F}ourier--{L}egendre
  expansions and {E}uler sums.
\newblock {\em Boll. Unione Mat. Ital.}, 12(4):623--656, 2019.
\newblock \url{arXiv:1806.08411v1} [math.NT].

\bibitem{Catalan1865}
E.~Catalan.
\newblock \selectlanguage{french}{M}\'emoire sur la transformation des s\'eries
  et sur quelques int\'egrales d\'efinies\selectlanguage{english}.
\newblock In {\em \selectlanguage{french}{M}\'emoires couronn\'es et m\'emoires
  des savants \'etrangers\selectlanguage{english}}, volume~33.
  \selectlanguage{french}l'Acad\'emie royale des {S}ciences, des {L}ettres et
  des {B}eaux-{A}rts de {B}elgique\selectlanguage{english}, Bruxelles, Belgium,
  1865-1867.
\newblock (50 pages).

\bibitem{Catalan1883}
E.~Catalan.
\newblock \selectlanguage{french}{R}echerches sur la constante ${G}$, et sur
  les int\'egrales {E}ul\'eriennes\selectlanguage{english}.
\newblock In {\em \selectlanguage{french}{M}\'emoires de l'{A}cad\'emie
  {I}mp\'eriale des sciences de {S}t.-{P}\'etersbourg\selectlanguage{english}},
  volume~31. \selectlanguage{french}{C}ommissionaires de l'{A}cad\'emie
  {I}mp\'eriale des sciences\selectlanguage{english}, St. Petersburg, Russia,
  1883.
\newblock (51 pages).

\bibitem{ChanWanZudilin2013}
H.~H. Chan, J.~Wan, and W.~Zudilin.
\newblock Legendre polynomials and {R}amanujan-type series for {$1/\pi$}.
\newblock {\em Israel J. Math.}, 194(1):183--207, 2013.

\bibitem{ChowlaSelberg}
S.~Chowla and A.~Selberg.
\newblock On {E}pstein's zeta function ({I}).
\newblock {\em Proc.\ Natl.\ Acad.\ Sci.\ USA}, 35:371--374, 1949.

\bibitem{Chu2022}
W.~Chu.
\newblock Infinite series involving harmonic numbers of convergence rate 1/2.
\newblock {\em Integral Transforms Spec. Funct.}, 33(7):570--590, 2022.

\bibitem{ChuZhang2014}
W.~Chu and W.~Zhang.
\newblock Accelerating {D}ougall's {${}_5F_4$}-sum and infinite series
  involving {$\pi$}.
\newblock {\em Math. Comp.}, 83(285):475--512, 2014.

\bibitem{ChudnovskyChudnovsky1988}
D.~V. Chudnovsky and G.~V. Chudnovsky.
\newblock Approximations and complex multiplication according to {R}amanujan.
\newblock In {\em Ramanujan revisited ({U}rbana-{C}hampaign, {I}ll., 1987)},
  pages 375--472. Academic Press, Boston, MA, 1988.

\bibitem{Clausen1828}
T.~Clausen.
\newblock \begin{otherlanguage*}{german}{U}eber die {F}\"{a}lle, wenn die
  {R}eihe von der {F}orm {$y=1+\frac{\alpha}{1} \cdot
  \frac{\beta}{\gamma}x+\frac{\alpha.\alpha+1}{1.2} \cdot
  \frac{\beta.\beta+1}{\gamma.\gamma+1}x^2 + \text{etc.}$} ein {Q}uadrat von
  der {F}orm {$z=1 + \frac{\alpha'}{1} \cdot \frac{\beta'}{\gamma'} \cdot
  \frac{\delta'}{\varepsilon'}x + \frac{\alpha'.\alpha'+1}{1.2} \cdot
  \frac{\beta'.\beta'+1}{\gamma'.\gamma'+1} \cdot
  \frac{\delta'.\delta'+1}{\varepsilon'.\varepsilon'+1}x^2 + \text{etc.}$}
  hat\end{otherlanguage*}.
\newblock {\em J. Reine Angew. Math.}, 3:89--91, 1828.

\bibitem{CohenGuillera2021}
H.~Cohen and J.~Guillera.
\newblock Rational hypergeometric {R}amanujan identities for $1/\pi^c$: Survey
  and generalizations, 2021.
\newblock \url{arXiv:2101.12592} [math.NT].

\bibitem{Cooper2017Theta}
S.~Cooper.
\newblock {\em Ramanujan's theta functions}.
\newblock Springer, Cham, Switzerland, 2017.

\bibitem{DavydychevKalmykov2001}
A.~I. Davydychev and M.~Y. Kalmykov.
\newblock New results for the {$\varepsilon$}-expansion of certain one-, two-
  and three-loop {F}eynman diagrams.
\newblock {\em Nuclear Phys. B}, 605(1-3):266--318, 2001.

\bibitem{DavydychevKalmykov2004}
A.~I. Davydychev and M.~Y. Kalmykov.
\newblock Massive {F}eynman diagrams and inverse binomial sums.
\newblock {\em Nuclear Phys. B}, 699(1-2):3--64, 2004.

\bibitem{Disegni2022GZp}
D.~Disegni.
\newblock The universal {$p$}-adic {G}ross--{Z}agier formula.
\newblock {\em Invent. Math.}, 230(2):509--649, 2022.

\bibitem{EvansGreene2009}
R.~Evans and J.~Greene.
\newblock Clausen's theorem and hypergeometric functions over finite fields.
\newblock {\em Finite Fields Appl.}, 15(1):97--109, 2009.

\bibitem{FLST2022hypFF}
J.~Fuselier, L.~Long, R.~Ramakrishna, H.~Swisher, and F.-T. Tu.
\newblock Hypergeometric functions over finite fields.
\newblock {\em Mem. Amer. Math. Soc.}, 280(1382):vii+124 pages, 2022.

\bibitem{GlasserZucker1980}
M.~L. Glasser and I.~J. Zucker.
\newblock Lattice sums.
\newblock In {\em Theoretical Chemistry: Advances and Perspectives}, pages
  67--139. Academic Press, New York, NY, 1980.

\bibitem{Goncharov1997}
A.~B. Goncharov.
\newblock The double logarithm and {M}anin's complex for modular curves.
\newblock {\em Math. Res. Lett.}, 4(5):617--636, 1997.

\bibitem{Goncharov1998}
A.~B. Goncharov.
\newblock Multiple polylogarithms, cyclotomy and modular complexes.
\newblock {\em Math. Res. Lett.}, 5(4):497--516, 1998.

\bibitem{GradshteynRyzhik}
I.~S. Gradshteyn and I.~M. Ryzhik.
\newblock {\em Table of Integrals, Series, and Products}.
\newblock Academic Press, Burlington, MA, 7th edition, 2007.
\newblock (Translated from Russian by Scripta Technica, Inc., edited by Alan
  Jeffrey and Daniel Zwillinger).

\bibitem{GreeneStanton1986}
J.~Greene and D.~Stanton.
\newblock A character sum evaluation and {G}aussian hypergeometric series.
\newblock {\em J. Number Theory}, 23(1):136--148, 1986.

\bibitem{GrossZagierII}
B.~Gross, W.~Kohnen, and D.~Zagier.
\newblock Heegner points and derivatives of ${L}$-series. {II}.
\newblock {\em Math.\ Ann.}, 278:497--562, 1987.

\bibitem{GrossZagier1985}
B.~H. Gross and D.~B. Zagier.
\newblock On singular moduli.
\newblock {\em J.\ Reine Angew.\ Math.}, 355:191--220, 1985.

\bibitem{GrossZagierI}
B.~H. Gross and D.~B. Zagier.
\newblock Heegner points and derivatives of ${L}$-series.
\newblock {\em Invent.\ Math.}, 84(2):225--320, 1986.

\bibitem{Guillera2003}
J.~Guillera.
\newblock About a new kind of {R}amanujan-type series.
\newblock {\em Experiment. Math.}, 12(4):507--510, 2003.

\bibitem{Guillera2008ext10}
J.~Guillera.
\newblock Hypergeometric identities for 10 extended {R}amanujan-type series.
\newblock {\em Ramanujan J.}, 15(2):219--234, 2008.

\bibitem{Guillera2010}
J.~Guillera.
\newblock A new {R}amanujan-like series for $1/\pi^2$, 2010.
\newblock \url{arXiv:1003.1915} [math.NT].

\bibitem{Guillera2013div}
J.~Guillera.
\newblock W{Z}-proofs of ``divergent'' {R}amanujan-type series.
\newblock In {\em Advances in combinatorics}, pages 187--195. Springer,
  Heidelberg, 2013.

\bibitem{GuilleraRogers2014}
J.~Guillera and M.~Rogers.
\newblock Ramanujan series upside-down.
\newblock {\em J. Aust. Math. Soc.}, 97(1):78--106, 2014.

\bibitem{Hejhal1983}
D.~A. Hejhal.
\newblock {\em The {S}elberg Trace Formula for
  $\textit{\text{PSL}}(2,\mathbb{R})$ (Volume 2)}, volume 1001 of {\em Lecture
  Notes in Mathematics}.
\newblock Springer-Verlag, Berlin, Germany, 1983.

\bibitem{Pilehroods2011Hurwitz}
K.~Hessami~Pilehrood and T.~Hessami~Pilehrood.
\newblock Bivariate identities for values of the {H}urwitz zeta function and
  supercongruences.
\newblock {\em Electron. J. Combin.}, 18(2):Paper 35, 30 pages, 2011.

\bibitem{HouHeWang2023}
Q.~Hou, H.~He, and X.~Wang.
\newblock Ramanujan-inspired series for $1/\pi$ involving harmonic numbers.
\newblock \url{arXiv:2305.00498v4} [math.NT], 2023.

\bibitem{Ince1956ODE}
E.~L. Ince.
\newblock {\em Ordinary differential equations}.
\newblock Dover Publications, New York, NY, 1956.

\bibitem{IwaniecGSM17}
H.~Iwaniec.
\newblock {\em Topics in Classical Automorphic Forms}, volume~17 of {\em
  Graduate Studies in Mathematics}.
\newblock American Mathematical Society, Providence, RI, 1997.

\bibitem{KalmykovVeretin2000}
M.~Y. Kalmykov and O.~Veretin.
\newblock Single-scale diagrams and multiple binomial sums.
\newblock {\em Phys. Lett. B}, 483(1-3):315--323, 2000.

\bibitem{Kalmykov2007}
M.~Y. Kalmykov, B.~F.~L. Ward, and S.~A. Yost.
\newblock Multiple (inverse) binomial sums of arbitrary weight and depth and
  the all-order $\varepsilon$-expansion of generalized hypergeometric functions
  with one half-integer value of parameter.
\newblock {\em J. High Energy Phys.}, 2007(10):048, 2007.

\bibitem{Koutschan2013}
C.~Koutschan.
\newblock Creative telescoping for holonomic functions.
\newblock In {\em Computer algebra in quantum field theory}, Texts Monogr.
  Symbol. Comput., pages 171--194. Springer, Vienna, Austria, 2013.

\bibitem{Laporta2008}
S.~Laporta.
\newblock Analytical expressions of three- and four-loop sunrise {F}eynman
  integrals and four-dimensional lattice integrals.
\newblock {\em Internat. J. Modern Phys. A}, 23(31):5007--5020, 2008.
\newblock \url{arXiv:0803.1007v4} [hep-ph].

\bibitem{Li2018}
Y.~Li.
\newblock Average {CM}-values of higher {G}reen's function and factorization.
\newblock {\em Amer. J. Math.}, 144(5):1241--1298, 2022.
\newblock \url{arXiv:1812.08523v2} [math.NT].

\bibitem{Li2021}
Y.~Li.
\newblock Algebraicity of higher {G}reen functions at a {CM} point.
\newblock {\em Invent. Math.}, 234(1):375--418, 2023.
\newblock \url{arXiv:2106.13653v2} [math.NT].

\bibitem{McSpiritOno2023}
E.~McSpirit and K.~Ono.
\newblock Hypergeometry and the {AGM} over finite fields.
\newblock In {\em Classical hypergeometric functions and generalizations},
  volume 818 of {\em Contemp. Math.}, pages 197--209. Amer. Math. Soc.,
  Providence, RI, 2025.

\bibitem{MellitThesis}
A.~Mellit.
\newblock {\em Higher {G}reen's Functions for Modular Forms}.
\newblock PhD thesis, Rheinische Friedrich-Wilhelms-Universit\"{a}t Bonn, June
  2008.
\newblock \url{arXiv:0804.3184v1} [math.NT].

\bibitem{Preston1992}
R.~Preston.
\newblock Profiles: The mountains of pi.
\newblock {\em The New Yorker}, pages 36--67, March 2, 1992.

\bibitem{RogersWanZucker2015}
M.~Rogers, J.~G. Wan, and I.~J. Zucker.
\newblock Moments of elliptic integrals and critical ${L}$-values.
\newblock {\em Ramanujan J.}, 37(1):113--130, 2015.
\newblock {\tt arXiv:1303.2259v2} [math.NT].

\bibitem{SelbergChowla}
A.~Selberg and S.~Chowla.
\newblock On {E}pstein's zeta-function.
\newblock {\em J.\ Reine Angew.\ {M}ath.}, 227:86--110, 1967.

\bibitem{Slater}
L.~J. Slater.
\newblock {\em Generalized Hypergeometric Functions}.
\newblock Cambridge University Press, Cambridge, UK, 1966.

\bibitem{Sun2023}
Z.-W. Sun.
\newblock New series involving binomial coefficients (ii).
\newblock \url{arXiv:2307.03086v9} [math.NT], 2023.

\bibitem{Sun2022}
Z.-W. Sun.
\newblock Series with summands involving harmonic numbers.
\newblock In M.~B. Nathanson, editor, {\em Combinatorial and Additive Number
  Theory VI}, volume 464 of {\em Springer Proc. Math. Stat.} Springer, Cham,
  Switzerland, 2025.
\newblock \url{arXiv:2210.07238v9} [math.NT].

\bibitem{SunZhou2024sum3k4k}
Z.-W. Sun and Y.~Zhou.
\newblock Evaluations of $ \sum_{k=1}^\infty \frac{x^k}{k^2\binom{3k}{k}}$ and
  related series.
\newblock \url{arXiv:2401.12083v1} [math.CO], 2024.

\bibitem{TuYang2018}
F.-T. Tu and Y.~Yang.
\newblock Evaluation of certain hypergeometric functions over finite fields.
\newblock {\em SIGMA Symmetry Integrability Geom. Methods Appl.}, 14:Paper No.
  050, 18 pp., 2018.

\bibitem{ViazovskaThesis}
M.~Viazovska.
\newblock {\em Modular Functions and Special Cycles}.
\newblock PhD thesis, Rheinische Friedrich-Wilhelms-Universit\"{a}t Bonn,
  November 2013.

\bibitem{Viazovska2015}
M.~Viazovska.
\newblock C{M} values of higher {G}reen's functions and regularized {P}etersson
  products.
\newblock In {\em Arithmetic and geometry}, volume 420 of {\em London Math.
  Soc. Lecture Note Ser.}, pages 493--503. Cambridge Univ. Press, Cambridge,
  UK, 2015.

\bibitem{Viazovska2012}
M.~Viazovska.
\newblock {P}etersson inner products of weight one modular forms.
\newblock {\em J. Reine Angew. Math.}, 749:133--159, 2019.
\newblock \url{arXiv:1211.4715v2} [math.NT].

\bibitem{Viazovska2011}
M.~Viazovska.
\newblock {CM} values of higher {G}reen's functions.
\newblock In J.~Fres\'{a}n, D.~Jetchev, P.~Jossen, and R.~Pink, editors, {\em
  Motives and Complex Multiplication}, Progress in Mathematics, Boston, MA,
  2023+. Birkh\"{a}user.
\newblock {\tt arXiv:1110.4654v1} [math.NT].

\bibitem{Wan2012}
J.~G. Wan.
\newblock Moments of products of elliptic integrals.
\newblock {\em Adv.\ Appl.\ Math.}, 48:121--141, 2012.
\newblock \url{arXiv:1101.1132v1} [math.CA].

\bibitem{Watson1939}
G.~N. Watson.
\newblock Three triple integrals.
\newblock {\em Quart.\ J.\ Math.}, 10:266--276, 1939.

\bibitem{Wei2023c}
C.~Wei.
\newblock On some conjectural series containing binomial coefficients and
  harmonic numbers.
\newblock \url{arXiv:2306.02641v1} [math.CO], 2023.

\bibitem{Wei2023Sun}
C.~Wei.
\newblock On some conjectures of {Z}.-{W}.\ {S}un involving harmonic numbers.
\newblock \url{arXiv:2304.09753v1} [math.CO], 2023.

\bibitem{Wei2023b}
C.~Wei.
\newblock Some fast convergent series for the mathematical constants $\zeta(4)$
  and $\zeta(5)$.
\newblock \url{arXiv:2303.07887v1} [math.CO], 2023.

\bibitem{WeiXu2023}
C.~Wei and C.~Xu.
\newblock On some conjectural series containing harmonic numbers of 3-order.
\newblock {\em J. Math. Anal. Appl.}, 546(1):Paper No. 129306, 12, 2025.
\newblock \url{arXiv:2308.06440v1} [math.CO].

\bibitem{WeiGuo2022}
T.~Wei and X.~Guo.
\newblock Solvable lattice sums and quadratic {D}irichlet {$L$}-values.
\newblock {\em Acta Arith.}, 203(3):227--238, 2022.

\bibitem{Weinzierl2004bn}
S.~Weinzierl.
\newblock {Expansion around half integer values, binomial sums and inverse
  binomial sums}.
\newblock {\em J. Math. Phys.}, 45:2656--2673, 2004.

\bibitem{WZ1990}
H.~S. Wilf and D.~Zeilberger.
\newblock Rational functions certify combinatorial identities.
\newblock {\em J. Amer. Math. Soc.}, 3(1):147--158, 1990.

\bibitem{Williams1999}
K.~S. Williams.
\newblock Some {L}ambert series expansions of products of theta functions.
\newblock {\em Ramanujan J.}, 3:367--384, 1999.

\bibitem{Zagier1986}
D.~Zagier.
\newblock Hyperbolic manifolds and special values of {D}edekind zeta-functions.
\newblock {\em Invent. Math.}, 83(2):285--301, 1986.

\bibitem{Zagier1998}
D.~Zagier.
\newblock A modular identity arising from mirror symmetry.
\newblock In M.-H. Saito, Y.~Shimizu, and K.~Ueno, editors, {\em Integrable
  Systems and Algebraic Geometry (Proceedings of the Taniguchi Symposium
  1997)}, pages 477--480. World Scientific, Singapore, 1998.

\bibitem{ZagierDilog2007}
D.~Zagier.
\newblock The dilogarithm function.
\newblock In P.~Cartier, B.~Julia, P.~Moussa, and P.~Vanhove, editors, {\em
  Frontiers in Number Theory, Physics, and Geometry II: On Conformal Field
  Theories, Discrete Groups and Renormalization}, pages 3--65. Springer,
  Berlin, Germany, 2007.

\bibitem{Zagier2008Mod123}
D.~Zagier.
\newblock Elliptic modular forms and their applications.
\newblock In K.~Ranestad, editor, {\em The 1-2-3 of Modular Forms, Lectures at
  a Summer School in Nordfjordeid, Norway}, pages 1--103. Springer-Verlag,
  Berlin, Germany, 2008.

\bibitem{Zagier2018ATDE}
D.~Zagier.
\newblock The arithmetic and topology of differential equations.
\newblock In {\em European {C}ongress of {M}athematics}, pages 717--776. Eur.
  Math. Soc., Z\"{u}rich, 2018.

\bibitem{Zeilberger1990a}
D.~Zeilberger.
\newblock A fast algorithm for proving terminating hypergeometric identities.
\newblock {\em Discrete Math.}, 80(2):207--211, 1990.

\bibitem{Zeilberger1991}
D.~Zeilberger.
\newblock The method of creative telescoping.
\newblock {\em J. Symbolic Comput.}, 11(3):195--204, 1991.

\bibitem{Zeilberger1993pun}
D.~Zeilberger.
\newblock Closed form (pun intended!).
\newblock In {\em A tribute to {E}mil {G}rosswald: number theory and related
  analysis}, volume 143 of {\em Contemp. Math.}, pages 579--607. Amer. Math.
  Soc., Providence, RI, 1993.

\bibitem{SWZhang1997}
S.~Zhang.
\newblock Heights of {H}eegner cycles and derivatives of $ {L}$-series.
\newblock {\em Invent.\ Math.}, 130:99--152, 1997.

\bibitem{Zhou2013Pnu}
Y.~Zhou.
\newblock Legendre functions, spherical rotations, and multiple elliptic
  integrals.
\newblock {\em Ramanujan J.}, 34(3):373--428, 2014.
\newblock \url{arXiv:1301.1735v4} [math.CA].

\bibitem{Zhou2013Int3Pnu}
Y.~Zhou.
\newblock On some integrals over the product of three {L}egendre functions.
\newblock {\em Ramanujan J.}, 35:311--326, 2014.
\newblock \url{arXiv:1304.1606v2} [math.CA].

\bibitem{AGF_PartI}
Y.~Zhou.
\newblock {K}ontsevich--{Z}agier integrals for automorphic {G}reen's functions.
  {I}.
\newblock {\em Ramanujan J.}, 38(2):227--329, 2015.
\newblock \url{arXiv:1312.6352v4} [math.CA].

\bibitem{EZF}
Y.~Zhou.
\newblock Ramanujan series for {E}pstein zeta functions.
\newblock {\em Ramanujan J.}, 40(2):367--388, 2016.
\newblock \url{arXiv:1410.8312v2} [math.CA].

\bibitem{AGF_PartII}
Y.~Zhou.
\newblock {K}ontsevich--{Z}agier integrals for automorphic {G}reen's functions.
  {II}.
\newblock {\em Ramanujan J.}, 42(3):623--688, 2017.
\newblock [see {\em Ramanujan J.} 49(1):231--235, 2019 for erratum/addendum]
  \url{arXiv:1506.00318v3} [math.NT].

\bibitem{Zhou2023SunCMZV}
Y.~Zhou.
\newblock Sun's series via cyclotomic multiple zeta values.
\newblock {\em SIGMA Symmetry Integrability Geom. Methods Appl.}, 19:Paper No.
  074, 20 pp., 2023.
\newblock \url{arXiv:2306.04638v3} [math.NT].

\bibitem{Zhou2022mkMpl}
Y.~Zhou.
\newblock Hyper-{M}ahler measures via {G}oncharov--{D}eligne cyclotomy.
\newblock In {\em Classical hypergeometric functions and generalizations},
  volume 818 of {\em Contemp. Math.}, pages 239--288. Amer. Math. Soc.,
  Providence, RI, 2025.
\newblock \url{arXiv:2210.17243v4} [math.NT].

\bibitem{ZuckerRobertson1984}
I.~J. Zucker and M.~M. Robertson.
\newblock Further aspects of the evaluation of {$\sum
  _{(m.n\not=0,0)}(am^{2}+bnm+cn^{2})^{-s}$}.
\newblock {\em Math. Proc. Cambridge Philos. Soc.}, 95(1):5--13, 1984.

\bibitem{Zudilin2004odd_zeta}
W.~Zudilin.
\newblock Arithmetic of linear forms involving odd zeta values.
\newblock {\em J. Th\'eor. Nombres Bordeaux}, 16(1):251--291, 2004.

\bibitem{Zudilin2018hypCY}
W.~Zudilin.
\newblock A hypergeometric version of the modularity of rigid {C}alabi--{Y}au
  manifolds.
\newblock {\em SIGMA Symmetry Integrability Geom. Methods Appl.}, 14:Paper No.
  086, 16 pp., 2018.

\end{thebibliography}
\end{document}